\documentclass[1 [leqno,11pt]{amsart}
\usepackage{amssymb,amsmath,latexsym,amsfonts,amsbsy,amsthm,mathtools,graphicx,color}    
\usepackage{comment}
\usepackage{float}                 
\usepackage{hyperref}       
\usepackage{scalerel,stackengine}  
\usepackage{cancel}    
\usepackage[dvipsnames]{xcolor}     
\stackMath     
\newcommand\reallywidehat[1]{%
\savestack{\tmpbox}{\stretchto{%
  \scaleto{%
    \scalerel*[\widthof{\ensuremath{#1}}]{\kern-.6pt\bigwedge\kern-.6pt}%
    {\rule[-\textheight/2]{1ex}{\textheight}}
  }{\textheight}%
}{0.5ex}}%
\stackon[1pt]{#1}{\tmpbox}%
}
\numberwithin{equation}{section} 
 
\allowdisplaybreaks  

\let\b=\beta

\let\d=\delta
\let\ep=\varepsilon

\let\s=\sigma
\let\f=\frac

\let\om=\omega

\let\Om=\Omega

\let\na=\nabla

\let\pa=\partial

\def\no{\noindent}

\def\eqdef{\buildrel\hbox{\footnotesize def}\over =}

\def\fc{{\mathfrak c}}

\newcommand{\beq}{\begin{equation}}
\newcommand{\eeq}{\end{equation}}
\newcommand{\ben}{\begin{eqnarray}}
\newcommand{\een}{\end{eqnarray}}
\newcommand{\beno}{\begin{eqnarray*}}
\newcommand{\eeno}{\end{eqnarray*}}

\newtheorem{theorem}{Theorem}[section]

\newtheorem{lemma}[theorem]{Lemma}
\newtheorem{proposition}[theorem]{Proposition}
\newtheorem{assumption}[theorem]{Assumption}

\newtheorem{remark}[theorem]{Remark}

\begin{document}
\title[Asymptotic stability in the critical space]{Asymptotic stability in the critical space of 2D monotone shear flow in the viscous fluid}
\author{Hui Li}
\address[H. Li]{Department of Mathematics, New York University Abu Dhabi, Saadiyat Island, P.O. Box 129188, Abu Dhabi, United Arab Emirates.}
\email{lihuiahu@126.com, lihui@nyu.edu}

\author{Weiren Zhao}
\address[W. Zhao]{Department of Mathematics, New York University Abu Dhabi, Saadiyat Island, P.O. Box 129188, Abu Dhabi, United Arab Emirates.}
\email{zjzjzwr@126.com, wz19@nyu.edu}

\begin{abstract}
In this paper, we study the long-time behavior of the solutions to  the two-dimensional incompressible free Navier Stokes equation (without forcing) with small viscosity $\nu$, when the initial data is close to stable monotone shear flows. We prove the asymptotic stability and obtain the sharp stability threshold $\nu^{\frac{1}{2}}$ for perturbations in the critical space $H^{log}_xL^2_y$. Specifically, if the initial velocity $V_{in}$ and the corresponding vorticity $W_{in}$ are $\nu^{\frac{1}{2}}$-close to the  shear flow $(b_{in}(y),0)$ in the critical space, i.e., $\|V_{in}-(b_{in}(y),0)\|_{L_{x,y}^2}+\|W_{in}-(-\pa_yb_{in})\|_{H^{log}_xL^2_y}\leq \ep \nu^{\frac{1}{2}}$, then the velocity $V(t)$ stay $\nu^{\frac{1}{2}}$-close to a shear flow $(b(t,y),0)$ that solves the free heat equation $(\pa_t-\nu\pa_{yy})b(t,y)=0$. We also prove the enhanced dissipation and inviscid damping, namely, the nonzero modes of vorticity and velocity decay in the following sense 
$\|W_{\neq}\|_{L^2}\lesssim \ep\nu^{\frac{1}{2}}e^{-c\nu^{\f13}t}$ and $\|V_{\neq}\|_{L^2_tL^2_{x,y}}\lesssim \ep\nu^{\frac{1}{2}}$. 
In the proof, we construct a time-dependent wave operator corresponding to the Rayleigh operator $b(t,y)\pa_x-\pa_{yy}b(t,y)\pa_x\Delta^{-1}$, which could be useful in future studies. 
\end{abstract}
\maketitle
\section{introduction}
We consider the two-dimensional incompressible Navier-Stokes equation on $\mathbb{T}\times\mathbb{R}$
\begin{equation}\label{eq:NS}
  \left\{
    \begin{array}{l}
    \pa_tV+V\cdot\na V+\na P-\nu\Delta V=0,\\
    \na\cdot V=0,\\
    V(0,x,y)=V_{in}(x,y),
    \end{array}
  \right.
\end{equation}
where $\nu>0$ denotes the viscosity, $V=(V^{(1)},V^{(2)})$ is the velocity field and $P$ is the pressure. Let $W=\pa_xV_2-\pa_y V_1$ be the vorticity, which satisfies
\begin{align}\label{eq:NS1}
  \pa_tW+V\cdot\na W-\nu \Delta W=0.
\end{align}
Let $b(t,y)$ solves the linear heat equation:
\begin{align}\label{eq:linearHeat}
 (\pa_t-\nu\pa_{yy})b(t,y)=0,\quad b(0,y)=b_{in}(y) 
\end{align}
then $V_{sh}=(b(t,y),0)$ is a special solution to \eqref{eq:NS} with  vorticity $W_{sh}=-\pa_yb(t,y)$. 

The purpose of this paper is to study the long-time behavior of solutions to \eqref{eq:NS} when the initial data $V_{in}(x,y)$ is close to the shear flow $(b_{in}(y),0)$. It is natural to introduce the perturbation equations. Let $V=(b(t,y),0)+u$, $W=-\pa_yb(t,y)+\om$, we get
\begin{equation}\label{eq:NS2}
  \left\{
    \begin{array}{l}
      \pa_t\om+b(t,y)\pa_x\om-\pa_{yy}b(t,y)u^{(2)}-\nu\Delta\om+u\cdot\na \om=0,\\
      u=\na^{\bot}\psi=(-\pa_y\psi,\pa_x\psi),\quad \Delta\psi=\om.
    \end{array}
  \right.
\end{equation}

\subsection{Historical comments}
The field of hydrodynamic stability started in the nineteenth century with the pioneering works of Stokes, Helmholtz, Reynolds, Rayleigh, Kelvin, Orr, Sommerfeld, and many others. The study of (in)stability of shear flows dates back to Rayleigh \cite{Rayleigh1880}, Kelvin \cite{Kelvin1887}, and Sommerfeld \cite{Sommerfeld1908}, Orr \cite{Orr1907}. Kelvin considered the linearized Navier-Stokes equation around the Couette flow $(y,0)$:
\begin{equation}\label{eq:LNS2}
  \left\{
    \begin{array}{ll}
      \partial_{t} \omega+y \partial_{x} \omega-\nu \Delta \omega=0,\\
      \psi=\Delta^{-1}\omega.
    \end{array}
  \right.
\end{equation}
Let $\hat{\om}(t,k,\eta)$ denotes the Fourier transform of $\om(t,x,y)$, then the solution of \eqref{eq:LNS2} can be written as
\begin{equation}\label{eq: Lin-sol}
  \begin{aligned}    
    &\hat{\om}(t,k,\eta)=\hat{\om}_{in}(k,\eta+kt)\exp\left(-\nu\int_0^t|k|^2+|\eta-ks+kt|^2ds\right),\\
    &\hat{\psi}(t,k,\eta)=\f{-\hat{\om}_{in}(k,\eta+kt)}{k^2+\eta^2}\exp\left(-\nu\int_0^t|k|^2+|\eta-ks+kt|^2ds\right),
  \end{aligned}
\end{equation}
which gives that
\begin{align}
  &\|\pa_yP_{\neq}\psi\|_{L^2}+\langle t\rangle \|\pa_xP_{\neq}\psi\|_{L^2}\leq C\langle t\rangle^{-1}\|P_{\neq}\om_{in}\|_{H^{2}},\label{eq: ID}\\
&\|P_{\neq}\om\|_{L^2}\leq C\|P_{\neq}\om_{in}\|_{L^2}e^{-c\nu t^3},\label{eq: ED}
\end{align}
where here we denote by $P_{\neq}f=f(x,y)-\frac{1}{2\pi}\int_{\mathbb{T}}f(x,y)dx$ the projection to the nonzero mode of $f$. The inequality \eqref{eq: ID} is the inviscid damping and \eqref{eq: ED} one is the enhanced dissipation. 

For the inviscid flow, Orr observed that the velocity will tend to a shear flow if the initial velocity is close to the Couette flow. This phenomenon is called inviscid damping. The mechanism leading to the inviscid damping is the vorticity mixing driven by the shear flow. The linear inviscid damping is rigorously proved by Case \cite{Case1960}. In the breakthrough paper \cite{BM2015},  Bedrossian-Masmoudi proved the nonlinear inviscid damping of the Couette flow when the perturbations are in the Gevrey-$m$ class ($1\leq m<2$). See also \cite{IonescuJia2020cmp} for the finite channel setting. For the general monotonic shear flow, Case \cite{Case1960} predicted that the inviscid damping still holds with some spectrum assumption on the corresponding Rayleigh operator. The presence of a nonlocal term makes this problem more challenging for general monotonic flow than that for Couette flow, even at the linear level \cite{LinZeng2011,Zillinger2017jde}.  It was later proved by Wei-Zhang-Zhao \cite{WeiZhangZhao2018} at the linear level. We also refer to \cite{Jia2020siam,Jia2020arma} for a simplified proof and the linear inviscid damping in Gevrey class.  Recent works by Ionescu-Jia \cite{IJ2020} and Masmoudi-Zhao \cite{MasmoudiZhao2020} have established the nonlinear inviscid damping for general linear stable monotone shear flows. 

For the viscous flow, the stability phenomenon described by \eqref{eq: ED}, where the decay rate of the vorticity is significantly higher than the diffusive decay rate of $e^{-\nu t}$, is known as enhanced dissipation. The mechanism leading to enhanced dissipation is also due to vorticity mixing. Generally speaking, the sheared velocity sends information to a higher frequency, enhancing the effect of diffusion. However, in experiments, the Couette flow is unstable at high Reynolds numbers, contradicting the linear analysis \cite{romanov1973}. This discrepancy is known as the Sommerfeld paradox \cite{Sommerfeld1908}.  It was suggested by Lord Kelvin \cite{Kelvin1887} that indeed the flow may be stable, but the stability threshold is decreasing as $\nu\to 0$, resulting in transition at a finite Reynolds number in any real system. In \cite{BGM2017}, Bedrossian-Germain-Masmoudi formulated the following  stability threshold problem:

{\it Given a norm $\|\cdot\|_X$, find a $\beta=\beta(X)$ so that
\begin{align*}
  &\|\omega_{in}\|_{X}\leq \nu^{\beta} \Rightarrow \text{stability},\\
&\|\omega_{in}\|_{X}\gg \nu^{\beta} \Rightarrow \text{instability}.
\end{align*}
}
Many works have been devoted to the transition threshold problem of Couette flow. We summarize the most recent stability threshold results for the 2D Couette flow as follows:
\begin{itemize}
    \item \cite{MasmoudiZhao2020cpde}: $\|\omega_{in}\|_{H_x^{log}L_y^2}\leq \ep \nu^{\f12}\Rightarrow \text{asymptotic stability}$;
    \item \cite{LiMasmoudiZhao2022critical}: $\|\omega_{in}\|_{H_x^{1}L_y^2}\geq M\nu^{\f12-\d}\Rightarrow \text{instability}$, for $M$ independent of $\nu$ and any $\d>0$;
    \item \cite{MasmoudiZhao2019}: $\|\omega_{in}\|_{H^{\s}}\leq \ep \nu^{\f13}\Rightarrow \text{asymptotic stability}$, with $\s\geq 40$;
    \item \cite{LMZ2022G}: 
    $\|\omega_{in}\|_{\mathcal{G}^s}\leq \ep \nu^{\b}\Rightarrow \text{asymptotic stability}$, with $1\leq s\leq \f{2-3\b}{1-3\b}$ and $\b\in[0,\f13]$. 
\end{itemize}
We also refer to \cite{BH2020, BMV2016, BVW2018, CLWZ2020} and references therein for many other interesting results and refer to \cite{BGM2020,BGM2015,ChenWeiZhang2020,WeiZhang2020} for the results in the 3D case. It is worth pointing out that when the perturbations are in the critical space $H_x^{log}L_y^2$, $\b=\f12$ is sharp.

For Couette flow, we can observe a strong connection between inviscid damping and enhanced dissipation, both mechanisms of stability arising from the vorticity mixing effect (see \eqref{eq:LNS2}). Furthermore, inviscid damping still holds in the viscous flow, and it remains uniform in $\nu$ (e.g.\cite{CLWZ2020, BH2020, LMZ2022G}). It is natural to ask whether there is still a strong connection between inviscid damping and enhanced dissipation for general monotonic shear flows. If the background monotonic shear flow is stable in inviscid flow, whether it remains stable in the presence of viscosity, and whether the inviscid damping holds uniformly in $\nu$ \cite{BV2022}. Compared to the Couette flow, the problem for general flows becomes more challenging for the following two reasons: 1, the background shear flow varies in time; 2, the nonlocal term destroys the transport-diffusion structure. By adding external forces $F=(-\nu \frac{d^2}{dy^2}b_{in},0)$ in \eqref{eq:NS} to keep the background flow unchanged, at the linear level, Chen-Wei-Zhang \cite{CWZ2022} and Jia \cite{Jiahao2022} provided positive answers for the stability of the background shear flow in the finite channel $\mathbb{T}\times[0,1]$ and unbounded channel $\mathbb{T}\times\mathbb{R}$, respectively, where the corresponding Rayleigh operator has purely continuous spectrum, see also \cite{GNRS2020} for the background flows in the $\mathcal{K}_+$ class. In this paper, we establish the nonlinear enhanced dissipation and uniform inviscid damping for a class of shear flows without the external forcing term, and obtain the sharp stability threshold $\nu^{\frac{1}{2}}$ for perturbations in the critical space $H^{log}_xL^2_y$. To tackle the problem generated from the non-local term, we construct a time-dependent wave operator corresponding to the Rayleigh operator $b(t,y)\pa_x-\pa_{yy}b(t,y)\pa_x\Delta^{-1}$, which could be useful in future studies.

\subsection{Main results}

Let us now list some assumptions on the shear flow $(b_{in}(y),0)$.
\begin{assumption}\label{assum}
We normalize that $b_{in}(0)=0$, and assume that the shear flow $(b_{in}(y),0)$ satisfies the following conditions:
  \begin{itemize}
\item (Monotone) There is $1>c_m>0$ such that for all $y\in\mathbb{R}$, it holds that $$0<c_m\leq \frac{d}{dy}b_{in}(y)\leq c_m^{-1}.$$
\item (Linear growth) There is $c_m\leq C_-\leq C_+\leq c_m^{-1}$ such that 
$$
\lim_{y\to \pm\infty}\frac{d}{dy}b_{in}(y)=C_{\pm}.
$$
\item (Regularity and decay)  $\frac{d^2}{dy^2}b_{in}(y)\in W^{1,1}\cap H^3$ and $(1+y^2)^{\f14}\frac{d^2}{dy^2}b_{in}(y)\in H^1$.  
\item (Stability) For any $t\geq 0$, the Rayleigh operator $\mathcal{R}_{b_t}=b(t,y)\mathrm{Id}-\pa_{yy}b(t,y)(\Delta)^{-1}$ has no embedded eigenvalues or eigenvalues.
\end{itemize}
\end{assumption}

\begin{theorem}\label{thm-main}
Let $\om$ be a solution of \eqref{eq:NS2} with $b(t,y)$ satisfies Assumption \ref{assum}. There exists $\ep_0,\nu_0>0$, such that $\left\|\ln \left(e+\left|D_x\right|\right) \omega_{i n}\right\|_{L_{x, y}^2}+\left\|u_{i n}\right\|_{L_{x, y}^2} \leq \varepsilon_0 \nu^\beta$ for $\beta\ge \frac{1}{2}$, then
\begin{align*}
  \left\|\ln \left(e+\left|D_x\right|\right) \omega_{\neq}(t)\right\|_{L_{x, y}^2}\le C  e^{-c\nu^{\frac{1}{3}}t}\left\|\ln \left(e+\left|D_x\right|\right) \omega_{i n}\right\|_{L_{x, y}^2}.
\end{align*}
Moreover, we have the inviscid damping type estimate,
\begin{align*}
    \int_{0}^{+\infty}\| u^{(2)}_{\neq}(t,x,y)\|_{L^\infty_{x,y}}^2dt &+\int_{0}^{+\infty}\||D_x|^{\frac{1}{2}}u^{(2)}_{\neq}(t,x,y)\|_{L^2_xL^\infty_{y}}^2dt\\
    &+\int_{0}^{+\infty}\|\pa_xu^{(1)}_{\neq}(t,x,y)\|_{L^2_{x,y}}^2dt
    \le C \left\|\ln \left(e+\left|D_x\right|\right) \omega_{i n}\right\|_{L_{x, y}^2}.
\end{align*}
The constants $c,C$ are independent of $\nu$.
\end{theorem}

 \begin{remark}
     In the proof, we construct a wave operator to absorb the nonlocal term and recover the transport-diffusion structure. Let us mention some recent applications of the wave operator method in solving fluid problems: In \cite{LiWeiZhang2020}, the authors use the wave operator method to solve Gallay's conjecture on pseudospectral and spectral bounds of the Oseen vortices operator. In \cite{WeiZhangZhao2020}, the wave operator method was used to solve Beck and Wayne's conjecture. In \cite{MasmoudiZhao2020}, the authors use the wave operator method to prove the nonlinear inviscid damping for stable monotone shear flows. In \cite{Zhao2023inviscid}, the author constructs the wave operator and proves the asymptotic stability for stable monotone shear flows with non-constant background density for the two-dimensional ideal inhomogeneous fluids.
 \end{remark}

\begin{remark}
    We believe the wave operator constructed in this paper can be useful in other problems, especially in the study of the asymptotic stability of monotone shear flows. The authors are optimistic about extending the existing asymptotic stability results for Couette flow in many other fluid equations (such as the stratified fluids, inhomogeneous fluids, and plasma fluids) to general monotone shear flows under some other assumptions. 
\end{remark}

\begin{remark}
Uniqueness for the heat equation fails on unbounded domains. 
The solution $b(t,y)$ is unique under the growth assumption that for all $t>0$, $|b(t,y)|\leq Ae^{ay^2}$, where $A,a$ are arbitrary constants. Once the uniqueness holds, then by the fact that $c_{m}\leq \frac{d}{dy}b_{in}\leq c_{m}^{-1}$, we have $c_{m}\leq\pa_yb(t,y)\leq c_{m}^{-1}$ for all $t>0$ and $y\in\mathbb{R}$, thus, $b(t,y)$ has only linear growth as $y\to \pm\infty$. 
\end{remark}

\begin{remark}
  Different from the background shear flow studied in \cite{Jiahao2022}, we do not impose a compact support condition on  $\frac{d^2}{dy^2}b_{in}$. Instead, we assume that $(1+y^2)^{\f14}\frac{d^2}{dy^2}b_{in}(y)\in H^1$ as stated in Assumption \ref{assum}. It is worth noting that this decay condition can be relaxed to $\left(1+|y|^2\right)^{c}$ for any $c>0$, and similar results can still be obtained.
\end{remark}
\begin{remark}
  Background shear flows satisfying Assumption \ref{assum} do exist. For example, if $\frac{d^2}{dy^2}b_{in}(y)>0$ for $y\in \mathbb{R}$, then the Rayleigh operator $\mathcal{R}_{{b_t}}=b(t,y)\mathrm{Id}-\pa_{yy}b(t,y)(\Delta)^{-1}$ has no eigenvalues for all $t\geq 0$. 
\end{remark}
\begin{remark}
  As $b(t,y)$ satisfies a linear heat equation, for $t$ big enough $\pa_{yy}b(t,y)$ will be sufficiently small, and $\mathcal{R}_{b_t}$ will have no eigenvalues. It means that we only need to check the stability part of Assumption \ref{assum} for finite time regions. In our proof, the stability assumption for $t>0$ seems necessary. Whether the absence of eigenvalues at the initial moment ensures the absence of eigenvalues at any subsequent moment is an interesting question, and we will proceed to investigate this phenomenon.
\end{remark}

\subsection{Main challenges and ideas}
To study the general monotone shear flows, the main difficulty arises from the nonlocal linear term $\pa_{yy}b\pa_x\Delta^{-1}\om$ which breaks the transport diffusion structure and can cause instability. The stability phenomenon of the Couette flow is a result of this structure. To recover this structure for monotone shear flows, we introduce the wave operator $\mathbb{D}_t$ to absorb the nonlocal term, given by,
\begin{equation}
    \mathbb{D}_t\big[b(t,y)\partial_x\om-\partial_{yy}b(t,y)\partial_x\Delta^{-1}\om\big]=b(t,y)\partial_x\mathbb{D}_t[\om].
\end{equation}

\noindent{\bf Difficulties in constructing the wave operator:} To construct this operator, we need to carefully study the linearized Euler equation  $(\pa_t+b\pa_x-\pa_{yy}b\pa_x\Delta^{-1})\om=0$, and derive the representation formula of the solution of this equation. In the case of a finite channel, the representation formula is deduced by using the Dunford integral (the contour integral of the resolvent of the Rayleigh operator), see \cite{MasmoudiZhao2020, WeiZhangZhao2018, Zhao2023inviscid}. However, for an unbounded domain, the spectrum of the Rayleigh operator is also unbounded, and the validity of the Dunford integral may not hold. Instead, we employ the inverse Fourier-Laplace transform to obtain the representation formula which together with the new limiting absorption lemma leads to the wave operator. In the construction, we also carefully study the solution to the homogeneous problem and obtain some uniform upper and lower bounds. Furthermore, there are additional difficulties in constructing and estimating the wave operator due to the time-varying background flow, see Remark \ref{rmk-time-back}. For a more detailed discussion about the construction of the wave operator, please refer to Section \ref{sec-wave}.

\noindent{\bf Difficulties in applying the wave operator:} After applying the wave operator and changing of coordinate $(t,x,y)\to (t,z,v)$ with $v=b(t,y)$ and $z=x-tv$, we obtain a new linearized equation in the form:
\begin{align*}
\pa_tF-\nu \left(\pa_{zz}F+ \left(\pa_yb(t, b^{-1}(t,v))\right)^2(\pa_v-t\pa_z)^2F\right)+\mathrm{Com}(t,z,v)=0,
\end{align*}
where $F(t,z,v)=(\mathbb{D}_t[\omega])(t,x,y)$. 
The presence of the non-constant coefficient $\left(\pa_yb(t, b^{-1}(t,v))\right)^2$ in the diffusion term prevents us from obtaining the solution formula \eqref{eq: Lin-sol} as in the case of Couette flow even at the linear level. To overcome this difficulty, we introduce a time-dependent Fourier multiplier of the ghost type denoted by $\mathring{\mathrm{A}}$ which was previously used in the proof of enhanced dissipation and inviscid damping for the nonlinear system. To give the estimate of the commutator 
$$\big[\mathring{\mathrm{A}},\left(\pa_yb(t, b^{-1}(t,v))\right)^2(\pa_v-t\pa_z)^2\big]F,$$ 
we directly provide the kernel the multiplier $\mathring{\mathrm{A}}$ on the physical side (see Lemma \ref{lem-A-nonlocal}), and obtain the energy dissipation estimate. This is addressed in Lemma \ref{lem-A-nonlocal}. Moreover, by using the Fourier multiplier $\mathring{\mathrm{A}}$, we also obtain the inviscid damping type estimate for the new known $F$, namely, 
\begin{align*}
    \|\pa_z(-\pa_{zz}-(\pa_v-t\pa_z)^2)^{-\f12}\mathring{\mathrm{A}}F\|_{L^2_tL^2_{z,v}}\lesssim \|F_{in}\|_{L^2_{z,v}}. 
\end{align*}

The last difficulty is to translate the enhanced dissipation and inviscid damping obtained for the new unknown $F$ to the original vorticity $\om$ and velocity $V$. In \cite{MasmoudiZhao2020,Zhao2023inviscid}, the authors chose to study the Fourier kernel of the wave operator to get the commutator estimates for the wave operator and the time-dependent Fourier multiplier. Here we take advantage of the fact that $\mathring{\mathrm{A}}$ is bounded from above and below and introduce an alternative approach.  We use the duality argument to recover the inviscid damping, which involves the $H^{-1}$ norm of $F$, namely, 
\begin{align*}
    \|\pa_xV\|_{L^2}^2&=\langle \pa_x\om,\pa_x\psi\rangle
    =\langle \pa_x\mathbb{D}_t[\om],\pa_x\mathbb{D}^1_t[\psi]\rangle\\
    &=\langle \pa_x(-\Delta)^{-\frac{1}{2}}\mathbb{D}_t[\om], \pa_x(-\Delta)^{\frac{1}{2}}\mathbb{D}^1_t[\psi]\rangle \\
    &\leq \|\pa_z(-\Delta_{L})^{-\frac{1}{2}}F\|_{L^2}\|\pa_xV\|_{L^2}.
\end{align*}
Let us now outline the general structure of the paper. In section \ref{sec-wave}, we construct the wave operator by studying the corresponding linearized Euler equation. In section \ref{sec-lin}, we introduce the ghost-type Fourier multiplier and complete the energy estimate for the linearized system. In section \ref{sec-nonlin}, we use a bootstrap argument and prove the nonlinear result. 

\no{\bf Notations}: Through this paper, we will use the following notations. 

We use $C$ to denote a positive big enough constant that may be different from line to line.  In situations where $C$ depends on a specific variable $M$, we use the notation $C(M)$ to emphasize this dependence. However, if $C$ only relies on properties of $b_{in}$ (such as $c_m$ and $\left\|b_{in}''\right\|_{H^3}$), we usually do not explicitly mention the dependence on $b_{in}$.

We use $f\lesssim g$ ($f\approx g$) to denote
\begin{align*}
  f\le C g\quad (C^{-1}g\le f\le C g).
\end{align*}

We use $t$ and $s$ to represent time variables, and $x$, $y$, $z$, and $v$ to represent spatial variables. Given a function $f(t,y)$, we denote its spatial derivation with respect to $y$ as
\begin{align*}
  f'(t,y)=\pa_yf(t,y),\qquad f''(t,y)=\pa_{yy}f(t,y),
\end{align*}
and its temporal derivation with respect to $t$ as
$
  \dot f(t,y)=\pa_tf(t,y).
$
In cases where there is no ambiguity, we use the following abbreviated notation
\begin{align*}
  f_t=f_t(y)=f(t,y),\quad \dot f_t=\dot f_t(y)=\dot f(t,y),\quad f_t'=f_t'(y)=f'(t,y),\quad f_t''=f_t''(y)=f''(t,y).
\end{align*}

We denote the Fourier transform in $x$ of $f(x,y)$ by
\begin{align*}
  \mathcal F_x f(x,y)=\tilde f_k(y)=\tilde f(k,y)=\int_{\mathbb T} f(x,y)e^{- ik x} dx,
\end{align*}
and its Fourier transform in $(x,y)$ by
\begin{align*}
  \mathcal F_{x,y} f(x,y)=\hat f(k,\xi)=\int_{\mathbb T}\int_{\mathbb R} f(x,y)e^{-ikx}e^{- i\xi y} dy dx,
\end{align*}
and the corresponding inverse Fourier transform by
\begin{align*}
  \left(\mathcal F^{-1}_k\tilde f(k,y)\right)(x,y)&=\frac{1}{2\pi}\sum_{k\in\mathbb Z}\tilde f(k,y)e^{ikx},\\
  \left(\mathcal F^{-1}_{k,\xi}\hat f(k,\xi)\right)(x,y)&=\frac{1}{4\pi^2}\sum_{k\in\mathbb Z} \int_{\mathbb R}\hat f(k,\xi) e^{i\xi y} d \xi e^{ikx}.
\end{align*}

We denote the projection to the $k$th mode of $f(x,y)$ by
\begin{align*}
  P_kf(x,y)=f_k(x,y)=\frac{1}{2\pi}\left(\int_{\mathbb{T}}f(x',y)e^{-ikx'}dx'\right)e^{ikx},
\end{align*}
and denote the projection to the non-zero mode by
\begin{align*}
  P_{\neq}f(x,y)=f_{\neq}(x,y)=\sum_{k\in\mathbb Z\setminus \{0\}}f_k(x,y).
\end{align*}
We also use $\hat f_k(\xi)$ to denote $\hat f(k,\xi)$ to emphasize it is the Fourier transform of the $k$ mode.

We use $\left\|\cdot\right\|_{H^{log}_x}$ to denote the following Sobolev norm $\left\|f\right\|_{H^{log}_x}=\left\|\ln(e+|D_x|)f\right\|_{L^{2}_x}$.
\section{Wave operator}\label{sec-wave}
In this section, we construct the wave operator and study its properties. The section will be useful for future studies. Here, we utilize techniques from \cite{LiMasmoudiZhao2022critical, MasmoudiZhao2020, WeiZhangZhao2018, Zhao2023inviscid}, adopting the notation from reference \cite{LiMasmoudiZhao2022critical} for consistency. For the sake of completeness, we provide proof for most of the conclusions.

We first fix $s$, regard $\left(b_s(y),0\right)$ as a steady background shear flow, and introduce the following linearized Euler equation around $\left(b_s(y),0\right)$:
\begin{equation}\label{eq:LinEuler}
  \left\{
    \begin{array}{l}
      \pa_t\om+b(s,y)\pa_x\om-b''(s,y)u^{(2)}=0,\\
u=\na^{\bot}\psi=(-\pa_y\psi,\pa_x\psi),\quad \Delta\psi=\om.
    \end{array}
  \right.
\end{equation}
 Recalling that $b_s$ satisfies the heat equation \eqref{eq:linearHeat}, we have $\dot{b}_s=\nu b''_s$ for all $s>0$ and $y\in \mathbb{R}$. 

In Proposition \ref{pro-wave}, we define the wave operator $\mathbb{D}_{s,k}$ for each wave number $k$. For any function $\om(s,x,y)\in C([0,T);L^2_{x,y})$, we define 
\begin{align}\label{eq: waveoperatorinxy}
  \mathbb{D}_s[\om](s,x,y)=\frac{1}{2\pi}\tilde{\om}_0(s,y)+\frac{1}{2\pi}\sum_{k}\mathbb{D}_{s,k}[\tilde{\om}_k](s,y) e^{ikx}
\end{align}

\subsection{Rayleigh equation and the representation formula}
By taking the Fourier transform in $x$ of \eqref{eq:LinEuler}, we have 
\begin{align}\label{eq:LinEulom}
\Big(\pa_t+ik \mathcal{R}_{s,k}\Big)\tilde{\om}(t,k,y)=0,\quad \tilde{\om}(0,k,y)=\tilde{\om}_{in,k}(y),
\end{align}
where $\mathcal{R}_{s,k}=b_s(y) \mathrm{Id}-b''_s(y)\Delta_k^{-1}$ is the Rayleigh operator and $\Delta_k=\pa_{yy}-k^2$. Let $\mathcal{L}_{s,k}=\Delta_k^{-1}(b_s(y) \Delta_k-b''_s(y)\mathrm{Id})$ we have 
\begin{align}\label{eq:LinEul2}
  (\pa_t+ik\mathcal{L}_{s,k})\tilde{\psi}(t,k,y)=0,\quad \tilde{\psi}(0,k,y)=\tilde{\psi}_{in,k}(y)=\Delta_k^{-1}\tilde{\om}_{in,k}(y).
\end{align}
Under our assumption, the Rayleigh operator $\mathcal{R}_{s,k}$ has no embedded eigenvalue or eigenvalue for all $s\ge0$ and $k\in\mathbb{Z}\setminus\{0\}$, so is $\mathcal{L}_{s,k}$. By a compactness argument, one can easily check that the continuous spectrum of $\mathcal{L}_{s,k}$ cannot be distributed on $\mathbb{C}\setminus\mathbb{R}$, see Remark \ref{rmk-spec}. Thus, by the stability assumption in Assumption \eqref{assum}, we have $\s(\mathcal{L}_{s,k})=\mathbb{R}=\mathrm{Ran}\, b_s$. For $\fc=\fc_r+i\fc_i$, $\fc_i\neq0$, let
\begin{align}\label{eq-tPhi}
  \mathring\Phi_s(k,y,\fc)=(\fc-\mathcal{L}_{s,k})^{-1}(\tilde{\psi}_{in,k}(y)),\qquad\mathring\Phi^{\pm}_s(k,y,\fc_r)=\lim_{\fc_i\to 0\pm}\mathring\Phi_s(k,y,\fc_r+i\fc_i).
\end{align}
We have the following representation formula:
\begin{equation}\label{eq:Rep}
    \tilde{\psi}(t,k,y)=\lim_{T\to \infty}\frac{1}{2\pi i}\int_{-T}^{T}e^{-i \fc_r k t}\big[\mathring\Phi^{-}_s(k,y,\fc_r)-\mathring\Phi^{+}_s(k,y,\fc_r)\big]d\fc_r.
\end{equation}
The identity \eqref{eq:Rep} can be formally regarded as the extension of the Dunford integral for unbounded spectral operators. We give the rigorous proof of \eqref{eq:Rep} in the proof of Proposition \ref{prop-repr}. Also note that in the construction of the wave operator, we only need \eqref{eq:Rep} at $t=0$. 

Let $\Phi^{\pm}_s(k,y,y',\ep)=\mathring\Phi_s(k,y,b_s(y')\pm i \ep)=(b_s(y')\pm i\ep-\mathcal{L}_{s,k})^{-1}(\tilde{\psi}_{in,k}(y))$ with $\ep>0$, then $\Phi^{\pm}_s$ solves the inhomogeneous Rayleigh equation
\begin{align}\label{eq:inhomRay}
  (b_s(y)-b_s(y')\mp i\ep)(\pa_{yy}-k^2)\Phi^{\pm}_s-b_s''(y)\Phi^{\pm}_s=-\Delta_k\tilde{\psi}_{in,k}(y)=-\tilde{\om}_{in,k}(y).
\end{align}
We also introduce the homogeneous Rayleigh equation for $\ep\ge0$
\begin{equation}\label{eq:homRay}
  \begin{aligned}    
   \left\{\begin{aligned}
&(b_s(y)-b_s(y')\mp i\ep)(\pa_{yy}-k^2)\phi^{\pm,\ep}_s(y)-b_s''(y)\phi^{\pm,\ep}_s(y)=0,\\
&\frac{\phi^{\pm,\ep}_s(y)}{b_s(y)-b_s(y')\mp i\ep}\Big|_{y=y'}=1,\quad \frac{d}{dy}\Big(\frac{\phi^{\pm,\ep}_s(y)}{b_s(y)-b_s(y')\mp i\ep}\Big)\Big|_{y=y'}=0.
\end{aligned}\right. 
  \end{aligned}
\end{equation}

\subsection{Solving the homogeneous Rayleigh equation}
In this section, we solve \eqref{eq:homRay} and prove the following proposition:
\begin{proposition}[Existence]\label{prop-phi}
For any $k\neq 0$, there exit $0<\varepsilon_{0,k}\le 1$, $0<C\le1$, such that for $y'\in\mathbb R$ and $0\le\ep\le \varepsilon_{0,k}$, \eqref{eq:homRay} has a regular solution
  \begin{align*}
   \phi^{\pm}_{s}(k,y,y',\ep)=\big(b_s(y)-b_s(y')\mp i\ep\big)\phi_{s,1}(k,y,y')\phi^{\pm}_{s,2}(k,y,y',\ep),
  \end{align*}
  which satisfies $\phi^{\pm}_{s}(k,y,y',\ep)|_{y=y'}=\mp i\ep$ and $\pa_y\phi^{\pm}_{s}(k,y,y',\ep)|_{y=y'}=b_s'(y')$. Here $\phi_{s,1}(k,y,y')$ is a real function that solves
  \begin{equation}\label{eq-phi1}
    \left\{
      \begin{array}{ll}
        \pa_y\left((b_s(y)-b_s(y'))^2\pa_y\phi_{s,1}(k,y,y')\right)-k^2(b_s(y)-b_s(y'))^2\phi_{s,1}(k,y,y')=0,\\
       \phi_{s,1}(k,y,y')|_{y=y'}=1,\quad\pa_y \phi_{s,1}(k,y,y')|_{y=y'}=0,
      \end{array}
    \right.
  \end{equation}
and $\phi^{\pm}_{s,2}(k,y,y',\ep)$ solves
\begin{equation}\label{eq-phi2}
  \left\{
    \begin{array}{ll}
      &\pa_y\Big((b_s(y)-b_s(y')\mp i\ep)^2\phi_{s,1}^2(k,y,y')\pa_y\phi^{\pm}_{s,2}(k,y,y',\ep)\Big)\\
      &\quad \pm\frac{2i{\ep} b_s'(y)(b_s(y)-b_s(y')\mp i\ep)}{(b_s(y)-b_s(y'))}\phi_{s,1}(k,y,y')\pa_y\phi_{s,1}(k,y,y')\phi^{\pm}_{s,2}(k,y,y',\ep)=0,\\
      &\phi^{\pm}_{s,2}(k,y,y',\ep)|_{y=y'}=1,\quad\pa_y\phi^{\pm}_{s,2}(k,y,y',\ep)|_{y=y'}=0.
    \end{array}
  \right.
\end{equation}
It holds that
\begin{align}
      &\phi_{s,1}(k,y,y')\ge1,\quad (y-y')\pa_y\phi_{s,1}(k,y,y')\geq 0,\label{eq-phi1-est-1}\\
   &C^{-1}e^{C^{-1}|k|(|y-y'|)}\le\phi_{s,1}(k,y,y')\le Ce^{C|k||y-y'|},\label{eq-phi1-est-2}\\
    &|\pa_y\phi_{s,1}(k,y,y')|\leq C|k|\min\{|k|(y-y'),1\}\phi_{s,1}(k,y,y'),\label{eq-phi1-est-3}\\
    &|\pa_{yy}\phi_{s,1}(k,y,y')|\leq Ck^2\phi_{s,1}(k,y,y'),\\
    & 0\leq \phi_{s,1}(k,y,y')-1\le C\min \{1, |k|^2|y-y'|^2\}\phi_{s,1}(k,y,y'),\label{eq-phi1-est-4}\\
    &\frac{\phi_{s,1}(k,y,y')}{\phi_{s,1}(k,y_1,y')}\le Ce^{-C|k|(y-y_1)}\text{ for }y_1\le y\le y',\label{eq-phi1-est-5}\\
     &\frac{\phi_{s,1}(k,y,y')}{\phi_{s,1}(k,y_1,y')}\le Ce^{-C|k|(y_1-y)}\text{ for }y_1\ge y\ge y',  \label{eq-phi1-est-6}
\end{align}
  \begin{align}
    &|\phi^{\pm}_{s,2}(k,y,y',\ep)-1|\le C\min \{|k|\varepsilon,|k|^2\varepsilon|y-y'|, |k|^2|y-y'|^2\},\label{eq-phi2-est-1}\\
      &|\pa_y\phi^{\pm}_{s,2}(k,y,y',\ep)|\le C|k|^2\min \{\ep, |y-y'|\},\label{eq-phi2-est-2}\\
      & \|\pa_{yy}\phi^{\pm}_{s,2}(k,y,y',\ep)\|_{L^\infty_y}\le C|k|^2,\label{eq-phi2-est-3}
  \end{align}
  and
    \begin{align}\label{eq-phi-est}
    C^{-1}(|y-y'|+|\ep|)e^{C_4|k||y-y'|}\le|\phi^{\pm}_{s}(k,y,y',\ep)|\le C(|y-y'|+|\ep|)e^{|k||y-y'|},
  \end{align}
  where $C>0$ is a constant independent of $\ep$, $k$, $y'$ and $s$.
\end{proposition}
\begin{proof} 
A direct calculation gives that if $\phi_{s,1}(k,y,y')$ solves \eqref{eq-phi1} and $\phi^{\pm}_{s,2}(k,y,y',\ep)$ solves \eqref{eq-phi2}, then $\phi^{\pm}_{s}(k,y,y',\ep)$ solves \eqref{eq:homRay}. 

\no{\bf Existence of $\phi_{s,1}(k,y,y')$:}
From \eqref{eq-phi1}, it is easy to check that $\phi_{s,1}(k,y,y')$ solves the following equation:
\begin{align*}
\phi_{s,1}(k,y,y')=1+k^2T_1[\phi_{s,1}],
\end{align*}
where 
\begin{align*}
T_1[\phi_{s,1}](k,y,y')=\int_{y'}^y\frac{1}{(b_s(y_1)-b_s(y'))^2}\int_{y'}^{y_1}(b_s(y_2)-b_s(y'))^2\phi_{s,1}(k,y_2,y')dy_2dy_1.
\end{align*}

For a function $g(y,y')$ defined on $\mathbb R\times \mathbb R$, we define
\begin{align*}
  \|g\|_{X}\eqdef \sup_{(y,y')\in\mathbb R\times \mathbb R}\left|\frac{g(y,y')}{\cosh(A(y-y'))}\right|,\quad\|g\|_{Y}\eqdef \sum_{j=0}^2A^{-j} \|\pa_y^kg\|_{X}.
\end{align*}

Let
\begin{align*}
  T_0 g(y,y')=&\int^y_{y'} g(y_1,y') dy_1\\
  T_{l,j}g(y,y')=&\frac{1}{\big(b_s(y)-b_s(y')\big)^j}\int^{y}_{y'}\big(b_s(y_2)-b_s(y')\big)^lg(y_2,y')dy_2.
\end{align*}
Then we have $T_1g(y,y')=T_0\circ T_{2,2}g(y,y')$.

It is clear that
\begin{align*}
  \|T_0 g\|_{X}=&\sup_{(y,y')\in\mathbb R\times \mathbb R}\left|\frac{1}{\cosh(A(y-y'))}\int^y_{y'}\frac{g(y_1,y')}{\cosh(A(y_1-y'))}\cosh(A(y_1-y'))  dy_1\right|\\
  \le&\sup_{(y,y')\in\mathbb R\times \mathbb R}\left|\frac{1}{\cosh(A(y-y'))}\int^y_{y'}\cosh(A(y_1-y'))  dy_1\right|\|g\|_{X}\le\frac{1}{A}\|g\|_{X},
\end{align*}
and
\begin{align*}
  \|T_{2,2} g\|_{X}=&\sup_{(y,y')\in\mathbb R\times \mathbb R}\left|\frac{\int^y_{y'}\frac{g(y_2,y')\big(b_s(y_2)-b_s(y')\big)^2}{\cosh(A(y_2-y'))}\cosh(A(y_2-y'))  dy_2}{\cosh(A(y-y'))\big(b_s(y)-b_s(y')\big)^2}\right|\\
  \le&\sup_{(y,y')\in\mathbb R\times \mathbb R}\left|\frac{1}{\cosh(A(y-y'))}\int^y_{y'}\cosh(A(y_2-y'))  dy_2\right|\|g\|_{X}\le\frac{1}{A}\|g\|_{X}.
\end{align*}
Here we use the fact that $b_s(y)$ is monotonic and that $\left|b_s(y_2)-b_s(y')\right|\le \left|b_s(y)-b_s(y')\right|$ for $\left|y_2-y'\right|\le \left|y-y'\right|$.

It follows that
\begin{align*}
\|T_1[\phi_{s,1}]\|_{X}\le\frac{1}{A^2}\|\phi_{s,1}\|_{X}.
\end{align*}
By taking $A=2k$, we have that $\mathrm{Id}-T_1$ is invertible in the space $X$, thus there exists 
\begin{align*}
  \phi_{s,1}(y,y')=\left(\mathrm{Id}-T_1\right)^{-1}1,
\end{align*}
with the bound
\begin{align*}
\|\phi_{s,1}\|_X\leq C,\qquad |\phi_{s,1}(k,y,y')|\leq Ce^{C|k||y-y'|}.
\end{align*}
 
By taking $A$ slightly larger that depends on $c_m$, we could also prove that $\phi_{s,1}$ exists in $Y$, thus $\phi_{s,1}$ is two order differential. We refer the readers to Proposition 5.3 of \cite{LiMasmoudiZhao2022critical}.

\no{\bf Properties of $\phi_{s,1}(k,y,y')$:} From the expression
\begin{align*}
\phi_{s,1}(k,y,y')=1+k^2\int_{y'}^y\frac{1}{(b_s(y_1)-b_s(y'))^2}\int_{y'}^{y_1}(b_s(y_2)-b_s(y'))^2\phi_{s,1}(k,y_2,y')dy_2dy_1,
\end{align*} 
we can see that 
\begin{align*}
\phi_{s,1}(k,y,y')\ge1,\quad (y-y')\pa_y\phi_{s,1}(k,y,y')\geq 0.
\end{align*}
Then we have
\begin{align*}
0&\leq \phi_{s,1}(k,y,y')-1=k^2\int_{y'}^y\frac{1}{(b_s(y_1)-b_s(y'))^2}\int_{y'}^{y_1}(b_s(y_2)-b_s(y'))^2\phi_{s,1}(k,y_2,y')dy_2dy_1\\
&\leq k^2\int_{y'}^y\frac{1}{(b_s(y_1)-b_s(y'))^2}\int_{y'}^{y_1}(b_s(y_2)-b_s(y'))^2dy_2dy_1\phi_{s,1}(k,y,y')\\
&\leq Ck^2|y-y'|^2\phi_{s,1}(k,y,y'),
\end{align*}
which gives $0\leq \phi_{s,1}(k,y,y')-1\le C\min \{1, |k|^2|y-y'|^2\}\phi_{s,1}(k,y,y')$. 

Let us introduce $F(y,y')=\frac{\pa_y\phi_{s,1}(k,y,y')}{\phi_{s,1}(k,y,y')}$. We deduce from \eqref{eq-phi1} that 
\begin{align*}
 F'(y,y')+F^2(y,y')+\frac{2b_s'(y)F(y,y')}{b_s(y)-b_s(y')}-k^2=0.
\end{align*}
Recalling that $\phi_{s,1}(k,y',y')=1$ and $\pa_y\phi_{s,1}(k,y',y')=0$, we have $F(y',y')=0$. Then we can see that
\begin{align*}
  \lim_{y\to y'}F'(y,y')=&k^2-\lim_{y\to y'}F^2(y,y')-\lim_{y\to y'}\frac{2b_s'(y)F(y,y')}{b_s(y)-b_s(y')}\\
  =&k^2-\lim_{y\to y'}2b_s'(y)\frac{y-y'}{b_s(y)-b_s(y')}\frac{F(y,y')-F(y',y')}{y-y'}\\
  =&k^2-2\lim_{y\to y'}F'(y,y').
\end{align*}
Therefore, $F'(y',y')=\frac{k^2}{3}>0$. It follows from
\begin{align*}
\pa_y\phi_{s,1}(k,y,y')=\frac{1}{\big(b_s(y)-b_s(y')\big)^2}\int^{y}_{y'}\phi_{s,1}(k,y_2,y')\big(b_s(y_2)-b_s(y')\big)^2dy_2,
\end{align*}
that $\pa_y\phi_{s,1}(k,y,y')>0$ for $y>y'$ and $\pa_y\phi_{s,1}(k,y,y')<0$ for $y<y'$. Therefore $\phi_{s,1}(k,y,y')\ge \phi_{s,1}(k,y_2,y')$ for $y_2\in[y,y']$ or $y_2\in[y',y]$. Then we have $\frac{F(y,y')}{b_s(y)-b_s(y')}\ge0$ for $\forall y\in\mathbb R$, and
\begin{align*}
  |F(y,y')|=&\left|\frac{\pa_y\phi_{s,1}(k,y,y')}{\phi_{s,1}(k,y,y')}\right|\\
  =&\left|\frac{k^2}{\big(b_s(y)-b_s(y')\big)^2}\int^{y}_{y'}\frac{\phi_{s,1}(k,y_2,y')}{\phi_{s,1}(k,y,y')}\big(b_s(y_2)-b_s(y')\big)^2dy_2\right|
  \le k^2|y-y'|.
\end{align*}
Next, we show that $|F(y,y')|\le |k|$. If $F$ attains its maximum (minimum) at $y_0^*$, we have $F'(y_0^*,y')=0$ and 
\begin{align*}
  F^2(y_0^*,y')=k^2-\frac{2b_s'(y_0^*)F(y_0^*,y')}{b_s(y_0^*)-b_s(y')}\le k^2.
\end{align*}
Therefore $|F(y,y')|\le |k|$. Otherwise, $F$ has no maximum or minimum point. If $F(y,y')>|k|$ at $y=y_1^*$, we know that $y_1^*>y'$ as $F(y_1^*,y')>0$. Then we have $F'(y_1^*,y')< 0$. Recall that $F'(y',y')=\frac{1}{3}>0$, so there exits $y_2^*\in(y',y_1^*)$ such that $F'(y_2^*,y')=0$ and $|k|\geq F(y_2^*,y')\ge F(y_1^*,y')>|k|$, which is impossible. If $F(y,y')<-|k|$ at $y_3^*$, we know that $y_3^*<y'$ and $F'(y_3^*,y')<0$, then $F(y,y')$ decreases strictly on  $[y_3^*,y']$, which contradicts $F(y',y')=0$.

As a conclusion, we have $|F(y,y')|\le |k|\min(1,|k||y-y'|)$. It follows that
\begin{align*}
  e^{-|k||y_2-y|}\le \frac{\phi_{s,1}(k,y_2,y')}{\phi_{s,1}(k,y,y')}\le e^{|k||y_2-y|}.
\end{align*}
Next, we show that for $y$ such that $|y-y'|\ge \frac{1}{|k|}$, $|F(y,y')|\ge C|k|>0$. Without loss of generality, we only consider the case $y>y'$. For the case $y\ge y'+\frac{1}{|k|}$, we have
\begin{align*}
  F(y,y')=&\frac{k^2}{\big(b_s(y)-b_s(y')\big)^2}\int^{y}_{y'}\frac{\phi_{s,1}(k,y_2,y')}{\phi_{s,1}(k,y,y')}\big(b_s(y_2)-b_s(y')\big)^2dy_2\\
  \ge&\frac{k^2}{\big(b_s(y)-b_s(y')\big)^2}\int^{y}_{y'}e^{-|k||y-y_2|}\big(b_s(y_2)-b_s(y')\big)^2dy_2\\
  \ge&k^2\frac{\big(b_s(y-\frac{1}{2|k|})-b_s(y')\big)^2}{\big(b_s(y)-b_s(y')\big)^2}\int^{y}_{y-\frac{1}{2|k|}}e^{-|k||y-y_2|}dy_2
  \ge C|k|,
\end{align*}
and for $y'<y<y'+\frac{1}{|k|}$, we have $1\leq \phi_{s,1}(k,y,y')\leq C$, and 
\begin{align*}
  F(y,y')=&\frac{k^2}{\big(b_s(y)-b_s(y')\big)^2}\int^{y}_{y'}\frac{\phi_{s,1}(k,y_2,y')}{\phi_{s,1}(k,y,y')}\big(b_s(y_2)-b_s(y')\big)^2dy_2\\
  \ge&\frac{k^2}{\big(b_s(y)-b_s(y')\big)^2}\int^{y}_{y'}C^{-1}\big(b_s(y_2)-b_s(y')\big)^2dy_2
  \ge k^2|y-y'|.
\end{align*}
Here $C$ depends only on $c_m$ and is independent of $s$. Similarly, we have $F(y,y')\le -C\min(|k|^2,|k||y-y'|)$, which gives \eqref{eq-phi1-est-5} and \eqref{eq-phi1-est-6}. The estimate $|\pa_{yy}\phi_{s,1}(k,y,y')|\leq Ck^2\phi_{s,1}(k,y,y')$ follows direction from the equation \eqref{eq-phi1} and the upper bound of $F$. 

\no{\bf Existence of $\phi^{\pm}_{s,2}(k,y,y',\ep)$:} By \eqref{eq-phi2}, we write
\begin{equation}\label{eq-phi2-exp}
\begin{aligned}
&\phi^{\pm}_{s,2}(k,y,y',\ep)=1+T_2[\phi^{\pm}_{s,2}]\\
\eqdef&1\mp2i{\ep}\int^y_{y'} \frac{1}{\big(b_s(y_1)-b_s(y')\mp i\ep\big)^2\phi_{s,1}^2(k,y_1,y')}\\
&\quad \times \int^{y_1}_{y'}\frac{b_s'(y_2)(b_s(y_2)-b_s(y')\mp i\ep)}{b_s(y_2)-b_s(y')}\phi_{s,1}^2(k,y_2,y')F(y_2,y')\phi^{\pm}_{s,2}(k,y_2,y',\ep)dy_2dy_1.
\end{aligned}
\end{equation}
It holds that
\begin{align*}
  &\left|\int^{y_1}_{y'}\frac{b_s'(y_2)(b_s(y_2)-b_s(y')\mp i\ep)}{b_s(y_2)-b_s(y')}\phi_{s,1}(k,y_2,y')\phi_{s,1}'(k,y_2,y')\phi^{\pm}_{s,2}(k,y_2,y',\ep)dy_2\right|\\
  \le&\left|\int^{y_1}_{y'}b_s'(y_2)\phi_{s,1}(k,y_2,y')\phi_{s,1}'(k,y_2,y')\phi^{\pm}_{s,2}(k,y_2,y',\ep)dy_2\right|\\
  &+{\ep}\left|\int^{y_1}_{y'}\frac{b_s'(y_2)F(y_2,y')}{b_s(y_2)-b_s(y')}\phi_{s,1}^2(k,y_2,y')\phi^{\pm}_{s,2}(k,y_2,y',\ep)dy_2\right|.
\end{align*}
Recalling that $(y-y')\phi_{s,1}'(k,y,y')\ge0$ and $\frac{F(y,y')}{b_s(y)-b_s(y')}\ge0$, we have
\begin{align*}
  &\left|\int^{y_1}_{y'}b_s'(y_2)\phi_{s,1}(k,y_2,y')\phi_{s,1}'(k,y_2,y')\phi^{\pm}_{s,2}(k,y_2,y',\ep)dy_2\right|\\
  \le& C \|\phi^{\pm}_{s,2}\|_{L^\infty}\left|\int^{y_1}_{y'}\phi_{s,1}(k,y_2,y')\phi_{s,1}'(k,y_2,y')dy_2\right|=C\|\phi^{\pm}_{s,2}\|_{L^\infty} \big(\phi_{s,1}^2(k,y_1,y')-1\big),
\end{align*}
and
\begin{align*}
  &\left|\int^{y_1}_{y'}\frac{b_s'(y_2)F(y_2,y')}{b_s(y_2)-b_s(y')}\phi_{s,1}^2(k,y_2,y')\phi^{\pm}_{s,2}(k,y_2,y',\ep)dy_2\right|\\
  \le& C |k|^2\|\phi^{\pm}_{s,2}\|_{L^\infty}\int^{y_1}_{y'}\phi_{s,1}^2(k,y_2,y')dy_2\\
  \le& |k|^2\|\phi^{\pm}_{s,2}\|_{L^\infty}\int^{y'+1/|k|}_{y'}\phi_{s,1}^2(k,y_2,y')dy_2\\
  &+C|k|\|\phi^{\pm}_{s,2}\|_{L^\infty}\left|\int^{y_1}_{y'+1/|k|}\phi_{s,1}(k,y_2,y')\phi_{s,1}'(k,y_2,y')dy_2\right|\\
  \le& C|k|\|\phi^{\pm}_{s,2}\|_{L^\infty} \big(\phi_{s,1}^2(k,y_1,y')+C\big).
\end{align*}
Here we use the fact that $|F(y,y')|\le|k|\min(|k||y-y'|,1)$ and $|F(y,y')|\ge C|k|>0$ for $|y-y'|\ge 1/|k|$ and $\phi_{s,1}^2(k,y,y')\leq C$ for $|y-y'|\leq \frac{1}{|k|}$. 

It follows from the fact $\big|b_s(y)-b_s(y')\mp i\ep\big|^2\geq C^{-1}(\ep^2+|y-y'|^2)$ that
\begin{align*}
  &\int^y_{y'} \frac{C\|\phi^{\pm}_{s,2}\|_{L^\infty} \big(\phi_{s,1}^2(k,y_1,y')-1\big)}{\left|b_s(y_1)-b_s(y')\mp i\ep\right|^2\phi_{s,1}^2(k,y_1,y')}dy_1\\
  \le&\int^{y'+\frac{1}{|k|}}_{y'} \frac{C\|\phi^{\pm}_{s,2}\|_{L^\infty} \big(\phi_{s,1}^2(k,y_1,y')-1\big)}{\left|b_s(y_1)-b_s(y')\mp i\ep\right|^2\phi_{s,1}^2(k,y_1,y')}dy_1+\int^{y}_{y'+\frac{1}{|k|}} \frac{C\|\phi^{\pm}_{s,2}\|_{L^\infty} \big(\phi_{s,1}^2(k,y_1,y')-1\big)}{\left|b_s(y_1)-b_s(y')\mp i\ep\right|^2\phi_{s,1}^2(k,y_1,y')}dy_1\\
  \le&\int^{y'+\frac{1}{|k|}}_{y'}C|k|^2\|\phi^{\pm}_{s,2}\|_{L^\infty}dy_1+\int^{y}_{y'+\frac{1}{|k|}} \frac{C\|\phi^{\pm}_{s,2}\|_{L^\infty}}{\left|y_1-y'\right|^2}dy_1\\
  \le& C|k|\|\phi^{\pm}_{s,2}\|_{L^\infty},
\end{align*}
and
\begin{align*}
  &\int^y_{y'} \frac{C|k|\|\phi^{\pm}_{s,2}\|_{L^\infty} \big(\phi_{s,1}^2(k,y_1,y')+C\big)}{\left|b_s(y_1)-b_s(y')\mp i\ep\right|^2\phi_{s,1}^2(k,y_1,y')}dy_1\le \int^y_{y'} \frac{C|k|\|\phi^{\pm}_{s,2}\|_{L^\infty} }{\left|y_1\right|^2+\varepsilon^2}dy_1  \le C\frac{|k|}{\ep}\|\phi^{\pm}_{s,2}\|_{L^\infty}.
\end{align*}
Combining  the above estimates, we have 
\begin{align*}
\|T_2[\phi^{\pm}_{s,2}]\|_{L^{\infty}}\leq C\ep|k|\|\phi^{\pm}_{s,2}\|_{L^\infty}. 
\end{align*}
Here $C$ depends only on $c_m$ and is independent of $s$. 
Thus by choosing $\varepsilon_{0,k}\ll \frac{1}{|k|}$, for $0\le\ep\le\varepsilon_{0,k}$, we have that $\mathrm{Id}-T_2$ is invertible, and
\begin{align*}
\phi^{\pm}_{s,2}(k,y,y',\ep)=(\mathrm{Id}-T_2)^{-1}[1](k,y,y',\ep).
\end{align*}

It is clear that 
\begin{align*}
  \|\phi^{\pm}_{s,2}(k,y,y',\ep)-1\|_{L^\infty_y}\le C |k| |\varepsilon|.
\end{align*}
By taking derivative of \eqref{eq-phi2-exp}, it holds that
\begin{align*}
  \pa_y\phi^{\pm}_{s,2}(k,y,y',\ep)=
&\mp2i{\ep}\frac{1}{\big(b_s(y)-b_s(y')\mp i\ep\big)^2\phi_{s,1}^2(k,y,y')}\\
& \times \int^{y}_{y'}\frac{b_s'(y_2)(b_s(y_2)-b_s(y')\mp i\ep)}{b_s(y_2)-b_s(y')}\phi_{s,1}^2(k,y_2,y')F(y_2,y')\phi^{\pm}_{s,2}(k,y_2,y',\ep)dy_2.
\end{align*}
By using a similar argument to $\|\phi^{\pm}_{s,2}(k,y,y',\ep)-1\|_{L^\infty_y}$, one can deduce that
\begin{align*}
  \|\pa_y\phi^{\pm}_{s,2}(k,y,y',\ep)\|_{L^\infty_y}\le C|k|^2\varepsilon.
\end{align*}

One can deduce from \eqref{eq-phi2} that
\begin{align*}
  \pa_y^2\phi^{\pm}_{s,2}(k,y,y',\ep)=&-\frac{2b_s'(y)\pa_y\phi^{\pm}_{s,2}(k,y,y',\ep)}{(b_s(y)-b_s(y')\mp i\ep}-\frac{2\pa_y\phi_{s,1}(k,y,y')\pa_y\phi^{\pm}_{s,2}(k,y,y',\ep)}{\phi_{s,1}(k,y,y')}\\
  &\mp\frac{2i{\ep} b_s'(y)\pa_y\phi_{s,1}(k,y,y')\phi^{\pm}_{s,2}(k,y,y',\ep)}{\left((b_s(y)-b_s(y')\mp i\ep\right)\left((b_s(y)-b_s(y')\right)\phi_{s,1}(k,y,y')},
\end{align*}
and consequently, $\|\pa_y^2\phi^{\pm}_{s,2}(k,y,y',\ep)\|_{L^\infty_y}\le C|k|^2$.

This completes the proof of this proposition.
\end{proof}
\subsection{Higher regularity estimates for the homogenous solution}
In this subsection, we focus on the higher regularity estimates of $\phi_{s,1}(k,y,y')$ obtained in Proposition \ref{prop-phi}. We introduce a good derivative $\pa_G=\pa_{y}+\pa_{y'}$. In this paper, since we only focus on a finite number of derivatives, the constants in the estimates may depend on the number of derivatives. The estimates for the infinite number of derivatives cases (Gevrey regularity estimate), will be discussed in the forthcoming paper in this series. 
\begin{proposition}\label{prop-est-G}
Let $\phi_{s,1}(k,y,y')$ be the solution of \eqref{eq-phi1}. It holds that
\begin{align}
  \left|\frac{\pa_G\phi_{s,1}(k,y,y')}{\phi_{s,1}(k,y,y')}\right|\le C\min\{ |y-y'|, k^2|y-y'|^3\},\label{eq-est-G}\\
  \left|\pa_y\Big(\frac{\pa_G\phi_{s,1}(k,y,y')}{\phi_{s,1}(k,y,y')}\Big)\right|\leq C\min\{1,|k|^2|y-y'|^2\},\label{eq-est-Gy}
\end{align}
\begin{align}
  \left|\frac{\pa_G^2\phi_{s,1}(k,y,y')}{\phi_{s,1}(k,y,y')}\right|
\leq C\min\{\max \left\{|y-y'|,|y-y'|^2\right\},|k|^2|y-y'|^3\}. \label{eq-est-GG}
\end{align}
Here the constant $C$ depends on $c_m$ and $\left\|b''_{in}\right\|_{H^3}$, and is independent of $s,\nu, k$. 
\end{proposition}
\begin{proof}
Recalling \eqref{eq-phi1} and considering that $\pa_{y'}\left(\phi_{s,1}(k,y',y')\right)= \left(\pa_G\phi_{s,1}\right)(k,y,y')|_{y=y'}$, we deduce that $\pa_G\phi_{s,1}(k,y,y')$ satisfies the following equation
\begin{equation}\label{eq-phi1-goodderivative}
  \begin{aligned}    
            &\pa_y\left((b_s(y)-b_s(y'))^2\phi_{s,1}^2(k,y,y')\pa_y\Big(\frac{\pa_G\phi_{s,1}(k,y,y')}{\phi_{s,1}(k,y,y')}\Big)\right)\\
        =&-\pa_y\left(\frac{2(b'_s(y)-b_s'(y'))}{b_s(y)-b_s(y')}\right)(b_s(y)-b_s(y'))^2\phi_{s,1}(k,y,y')\pa_y\phi_{s,1}(k,y,y'),\\
      &\qquad \pa_G\phi_{s,1}(k,y,y')|_{y=y'}=0,\quad\pa_y \pa_G\phi_{s,1}(k,y,y')|_{y=y'}=0.
  \end{aligned}
\end{equation}
It follows that
\begin{align*}
&\frac{\pa_G\phi_{s,1}(k,y,y')}{\phi_{s,1}(k,y,y')}\\
=&-2\int_{y'}^y\frac{1}{(b_s(y_1)-b_s(y'))^2\phi_{s,1}^2(k,y_1,y')}\\
&\quad \times\int_{y'}^{y_1}\pa_{y_2}\left(\frac{(b'_s(y_2)-b_s'(y'))}{b_s(y_2)-b_s(y')}\right)(b_s(y_2)-b_s(y'))^2\phi_{s,1}(k,y_2,y')\pa_y\phi_{s,1}(k,y_2,y')dy_2dy_1.
\end{align*}
 
By the monotonicity of $\pa_y(\phi_{s,1}(k,y_2,y')^2-1)$, we have 
\begin{align*}
&\left|\int_{y'}^{y_1}\pa_{y_2}\left(\frac{(b'_s(y_2)-b_s'(y'))}{b_s(y_2)-b_s(y')}\right)(b_s(y_2)-b_s(y'))^2\phi_{s,1}(k,y_2,y')\pa_y\phi_{s,1}(k,y_2,y')dy_2\right|\\
\le&C \left\|b''\right\|_{C^1}(b_s(y_1)-b_s(y'))^2\int_{y'}^{y_1}\pa_y(\phi_{s,1}(k,y_2,y')^2-1)dy_2\\
\le& C(b_s(y_1)-b_s(y'))^2(\phi_{s,1}(k,y_1,y')^2-1).
\end{align*}
By using \eqref{eq-phi1-est-4}, we get
\begin{align*}
  \left|\frac{\pa_G\phi_{s,1}(k,y,y')}{\phi_{s,1}(k,y,y')}\right|\le C\min\{ |y-y'|, k^2|y-y'|^3\},\\
  \left|\pa_y\Big(\frac{\pa_G\phi_{s,1}(k,y,y')}{\phi_{s,1}(k,y,y')}\Big)\right|\leq C\min\{1,|k|^2|y-y'|^2\}.
\end{align*}

Similarly, we have
\begin{equation}\label{eq-phi1-goodderivative-2}
  \begin{aligned}    
            &\pa_y\left((b_s(y)-b_s(y'))^2\phi_{s,1}^2(k,y,y')\pa_y\Big(\frac{\pa_G^2\phi_{s,1}(k,y,y')}{\phi_{s,1}(k,y,y')}\Big)\right)\\
        =&-\pa_y\left(\frac{4(b'_s(y)-b_s'(y'))}{b_s(y)-b_s(y')}\right)(b_s(y)-b_s(y'))^2\phi_{s,1}(k,y,y')\pa_y\pa_G\phi_{s,1}(k,y,y')\\
        &-\pa_y\pa_G\left(\frac{2(b'_s(y)-b_s'(y'))}{b_s(y)-b_s(y')}\right)(b_s(y)-b_s(y'))^2\phi_{s,1}(k,y,y')\pa_y\phi_{s,1}(k,y,y'),\\
      &\qquad \pa_G^2\phi_{s,1}(k,y',y')=0,\quad\pa_y \pa_G^2\phi_{s,1}(k,y',y')=0,
  \end{aligned}
\end{equation}
and 
\begin{align*}
&\frac{\pa_G^2\phi_{s,1}(k,y,y')}{\phi_{s,1}(k,y,y')}\\
=&-4\int_{y'}^y\frac{1}{(b_s(y_1)-b_s(y'))^2\phi_{s,1}^2(k,y_1,y')}\\
&\quad \times\int_{y'}^{y_1}\pa_y\left(\frac{(b'_s(y_2)-b_s'(y'))}{b_s(y_2)-b_s(y')}\right)(b_s(y_2)-b_s(y'))^2\phi_{s,1}^2(k,y_2,y')\pa_y \frac{\pa_G\phi_{s,1}(k,y_2,y')}{\phi_{s,1}(k,y_2,y')}dy_2dy_1\\
&-2\int_{y'}^y\frac{1}{(b_s(y_1)-b_s(y'))^2\phi_{s,1}^2(k,y_1,y')}\\
&\quad \times\int_{y'}^{y_1}\pa_y\left(\frac{(b'_s(y_2)-b_s'(y'))}{b_s(y_2)-b_s(y')}\right)(b_s(y_2)-b_s(y'))^2\pa_y \left(\phi_{s,1}^2(k,y_2,y')-1\right)\frac{\pa_G\phi_{s,1}(k,y_2,y')}{\phi_{s,1}(k,y_2,y')}dy_2dy_1\\
&-\int_{y'}^y\frac{1}{(b_s(y_1)-b_s(y'))^2\phi_{s,1}^2(k,y_1,y')}\\
&\quad \times\int_{y'}^{y_1}\pa_y\pa_G\left(\frac{(b'_s(y_2)-b_s'(y'))}{b_s(y_2)-b_s(y')}\right)(b_s(y_2)-b_s(y'))^2\pa_y \left(\phi_{s,1}^2(k,y_2,y')-1\right)dy_2dy_1\\
\eqdef&I+II+III.
\end{align*}
Subsequently, by employing \eqref{eq-phi1-est-4}, \eqref{eq-est-Gy} and \eqref{eq-est-G}, we obtain 
\begin{align*}
  \left|I\right|\le C\int_{y'}^y \int_{y'}^{y_1}\min\{1,|k|^2|y_2-y'|^2\}dy_2dy_1\le C\min\{|y-y'|^2,|k|^2|y-y'|^4\},
\end{align*}
and
\begin{align*}
  \left|II\right|\le& C\int_{y'}^y \frac{1}{\phi_{s,1}^2(k,y_1,y')} \int_{y'}^{y_1}\min\{ |y_2-y'|, k^2|y_2-y'|^3\}\pa_y \left(\phi_{s,1}^2(k,y_2,y')-1\right) dy_2dy_1\\
  \le& C\int_{y'}^y \frac{\phi_{s,1}^2(k,y_1,y')-1}{\phi_{s,1}^2(k,y_1,y')} \min\{ |y_1-y'|, k^2|y_1-y'|^3\} dy_1 \\
  &+C\int_{y'}^y \int_{y'}^{y_1} \frac{\phi_{s,1}^2(k,y_2,y')-1}{\phi_{s,1}^2(k,y_1,y')} \min\{ 1, k^2|y_2-y'|^2\} dy_2dy_1\\
  \le&C\int_{y'}^y  \min\{ |y_1-y'|, k^4|y_1-y'|^5\} dy_1+C\int_{y'}^y \int_{y'}^{y_1} \min\{ 1, k^2|y_2-y'|^2\} dy_2dy_1\\
  \le&C\min\{|y-y'|^2,|k|^2|y-y'|^4\}.
\end{align*}
Similar to the estimate of $\frac{\pa_G\phi_{s,1}(k,y,y')}{\phi_{s,1}(k,y,y')}$, it holds that
\begin{align*}
  \left|II\right|\le C \left\|b''\right\|_{C^2}\min\{ |y-y'|, k^2|y-y'|^3\}.
\end{align*}

Combing Lemma \ref{lem-back} the above estimates, we get
\begin{align*}
  \left|\frac{\pa_G^2\phi_{s,1}(k,y,y')}{\phi_{s,1}(k,y,y')}\right|
\leq C\min\{\max \left\{|y-y'|,|y-y'|^2\right\},|k|^2|y-y'|^3\},
\end{align*}
where $C$ is a constant depends only on $c_m$ and $\left\|b''_{in}\right\|_{H^3}$.
\end{proof}
\subsection{Time derivative estimate for the homogenous solution}
In this subsection, we give the estimate for $\pa_s\phi_{s,1}(k,y,y')\eqdef \dot{\phi}_{s,1}(k,y,y')$. 
\begin{proposition}\label{prop-est-s}
Let $\phi_{s,1}(k,y,y')$ be the solution of \eqref{eq-phi1}. It holds that 
\begin{align}
&\left|\frac{\dot{\phi}_{s,1}(k,y,y')}{\phi_{s,1}(k,y,y')}\right|
\leq C\nu\min\{|y-y'|,k^2|y-y'|^3\}. \label{eq-est-s}
\end{align}
Here the constant $C$ depends on $\left\|b''_{in}\right\|_{C^2}$ and is independent of $s,\nu, k$. 
\end{proposition}
\begin{proof}
Similar to \eqref{eq-phi1-goodderivative}, we deduce that $\dot{\phi}_{s,1}(k,y',y')$ satisfies
  \begin{equation}\label{eq-phi1-goodderivative-t}
    \left\{
      \begin{array}{ll}
        &\pa_y\left((b_s(y)-b_s(y'))^2\phi_{s,1}(k,y,y')^2\pa_y\Big(\frac{\dot{\phi}_{s,1}(k,y,y')}{\phi_{s,1}(k,y,y')}\Big)\right)\\
        &=-\pa_y\left(\frac{2(\dot{b}_s(y)-\dot{b}_s(y'))}{b_s(y)-b_s(y')}\right)(b_s(y)-b_s(y'))^2\phi_{s,1}(k,y,y')\pa_y\phi_{s,1}(k,y,y'),\\
      & \dot{\phi}_{s,1}(k,y,y')|_{y=y'}=0,\quad\pa_y \dot{\phi}_{s,1}(k,y,y')|_{y=y'}=0,
      \end{array}
    \right.
  \end{equation}
with expression
\begin{align*}
&\frac{\dot\phi_{s,1}(k,y,y')}{\phi_{s,1}(k,y,y')}\\
=&-\nu\int_{y'}^y\frac{1}{(b_s(y_1)-b_s(y'))^2\phi_{s,1}(k,y_1,y')^2}\\
&\quad \times\int_{y'}^{y_1}\pa_y\left(\frac{(b''_s(y_2)-b_s''(y'))}{b_s(y_2)-b_s(y')}\right)(b_s(y_2)-b_s(y'))^2\pa_y(\phi_{s,1}(k,y_2,y')^2-1)dy_2dy_1.
\end{align*}
Then one can obtain \eqref{eq-est-s} using a similar approach as Proposition \ref{prop-est-G}. 
\end{proof}

\subsection{Solving the inhomogeneous Rayleigh equation}\label{sec-in-Ray}
In this subsection, we solve the inhomogeneous Rayleigh equation \eqref{eq:inhomRay}, study the limit of \eqref{eq:Rep}, and give a representation formula of $\tilde\psi(t,k,y)$. 

Based on the solution $\phi^{\pm}_{s}(k,y,y',\ep)$ obtained in Proposition \ref{prop-phi}, we introduce the Wronskian of the homogeneous Rayleigh equation \eqref{eq:homRay}:
\begin{align*}
  \mathcal D_s^{\pm}(k,y',\ep)=&\int^{+\infty}_{-\infty}\frac{1}{{\phi_s^{\pm}}^2(k,y_1,y',\ep)}dy_1.
  \end{align*}
The following lemmas hold:
\begin{lemma}\label{lem-eigen}
A number $b_s(y')\pm i\ep\in\mathbb C$ with $0<\ep\le \ep_{0,k}$ is an eigenvalue of $\mathcal R_{s,k}$ if and only if $\mathcal D_s^{\pm}(k,y',\ep)=0$. 
\end{lemma}

\begin{lemma}\label{lem-lim-eigen}
  It holds that
  \begin{align*}
    \lim_{\ep\to0+}\mathcal D_s^{\pm}(k,y',\ep)=\mathcal J_{s,1}(k,y')\mp i\mathcal J_{s,2}(y'),
  \end{align*}
  where
  \begin{align*}
    \mathcal J_{s,1}(k,y')&=\frac{1}{b_s'(y')}\Pi_{s,1}(y')+\Pi_{s,2}(k,y')\\
    &=-\mathcal H\Big((b_{s}^{-1})''\Big)(b_s(y'))+\Pi_{s,2}(k,y')\\
    \mathcal J_{s,2}(y')&=\pi\frac{b_s''(y')}{{b_s'}(y')^3}=-\pi(b_{s}^{-1})''(b_s(y')).
  \end{align*}
Here ${\mathcal H}$ is the Hilbert transform, and
  \begin{align*}
  \Pi_{s,1}(y')= \text{P.V.}\int  \frac{b_s'(y')-b_s'(y)}{(b_s(y)-b_s(y'))^2} dy,\quad \Pi_{s,2}(k,y')= \int \frac{1}{(b_s(y)-b_s(y'))^2}\left(\frac{1}{\phi_{s,1}^2(k,y,y')}-1\right) dy.
\end{align*}
\end{lemma}
\begin{lemma}\label{lem-eigen-0}
A number $b_s(y') \in\mathbb R$ is an eigenvalue of $\mathcal R_{s,k}$ if and only if 
 \begin{align*}
  \mathcal J_{s,1}^2(y')+\mathcal J_{s,2}^2(y')=0.
\end{align*}
\end{lemma}
For the proof of the above three lemmas, we refer the readers to Appendix B of \cite{LiMasmoudiZhao2022critical}.

Next, we show some properties of $\mathcal J_{s,1}$ and $\mathcal J_{s,2}$.
\begin{lemma}\label{lem-bdPi2}
It holds that 
\begin{align}
-C_{**}|k|\leq \Pi_{s,2}(k,y')&\leq -C_*|k|,\label{eq-est-Pi2-1}\\
|\pa_{y'}\Pi_{s,2}(k,y')|&\leq C_*|k|,\label{eq-est-Pi2-2}\\
|\pa_{y'y'}\Pi_{s,2}(k,y')|&\leq C_*|k|,\label{eq-est-Pi2-3}
\end{align}
where $C_*$ and $C_{**}$ are  positive constants that depend on $c_m$ and $\left\|b''_{in}\right\|_{H^3}$, and are independent of $y',k,s,\nu$. 
\end{lemma}
\begin{proof}
We first decompose $\Pi_{s,2}$ into two parts,
\begin{align*}
  \Pi_{s,2}(k,y')=& \int_{|y-y'|\le \frac{1}{|k|}} \frac{1}{(b_s(y)-b_s(y'))^2} \frac{(1-\phi_{s,1} (k,y,y'))(1+\phi_{s,1} (k,y,y'))}{\phi_{s,1}^2(k,y,y')}  dy\\
  &+\int_{|y-y'|> \frac{1}{|k|}} \frac{1}{(b_s(y)-b_s(y'))^2}\frac{(1-\phi_{s,1} (k,y,y'))(1+\phi_{s,1} (k,y,y'))}{\phi_{s,1}^2(k,y,y')}  dy\\
  \eqdef&I+II.
\end{align*}

From \eqref{eq-phi1-est-2} and \eqref{eq-phi1-est-4}, we have $\frac{1}{\phi_{s,1}^2(k,y,y')}-1\le 0$, 
\begin{align*}
  0\ge I\ge -C \int_{|y-y'|\le \frac{1}{|k|}} \frac{|k|^2|y-y'|^2}{|y-y'|^2}  dy\ge-C|k|,
\end{align*}
and
\begin{align*}
  II\approx  -\int_{|y-y'|> \frac{1}{|k|}} \frac{1}{|y-y'|^2}  dy\approx -|k|,
\end{align*}
which gives \eqref{eq-est-Pi2-1}.

Taking derivative of $\Pi_{s,2}$, and using \eqref{eq-phi1-est-2}, \eqref{eq-phi1-est-4}, and \eqref{eq-est-G}, we deduce
\begin{align*}
  \left|\pa_{y'}\Pi_{s,2}(k,y')\right|\le& \int \left|\frac{1}{(b_s(y)-b_s(y'))^2}(\pa_{y'}+\pa_{y})\left(\frac{1}{\phi_{s,1}^2(k,y,y')}-1\right)\right|   dy\\
  &+\int  \left|\left(\frac{1}{\phi_{s,1}^2(k,y,y')}-1\right) (\pa_{y'}+\pa_{y})\frac{1}{(b_s(y)-b_s(y'))^2}\right|dy\\
  \le&C\int \left|\frac{1}{(b_s(y)-b_s(y'))^2} \frac{\pa_G\phi_{s,1} (k,y,y')}{\phi_{s,1}^3(k,y,y')} \right|dy\\
  &+C\int  \left|\left(\frac{1}{\phi_{s,1}^2(k,y,y')}-1\right) \frac{b_s'(y)-b_s'(y')}{(b_s(y)-b_s(y'))^3}\right|dy\\
  \le&C|k|.
\end{align*}

In the same way, by using \eqref{eq-est-GG}, we also have
\begin{align*}
  \left|\pa_{y'y'}\Pi_{s,2}(k,y')\right|\le C |k|.
\end{align*}
Here $C$ is a constant depends only on $c_m$ and $\left\|b''_{in}\right\|_{H^3}$.
\end{proof}

\begin{lemma}\label{lem-J-low}
  It holds that
  \begin{align}\label{est-J-low}
    \left|\mathcal J_{s,1}(y')\right|^2+\left|\mathcal J_{s,2}(y')\right|^2\ge \delta|k|^2.
  \end{align}
  Here $\delta>0$ is a constant that independent of $(y',s,\nu)$.
\end{lemma}

\begin{proof}
We introduce 
  \begin{align*}
    \widetilde \Pi_{\tau,1}(y')=\Pi_{\nu^{-1}\tau,1}(y'),\ \widetilde \Pi_{\tau,2}(y')=\Pi_{\nu^{-1}\tau,2}(y'),\\
    \widetilde{\mathcal J}_{\tau,1}(y')=\mathcal J_{\nu^{-1}\tau,1}(k,y'),\ \widetilde{\mathcal J}_{\tau,2}(k,y')=\mathcal J_{\nu^{-1}\tau,2}(k,y').
  \end{align*}
One can regard the above time-scaled quantities defined on $B_\tau(y)=b_{\nu^{-1}s}(y)$, which is the solution of a $\nu$ independent equation
\begin{align*}
  (\pa_\tau-\pa_{yy})B_{\tau}(y)=0,\quad B_0(y)=b_{in}(y).
\end{align*} 

To prove \eqref{est-J-low} it suffices to prove
\begin{align*}
  \left|\widetilde{\mathcal J}_{\tau,1}(y')\right|^2+\left|\widetilde{\mathcal J}_{\tau,2}(k,y')\right|^2\ge \delta|k|^2.
\end{align*}

We begin by giving the existence of $\tau_0$, which is independent of $y',\nu,k$, such that for $\tau>\tau_0$, we have $\left|\widetilde\Pi_{\tau,1}(y')\right|\le \frac{1}{2}C_*$, where $C_*$ is given in Lemma \ref{lem-bdPi2}.

Recalling
\begin{align*}
B_\tau''(y)=\frac{1}{\sqrt{4\pi \tau}}\int_{\mathbb{R}}e^{-\frac{(y-y')^2}{4\tau}}b''_{in}(y')dy',\quad B_\tau'''(y)=\frac{1}{\sqrt{4\pi \tau}}\int_{\mathbb{R}}e^{-\frac{(y-y')^2}{4\tau}}b'''_{in}(y')dy',
\end{align*} 
we have
\begin{align*}
  \left\|B_\tau''(y)\right\|_{L^2}\lesssim \frac{1}{\tau^{\frac{1}{4}}}\left\|b_{in}''(y)\right\|_{L^1},\quad\left\|B_\tau'''(y)\right\|_{L^2}\lesssim \frac{1}{\tau^{\frac{1}{4}}}\left\|b_{in}'''(y)\right\|_{L^1}.
\end{align*}
By choosing $\tau_0=CC^{-4}_*$ to be sufficiently large, we can ensure that for $\tau>\tau_0$ 
\begin{align*}
  \left|\frac{1}{B_\tau'(y')}\widetilde \Pi_{\tau,1}(y')\right|=&\left|\mathcal H\Big((B_{\tau}^{-1})''\Big)(B_\tau(y'))\right|\le C \left\|B_\tau''(y)\right\|_{H^1}\\
  \le&C\frac{1}{\tau^{\frac{1}{4}}}\left\|b_{in}''(y)\right\|_{W^{1,1}}\le \frac{1}{2}C_*,
\end{align*}
and then
\begin{align*}
  \left|\widetilde{\mathcal J}_{\tau,1}(k,y')\right|\ge \left|\widetilde\Pi_{\tau,2}(k,y')\right|-\left|\frac{1}{B_\tau'(y')}\widetilde\Pi_{\tau,1}(y')\right|\ge \frac{1}{2}C_*|k|,\text{ for } \tau>\tau_0.
\end{align*}

Next, we show that there exists a constant $Y$, independent of $\tau, \nu, k$, such that for $\tau \leq \tau_0$ and $|y'| > Y$, it holds that $\left|\widetilde\Pi_{\tau,1}(y')\right|\le \frac{1}{2}C_*$.
  
As $\sqrt{1+|v'|}$ is a $A_2$ weight, it holds that
\begin{align*}
  &\left\|\sqrt{1+|v'|}\mathcal H[\pa_v^2(B_\tau^{-1})](v')\right\|_{L^2_{v'}}+\left\|\sqrt{1+|v'|}\pa_{v'}\mathcal H[\pa_v^2(B_\tau^{-1})](v')\right\|_{L^2_{v'}}\\
  \lesssim& \left\|\sqrt{1+|B_\tau(y)|}B_\tau''(y)\right\|_{L^2_y}+\left\|\sqrt{1+|B_\tau(y)|}B_\tau'''(y)\right\|_{L^2_y}.
\end{align*}
Recalling that $b_{in}(0)=0$ and $ \left|B''_\tau(0)\right|\le C\left\|b''_{in}(y)\right\|_{H^1}$, we have
\begin{align*}
   \left|B_{\tau}(0)\right|\le \int^{\tau}_0 \left|B''_{\tau'}(0)\right| d \tau'\le C\tau.
\end{align*}
It follows from \eqref{eq-est-b1} that
\begin{align*}
  c_m |y'|-C\tau\le\left|B_{\tau}(y')\right|\le c_m^{-1}|y'|+C\tau.
\end{align*}
Then by \eqref{eq-est-b3} and \eqref{eq-est-b4}, we deduce that
\begin{align*}
  &\sqrt{1+|B_{\tau}(y')|}\left|\mathcal H[\pa_v^2(B_\tau^{-1})](B_{\tau}(y'))\right|\\
  \le& C c_m^{-\frac{1}{2}}(1+\tau)\left(\left\|\sqrt{1+|y|}b_{in}''(y)\right\|_{L^2}+\left\|\sqrt{1+|y|}b_{in}'''(y)\right\|_{L^2}+ \left\|b_{in}''(y)\right\|_{W^{1,1}}\right).
\end{align*}

Therefore, by taking $Y=C\tau_0^2c_m^{-3}C_*^{-2}$ sufficiently large, we have for $\tau\le \tau_0$ and $|y'|\ge Y$ that
\begin{align*}
  \left|\widetilde{\mathcal J}_{\tau,1}(k,y')\right|\ge \left|\widetilde\Pi_{\tau,2}(k,y')\right|-\left|\frac{1}{B_\tau'(y')}\widetilde\Pi_{\tau,1}(y')\right|\ge \frac{1}{2}C_*|k|.
\end{align*}

It is clear that $\widetilde\Pi_{\tau,1}(y')$ has a upper bound uniform in $k$ and $\tau$. As $\widetilde\Pi_{\tau,2}(k,y')\leq -C_*|k|$, there exists $K$ independent on $\tau,y'$ that $\left|\widetilde{\mathcal J}_{\tau,1}(k,y')\right|\ge \frac{1}{2}C_*|k|$ for $|k|\ge K$.

Based on our assumption, the Rayleigh operator $\mathcal R_{s,k}$ has no eigenvalue or embedded eigenvalue. So by Lemma \ref{lem-eigen-0}, it holds that $\left|\widetilde{\mathcal J}_{\tau,1}(k,y')\right|^2+\left|\widetilde{\mathcal J}_{\tau,2}(k,y')\right|^2\neq0$ for all $\tau$, $k$, and $y'$. Note that the set $D=\left\{(y',\tau)\big||y'|\le Y,0\le \tau\le \tau_0\right\}$ is a compact set. There exists $\delta>0$ such that
\begin{align*}
  \inf_{|k|\le K}\inf_{(y',\tau)\in D}\left|\widetilde{\mathcal J}_{\tau,1}(k,y')\right|^2+\left|\widetilde{\mathcal J}_{\tau,2}(k,y')\right|^2\ge\delta.
\end{align*}
As the definition of $B_\tau(y)$ is independent of $\nu$, here $\delta$ is also independent of $\nu$.

Combining all the estimates, we arrive at \eqref{est-J-low}.
\end{proof}

\begin{proposition}\label{prop:inhom}
 For any $k\neq 0$, there exist $0<\varepsilon_{0,k}\le 1$, $0<C\le1$, such that for $y'\in\mathbb R$ and $0<|\ep|\le \varepsilon_{0,k}$, 
\beq\label{eq:inhom-sol}
\begin{aligned}
\Phi_s^{\pm}(k,y,y',\ep)=&\Phi_{s,i,l}^{\pm}(k,y,y',\ep)+\mu_s^{\pm}[\tilde{\om}_{in,k}](y',\ep)\Phi_{s,h,l}^{\pm}(k,y,y',\ep)\\
=&\Phi_{s,i,r}^{\pm}(k,y,y',\ep)+\mu_s^{\pm}[\tilde{\om}_{in,k}](y',\ep)\Phi_{s,h,r}^{\pm}(k,y,y',\ep)
\end{aligned}
\eeq
is a unique solution to \eqref{eq:inhomRay} which decays at infinity. Here
\begin{align*}
\Phi_{s,i,l}^{\pm}(k,y,y',\ep)&=-\phi^{\pm}_{s}(k,y,y',\ep)\int_{-\infty}^y\frac{\int_{y'}^{y_1}\tilde{\om}_{in,k}(y_2)\phi_{s,1}(k,y_2,y')\phi^{\pm}_{s,2}(k,y_2,y',\ep)dy_2}{{\phi^{\pm}_{s}}^2(k,y_1,y',\ep)}dy_1,\\
\Phi_{s,h,l}^{\pm}(k,y,y',\ep)&=\phi^{\pm}_{s}(k,y,y',\ep)\int_{-\infty}^y\frac{1}{{\phi^{\pm}_{s}}^2(k,y_1,y',\ep)}dy_1,\\
\Phi_{s,i,r}^{\pm}(k,y,y',\ep)&=-\phi^{\pm}_{s}(k,y,y',\ep)\int_{+\infty}^y\frac{\int_{y'}^{y_1}\tilde{\om}_{in,k}(y_2)\phi_{s,1}(k,y_2,y')\phi^{\pm}_{s,2}(k,y_2,y',\ep)dy_2}{{\phi^{\pm}_{s}}^2(k,y_1,y',\ep)}dy_1,\\
\Phi_{s,h,r}^{\pm}(k,y,y',\ep)&=\phi^{\pm}_{s}(k,y,y',\ep)\int_{+\infty}^y\frac{1}{{\phi^{\pm}_{s}}^2(k,y_1,y',\ep)}dy_1,\\
\mu_s^\pm[\tilde{\om}_{in,k}](y',\ep)&=\frac{\int_{-\infty}^{+\infty}\frac{\int_{y'}^{y_1}\tilde{\om}_{in,k}(y_2)\phi_{s,1}(k,y_2,y')\phi^{\pm}_{s,2}(k,y_2,y',\ep)dy_2}{{\phi^{\pm}_{s}}^2(k,y_1,y',\ep)}dy_1}{\mathcal D_s^{\pm}(k,y',\ep)}.
\end{align*}
\end{proposition}
\begin{proof}
  By Proposition \ref{prop-phi}, it is easy to see that the integrations in the definition of $\Phi_{s,i,l}^{\pm}(k,y,y',\ep)$, $\Phi_{s,h,l}^{\pm}(k,y,y',\ep)$, $\Phi_{s,i,r}^{\pm}(k,y,y',\ep)$, and $\Phi_{s,h,r}^{\pm}(k,y,y',\ep)$ are well-defined for $0<\ep<\ep_{0,k}$. Recall that $\phi^{\pm}_{s}(k,y,y',\ep)$ is a solution to the homogeneous equation obtained in Proposition \ref{prop-phi}. The solution $\Phi_{s}^{\pm}(k,y,y',\ep)$ to the inhomogeneous equation \eqref{eq:inhomRay} satisfies
\begin{align}\label{eq:inhomRay2}
  \pa_y\left(\phi^{\pm}_{s}(k,y,y',\ep)^2\pa_y\left(\frac{\Phi_{s}^{\pm}(k,y,y',\ep)}{\phi^{\pm}_{s}(k,y,y',\ep)}\right)\right)=-\tilde{\om}_{in,k}(y_2)\phi_{s,1}(k,y,y')\phi^{\pm}_{s,2}(k,y,y',\ep). 
\end{align}
It is easy to check that $\Phi_{s}^{\pm}(k,y,y',\ep)$ given in \eqref{eq:inhom-sol} satisfies \eqref{eq:inhomRay2} and decays at infinity. The uniqueness of the solution follows from the assumption that the Rayleigh operator $\mathcal{R}_{s,k}$ has no eigenvalue.
\end{proof}

\begin{proposition}\label{prop-limit}
Let $\tilde{\om}_{in,k}\in C^{\infty}_c(\mathbb{R})$. For each $s$ and $k$, the following point-wise limit holds:
\begin{equation}\label{eq-Phi-lim}
  \begin{aligned}    
&\lim_{\ep\to 0+}\Phi_s^{\pm}(k,y,y',\ep)=\Phi_s^{\pm}(k,y,y')\\
&\eqdef\left\{\begin{aligned}
&-\phi_s(k,y,y')\int_{-\infty}^y\frac{\int_{y'}^{y_1}\tilde{\om}_{in,k}(y_2)\phi_{s,1}(k,y_2,y')dy_2}{\phi_s^2(k,y_1,y')}dy_1\\
&\quad \quad \quad  +\mu_s^{\pm}[\tilde{\om}_{in,k}](y')\phi_s(k,y,y')\int_{-\infty}^y\frac{1}{\phi_s^2(k,y_1,y')}dy_1, \quad  \text{for}\quad  y<y',\\
&-\phi_s(k,y,y')\int_{+\infty}^y\frac{\int_{y'}^{y_1}\tilde{\om}_{in,k}(y_2)\phi_{s,1}(k,y_2,y')dy_2}{\phi_s^2(k,y_1,y')}dy_1\\
&\quad \quad \quad +\mu_s^{\pm}[\tilde{\om}_{in,k}](y')\phi_s(k,y,y')\int_{+\infty}^y\frac{1}{\phi_s^2(k,y_1,y')}dy_1, \quad  \text{for}\quad  y>y'.
\end{aligned}\right.    
  \end{aligned}
\end{equation}
Here 
\begin{align*}
\phi_s(k,y,y')=(b_s(y)-b_s(y'))\phi_{s,1}(k,y,y'),\\
\mu_s^{\pm}[\tilde{\om}_{in,k}](y')=\frac{\mathcal J_{s,3}(k,y')\pm i\mathcal J_{s,4}(y')}{\mathcal J_{s,1}(k,y')\mp i\mathcal J_{s,2}(y')},
\end{align*}
and
\begin{align*}
  \mathcal J_{s,3}(k,y')= \text{P.V.}\int\frac{\int_{y'}^{y_1}\tilde{\om}_{in,k}(y_2)\phi_{s,1}(k,y_2,y')dy_2}{(b_s(y)-b_s(y'))^2\phi_{s,1}(k,y_1,y')^2}dy_1,\ \mathcal J_{s,4}(y')=\pi \frac{\tilde{\om}_{in,k}(y')}{b_s'(y')^2}.
\end{align*}
It also holds that
\begin{align}\label{eq-lim-phi-con}
  \lim_{y\to y'-}\Phi_s^{\pm}(k,y,y')=\lim_{y\to y'+}\Phi_s^{\pm}(k,y,y'),
\end{align}
which means $\Phi_s^{\pm}$ is continuous with respect to $y$.
\end{proposition}
\begin{proof}
First, we show $\lim\limits_{\ep\to 0+}\mu_s^\pm[\tilde{\om}_{in,k}](y',\ep)=\mu_s^{\pm}[\tilde{\om}_{in,k}](y')$. The limit of $\mathcal D_s^{\pm}(k,y',\ep)$ is given in Lemma \ref{lem-lim-eigen}, we only need to give the limit of the numerator of $\mu_s^\pm[\tilde{\om}_{in,k}]$.  We write
\begin{align*}
  &\int_{-\infty}^{\infty}\frac{\int_{y'}^{y_1}\tilde{\om}_{in,k}(y_2)\phi_{s,1}(k,y_2,y')\phi^{\pm}_{s,2}(k,y_2,y',\ep)dy_2}{{\phi^{\pm}_{s}}^2(k,y_1,y',\ep)}dy_1\\
  =&\int_{-\infty}^{\infty}\frac{\int_{y'}^{y_1}\tilde{\om}_{in,k}(y_2) \left(\phi_{s,1}(k,y_2,y')\phi^{\pm}_{s,2}(k,y_2,y',\ep)-1\right)dy_2}{{\phi^{\pm}_{s}}^2(k,y_1,y',\ep)}dy_1\\ 
  &+\int_{-\infty}^{\infty}\frac{\int_{y'}^{y_1}\tilde{\om}_{in,k}(y_2) dy_2}{\big(b_s(y_1)-b_s(y')\mp i\ep\big)^2} \left(\frac{1}{\phi_{s,1}^2(k,y_1,y'){\phi^{\pm}_{s,2}}^2(k,y_1,y',\ep)}-1\right)dy_1\\
  &+\int_{-\infty}^{\infty}\frac{\int_{y'}^{y_1}\tilde{\om}_{in,k}(y_2) dy_2}{\big(b_s(y_1)-b_s(y')\mp i\ep\big)^2} dy_1\\
  \eqdef&\widetilde I_s^{\pm}(y',\varepsilon)+\widetilde {II}_s^{\pm}(y',\varepsilon)+\widetilde {III}_s^{\pm}(y',\varepsilon). 
\end{align*}

From Proposition \ref{prop-phi}, we have $\left|\phi_{s,1}(k,y_2,y')\phi^{\pm}_{s,2}(k,y_2,y',\ep)-1\right|\le C(k)\min \{1, |y_2-y'|^2\}\phi_{s,1}(k,y_2,y')$, and then
\begin{align*}
  \left|\frac{\int_{y'}^{y_1}\tilde{\om}_{in,k}(y_2) \left(\phi_{s,1}(k,y_2,y')\phi^{\pm}_{s,2}(k,y_2,y',\ep)-1\right)dy_2}{{\phi^{\pm}_{s}}^2(k,y_1,y',\ep)}\right|\le C(k) \frac{\left|y_1-y'\right|^{\frac{1}{2}}\left\|\tilde{\om}_{in,k}\right\|_{L^2}}{1+(y_1-y')^2},
\end{align*}
where $C(k)$ is independent of $y'$ and $\varepsilon$. By the Lebesgue dominated convergence theorem, we get
\begin{align*}
  \lim_{\varepsilon\to0+}\widetilde I_s^{\pm}(y',\varepsilon)=\int_{-\infty}^{\infty}\frac{\int_{y'}^{y_1}\tilde{\om}_{in,k}(y_2) \left(\phi_{s,1}(k,y_2,y') -1\right)dy_2}{{\phi^{\pm}_{s}}^2(k,y_1,y',0)}dy_1,
\end{align*}
and
\begin{align*}
  \lim_{\varepsilon\to0+}\widetilde {II}_s^{\pm}(y',\varepsilon)=\int_{-\infty}^{\infty}\frac{\int_{y'}^{y_1}\tilde{\om}_{in,k}(y_2) dy_2}{\big(b_s(y_1)-b_s(y')\big)^2} \left(\frac{1}{\phi_{s,1}^2(k,y_1,y')}-1\right)dy_1.
\end{align*}

Letting $v=b_s(y_1)$ and $v'=b_s(y')$, we decompose $\widetilde {III}_s^{\pm}$ into its real and imaginary parts
\begin{align*}
  \widetilde {III}_s^{\pm}(b_s^{-1}(v'),\varepsilon)=&\int_{-\infty}^{\infty}\frac{\int_{b_s^{-1}(v')}^{b_s^{-1}(v)}\tilde{\om}_{in,k}(y_2) dy_2}{\big(v-v'\mp i\ep\big)^2} \frac{1}{b'_s(b_s^{-1}(v))}d v\\
  =&\int_{-\infty}^{\infty} \frac{1}{v-v'\mp i\ep} \pa_v \frac{\int_{b_s^{-1}(v')}^{b_s^{-1}(v)}\tilde{\om}_{in,k}(y_2) dy_2}{b'_s(b_s^{-1}(v))} d v\\
  =&\int_{-\infty}^{\infty} \frac{v-v'}{(v-v')^2+\ep^2} \pa_v \frac{\int_{b_s^{-1}(v')}^{b_s^{-1}(v)}\tilde{\om}_{in,k}(y_2) dy_2}{b'_s(b_s^{-1}(v))} d v\\
  &+i\int_{-\infty}^{\infty} \frac{\pm \varepsilon}{(v-v')^2+\ep^2} \pa_v \frac{\int_{b_s^{-1}(v')}^{b_s^{-1}(v)}\tilde{\om}_{in,k}(y_2) dy_2}{b'_s(b_s^{-1}(v))} d v\\
  \eqdef&\widetilde {III}_{s,r}^{\pm}(b_s^{-1}(v'),\varepsilon)+i\widetilde {III}_{s,i}^{\pm}(b_s^{-1}(v'),\varepsilon).
\end{align*}
It is clear that
\begin{align*}
  \lim_{\varepsilon\to0+}\widetilde {III}_{s,r}^{\pm}(y',\varepsilon)=& \text{P.V.}\int_{-\infty}^{\infty} \frac{1}{v-b_s(y')} \pa_v \frac{\int_{y'}^{b_s^{-1}(v)}\tilde{\om}_{in,k}(y_2) dy_2}{b'_s(b_s^{-1}(v))} d v,\\
   \lim_{\varepsilon\to0+}\widetilde {III}_{s,i}^{\pm}(y',\varepsilon)=&\pm \pi\frac{ \tilde{\om}_{in,k}(y')}{ \left(b'_s(y')\right)^2}.
\end{align*}
Accordingly,
\begin{equation}\label{eq-lim-mu-1}
  \begin{aligned}    
  &\lim_{\varepsilon\to0+}\int_{-\infty}^{\infty}\frac{\int_{y'}^{y_1}\tilde{\om}_{in,k}(y_2)\phi_{s,1}(k,y_2,y')\phi^{\pm}_{s,2}(k,y_2,y',\ep)dy_2}{{\phi^{\pm}_{s}}^2(k,y_1,y',\ep)}dy_1\\
  =& \text{P.V.}\int\frac{\int_{y'}^{y_1}\tilde{\om}_{in,k}(y_2)\phi_{s,1}(k,y_2,y')dy_2}{(b_s(y)-b_s(y'))^2\phi_{s,1}(k,y_1,y')^2}dy_1\pm i\pi  \frac{ \tilde{\om}_{in,k}(y')}{ \left(b'_s(y')\right)^2}, 
  \end{aligned}
\end{equation}
and $\lim\limits_{\ep\to 0+}\mu_s^\pm[\tilde{\om}_{in,k}](y',\ep)=\mu_s^{\pm}[\tilde{\om}_{in,k}](y')$.

Next, we show the point-wise limit of $\Phi_{s,i,l}^{\pm}(k,y,y',\ep)$ and $\Phi_{s,h,l}^{\pm}(k,y,y',\ep)$ for $y<y'$, as well as the limits of $\Phi_{s,i,r}^{\pm}(k,y,y',\ep)$ and $\Phi_{s,h,r}^{\pm}(k,y,y',\ep)$ for $y>y'$. Without loss of generality, we only give the proof for $y<y'$. For $y_1\le y<y'$ and $y_1\le y_2\le y'$, using Proposition \ref{prop-phi}, we have
\begin{equation}\label{eq-decay-part1}
  \begin{aligned}    
    &\left|\phi^{\pm}_{s}(k,y,y',\ep)\frac{\int_{y'}^{y_1}\tilde{\om}_{in,k}(y_2)\phi_{s,1}(k,y_2,y')\phi^{\pm}_{s,2}(k,y_2,y',\ep)dy_2}{{\phi^{\pm}_{s}}^2(k,y_1,y',\ep)}dy_1\right|\\
   \le& C(k) \frac{\left\|\tilde{\om}_{in,k}\right\|_{L^\infty}\left(e^{-|k||y'-y_1|}-1\right)e^{-|k||y-y_1|}}{|y_1-y'|},   
  \end{aligned}
\end{equation}
\begin{align}\label{eq-decay-part2}
  \left| \frac{\phi^{\pm}_{s}(k,y,y',\ep)}{{\phi^{\pm}_{s}}^2(k,y_1,y',\ep)}dy_1\right|\le C(k)\frac{ e^{-|k||y'-y_1|} e^{-|k||y-y_1|}}{|y_1-y'|},
\end{align}
where $C(k)$ is independent of $y'$ and $\varepsilon$. Then by the Lebesgue dominated convergence theorem, we get
\begin{equation}\label{eq-lim-Phi-i}
  \begin{aligned}    
      &\lim_{\varepsilon\to0+}\phi^{\pm}_{s}(k,y,y',\ep)\int_{-\infty}^y\frac{\int_{y'}^{y_1}\tilde{\om}_{in,k}(y_2)\phi_{s,1}(k,y_2,y')\phi^{\pm}_{s,2}(k,y_2,y',\ep)dy_2}{{\phi^{\pm}_{s}}^2(k,y_1,y',\ep)}dy_1\\
  =&(b_s(y)-b_s(y'))\phi_{s,1}(k,y,y')\int_{-\infty}^y\frac{\int_{y'}^{y_1}\tilde{\om}_{in,k}(y_2)\phi_{s,1}(k,y_2,y')dy_2}{(b_s(y_1)-b_s(y'))^2\phi_{s,1}(k,y_1,y')^2}dy_1,
  \end{aligned}
\end{equation}
and
\begin{align}\label{eq-lim-Phi-h}
    \lim_{\varepsilon\to0+}\phi^{\pm}_{s}(k,y,y',\ep)\int_{-\infty}^y\frac{1}{{\phi^{\pm}_{s}}^2(k,y_1,y',\ep)}dy_1=\int_{-\infty}^y\frac{(b_s(y)-b_s(y'))\phi_{s,1}(k,y,y')}{(b_s(y_1)-b_s(y'))^2\phi_{s,1}(k,y_1,y')^2}dy_1.
\end{align}

By using the l'Hospital's rule, one can easily check that
\begin{align*}
  \lim_{y\to y'-}\Phi_{s,i,l}^{\pm}(k,y,y',\ep)=\lim_{y\to y'+}\Phi_{s,i,r}^{\pm}(k,y,y',\ep),\ \lim_{y\to y'-}\Phi_{s,h,l}^{\pm}(k,y,y',\ep)=\lim_{y\to y'+}\Phi_{s,h,r}^{\pm}(k,y,y',\ep),
\end{align*}
which gives \eqref{eq-lim-phi-con}.

This finishes the proof of this proposition.
\end{proof}

\begin{proposition}[Representation formula]\label{prop-repr}
Suppose $\tilde{\om}_{in,k}\in C^{\infty}_c(\mathbb{R})$, then it holds for $t\geq 0$ that 
\begin{equation}\label{eq-rep-psi}
  \begin{aligned}    
     \tilde{\psi}(t,k,y) &=-\frac{1}{\pi}\int_{-\infty}^{y} e^{-ikb_s(y') t} \frac{\mathcal J_{s,1} \mathcal J_{s,4} +\mathcal J_{s,2} \mathcal J_{s,3} }{\mathcal J_{s,1}^2+\mathcal J_{s,2}^2}\int_{+\infty}^y\frac{\phi_s(k,y,y')}{\phi^2_s(k,y_1,y')}dy_1b_s'(y')dy'\\
  &\quad-\frac{1}{\pi}\int_{y}^{+\infty} e^{-ikb_s(y') t}\frac{\mathcal J_{s,1} \mathcal J_{s,4} +\mathcal J_{s,2} \mathcal J_{s,3} }{\mathcal J_{s,1}^2+\mathcal J_{s,2}^2}\int_{-\infty}^y\frac{\phi_s(k,y,y')}{\phi^2_s(k,y_1,y')}dy_1b_s'(y')dy'.
  \end{aligned}
\end{equation}
where $\mathcal J_{s,1}$ and $\mathcal J_{s,2}$ are defined in Lemma \ref{lem-lim-eigen}, and $\mathcal J_{s,1}$ and $\mathcal J_{s,2}$ are defined in Proposition \ref{prop:inhom}. 
\end{proposition} 
\begin{proof}
First, we give the rigorous deduction of the identity \eqref{eq:Rep}. Without loss of generality, here we only focus on the case that $k>0$. Recall \eqref{eq:LinEulom} and \eqref{eq:LinEul2}. By the standard energy method, one can show that there exists $c^*>0$ such that
\begin{equation}\label{eq-grow-est1}
   \begin{aligned}    
  \tilde{\om}(t,k,y)\le C e^{c^*kt}, \quad\pa_t\tilde{\om}(t,k,y)\le C e^{c^*kt},\text{ for }t\ge0,y\in\mathbb R,\\
   \tilde{\psi}(t,k,y)\le C e^{c^*kt}, \quad\pa_t \tilde{\psi}(t,k,y)\le C e^{c^*kt},\text{ for }t\ge0,y\in\mathbb R.     
   \end{aligned}
 \end{equation} 
Applying the Fourier-Laplace transform on \eqref{eq:LinEul2} we have
\begin{align*}
  \tilde\fc \tilde{\psi}^*(\tilde\fc,k,y)+i\mathcal L_{s,k}\tilde{\psi}^*(\tilde\fc,k,y)=\frac{1}{k}\tilde{\psi}_{in,k}(y),\text{ for }\tilde\fc_r> c^*,
\end{align*}
where $\tilde\fc=\tilde\fc_r+i\tilde\fc_i$ and
\begin{align*}
  \tilde{\psi}^*(\tilde\fc,k,y)=\int^{+\infty}_0e^{-\tilde\fc kt}\tilde{\psi}(t,k,y)dt.
\end{align*}
It follows that
\begin{align}\label{eq-resolvent1}
  \tilde{\psi}^*(\tilde\fc,k,y)=\left(\tilde\fc+i\mathcal L_{s,k} \right)^{-1}\frac{1}{k}\tilde{\psi}_{in,k}(y),\text{ for }\tilde\fc_r> c^*.
\end{align}

By the inverse Fourier-Laplace transform (which is known by various names, the Bromwich integral, and Mellin's inverse formula), we have (see, \cite{Terras1985, HP1957})
\begin{equation}\label{eq-inver-1}
  \frac{k}{2\pi i}\lim_{T\to+\infty} \int^{c^*+iT}_{c^*-iT} e^{\tilde\fc kt}\tilde{\psi}^*(\tilde\fc,k,y) d\tilde\fc=H(t)\tilde{\psi}(t,k,y),
\end{equation}
where $H(t)=\frac{1}{2}(1+\text{sgn}(t))$ is the Heaviside step function. So, at $t=0$, the inverse Fourier-Laplace transform fails to recover the original function. To address this issue, we introduce $\underline{\tilde{\psi}}(t',k,y)$ which satisfies
\begin{align}\label{eq: LinearEuler-Psi3}
  \pa_{t'}\underline{\tilde{\psi}}(t',k,y)-i k\mathcal L_{s,k}\underline{\tilde{\psi}}(t',k,y)=0,\quad \underline{\tilde{\psi}}(0,k,y)= \tilde{\psi}_{in,k}(y).
\end{align}
One can regard $\underline{\tilde{\psi}}(t',k,y)={\tilde{\psi}}(-t',k,y)$ for $t'\ge0$, and \eqref{eq: LinearEuler-Psi3} as the time-backward extension of \eqref{eq:LinEul2}. It is clear that \eqref{eq-grow-est1} also holds for $\underline{\tilde{\psi}}(t',k,y)$. Similar to $\tilde{\psi}$, we also have 
\begin{align}\label{eq-resolvent2}
  \underline{\tilde{\psi}}^*(\tilde\fc',k,y)=\left(\tilde\fc'-i\mathcal L_{s,k} \right)^{-1}\frac{1}{k}\tilde{\psi}_{in,k}(y),\text{ for }\tilde\fc_r'> c^*.
\end{align}
and
\begin{equation}\label{eq-inver-2}
\frac{k}{2\pi i}\lim_{T\to+\infty} \int^{c^*+iT}_{c^*-iT} e^{\tilde\fc' kt}\underline{\tilde{\psi}}^*(\tilde\fc',k,y) d\tilde\fc'=H(t')\underline{\tilde{\psi}}(t',k,y).
\end{equation}
Let $t=-t'$ and $\tilde \fc=-\tilde\fc'$, we write \eqref{eq-inver-2} as
\begin{equation}\label{eq-inver-3}
  \frac{k}{2\pi i}\lim_{T\to+\infty} \int^{-c^*+iT}_{-c^*-iT} e^{\tilde\fc t}\underline{\tilde{\psi}}^*(-\tilde\fc,k,y) d\tilde\fc=H(-t)\tilde{\psi}(t,k,y).
\end{equation}
Combing \eqref{eq-resolvent1}, \eqref{eq-inver-1}, \eqref{eq-resolvent2}, and \eqref{eq-inver-3}, we have for $t\in\mathbb R$ that
\begin{equation}
  \begin{aligned}    
   \tilde{\psi}(t,k,y)=&\frac{1}{2\pi i}\lim_{T\to+\infty} \int^{c^*+iT}_{c^*-iT} e^{\tilde\fc t}\left(\tilde\fc+i\mathcal L_{s,k} \right)^{-1}\tilde{\psi}_{in,k}(y) d\tilde\fc\\
  &-\frac{1}{2\pi i}\lim_{T\to+\infty} \int^{-c^*+iT}_{-c^*-iT} e^{\tilde\fc t}\left(\tilde\fc+i\mathcal L_{s,k} \right)^{-1}\tilde{\psi}_{in,k}(y) d\tilde\fc   
  \end{aligned}
\end{equation}
Let $\fc=i\tilde \fc$, we have
\begin{equation}\label{eq:Rep2}
  \begin{aligned}    
   \tilde{\psi}(t,k,y)=&\frac{1}{2\pi i}\lim_{T\to+\infty} \int^{T}_{-T} e^{-i \left(\fc_r-ic^*\right) t}\left(\fc_r-ic^*-\mathcal L_{s,k} \right)^{-1}\tilde{\psi}_{in,k}(y) d\fc_r\\
  &-\frac{1}{2\pi i}\lim_{T\to+\infty} \int^{T}_{-T} e^{-i \left(\fc_r+ic^*\right) t}\left(\fc_r+ic^*-\mathcal L_{s,k} \right)^{-1}\tilde{\psi}_{in,k}(y) d\fc_r,
  \end{aligned}
\end{equation}
which is consistent with the form of \eqref{eq:Rep}.

Recall \eqref{eq-tPhi} that $\mathring\Phi_s(k,y,\fc)=(\fc-\mathcal{L}_{s,k})^{-1}(\tilde{\psi}_{in,k}(y))\in H_y^2(\mathbb{R})$ which satisfies for $\fc_i\neq0$ that
\begin{align}\label{eq-tPhi2}
  (\pa_{yy}-k^2)\mathring\Phi_s-\frac{b_s''(y)}{b_s(y)-\fc}\mathring\Phi_s=-\frac{\tilde{\om}_{in,k}(y)}{b_s(y)-\fc}.
\end{align}
Next we will show that for $|\fc_r|$ big enough, we have $\left\|\mathring\Phi_s(k,y,\fc)\right\|_{L^2}\le C \frac{1}{|\fc_r|} \left\|\tilde{\om}_{in,k}(y)\right\|_{L^2}$, where $C$ is independent of $\fc_i\neq0$.

Taking the inner product of \eqref{eq-tPhi2} with $\mathring\Phi_s$, we have
\begin{align}\label{eq-en-tPhi}
  \left\|\mathring\Phi_s'\right\|_{L^2}^2+k^2\left\|\mathring\Phi_s\right\|_{L^2}^2=-\int_{\mathbb R}\frac{b_s''(y)}{b_s(y)-\fc}\left|\mathring\Phi_s(k,y,\fc)\right|^2  dy+ \int_{\mathbb R}\frac{\tilde{\om}_{in,k}(y)}{b_s(y)-\fc} \overline{\mathring\Phi_s}(k,y,\fc)dy.
\end{align}
We write 
\begin{align*}
  \int_{\mathbb R}\frac{b_s''(y)}{b_s(y)-\fc}\left|\mathring\Phi_s\right|^2 d y=\int_{\mathbb R\setminus D}\frac{b_s''(y)}{b_s(y)-\fc}\left|\mathring\Phi_s\right|^2 d y+\int_{D}\frac{b_s''(y)}{b_s(y)-\fc}\left|\mathring\Phi_s\right|^2 d y\eqdef I+II,
\end{align*}
where $D=[b_s^{-1}(\fc_r)-|\fc_r|^{-\frac{1}{4}},b_s^{-1}(\fc_r)+|\fc_r|^{-\frac{1}{4}}]$. Similar to the proof of Lemma \ref{lem-J-low}, by using Lemma \ref{lem-back}, one can easily check that there exists $Y_1$ such that
\begin{align*}
  \sup_{y\in \mathbb R\setminus D} \left|\frac{b_s''(y)}{b_s(y)-\fc}\right|\le C |\fc_r|^{-\frac{1}{4}}\le \frac{1}{4} \text{ for } |\fc_r|\ge Y_1.
\end{align*}
It follows that $|I|\le \frac{1}{4}\|\mathring\Phi_s\|_{L^2_y}^2$ for $|\fc_r|\ge Y_1$. 

For $II$, by integration by parts, we have
\begin{align*}
  II=&\int_{D}\frac{b_s''(y)}{b_s(y)-\fc}|\mathring\Phi_s|^2 d y=\int_{D}\frac{b_s''}{b_s'}|\mathring\Phi_s|^2\pa_y\ln(b_s-\fc)  d y\\
  =&-\int_{D}\frac{b_s'''b_s'-(b_s'')^2}{(b_s')^2}|\mathring\Phi_s|^2\ln(b_s-\fc)+ 2\frac{b_s''}{b_s'}\Re(\mathring\Phi_s\overline{\mathring\Phi_s'})\ln(b_s-\fc) d y+\frac{b_s''}{b_s'}|\psi|^2\ln(b_s-\fc)\Big|^{b_s^{-1}(\fc_r)-|\fc_r|^{-\frac{1}{4}}}_{b_s^{-1}(\fc_r)-|\fc_r|^{-\frac{1}{4}}}.
\end{align*}
Then by Lemma \ref{lem-back}, we have for $|\fc_r|\ge Y_1$ that
\begin{align*}
  |II|\le& C \|\mathring\Phi_s\|_{L^\infty_y}^2 \|\ln(b_s-\fc)\|_{L^1_y(D)}+C \|\mathring\Phi_s\|_{L^\infty_y}\|\mathring\Phi_s'\|_{L^2_y}\|\ln(b_s-\fc)\|_{L^2_y(D)}\\
  &+C|\fc_r|^{-\frac{1}{2}}\|\mathring\Phi_s\|_{L^\infty_y}^2 \left(\left|\ln \left(b_s \left(b_s^{-1}(\fc_r)-|\fc_r|^{-\frac{1}{4}}\right)-\fc\right)\right|+\left|\ln \left(b_s \left(b_s^{-1}(\fc_r)-|\fc_r|^{-\frac{1}{4}}\right)-\fc\right)\right|\right).
\end{align*}
Recall the definition of $D$.  We have for $p=1,2$ that
\begin{align*}
  &\int_{D}|\ln(b_s-\fc)|^p d y\le C\int_{|z|\le c_m^{-1}|\fc_r|^{-\frac{1}{4}}} |\ln(z)|^pdz\le C\int_{-\infty}^{\ln\big(c_m^{-1}|\fc_r|^{-\frac{1}{4}}\big)}|y|^pe^{y}dy\le C |\fc_r|^{-\frac{1}{4}}|\ln(|\fc_r|)|^p,
\end{align*}
and
\begin{align*}
  \left(\left|\ln \left(b_s \left(b_s^{-1}(\fc_r)-|\fc_r|^{-\frac{1}{4}}\right)-\fc\right)\right|+\left|\ln \left(b_s \left(b_s^{-1}(\fc_r)-|\fc_r|^{-\frac{1}{4}}\right)-\fc\right)\right|\right)\le C |\ln(|\fc_r|)|.
\end{align*}
By taking $Y_1$ big enough, we have $|II|\le \frac{1}{4}\|\mathring\Phi_s\|_{H^1_y}^2$.

As $\tilde{\om}_{in,k}$ has compact support, there exists a constant $Y_2$ such that 
\begin{align*}
  \left|\frac{\tilde{\om}_{in,k}(y)}{b_s(y)-\fc}\right|\le \frac{|\tilde{\om}_{in,k}(y)|}{1+|\fc_r|},\text{ for }|\fc_r|\ge Y_2.
\end{align*}
Combing the above estimate, we deduce from \eqref{eq-en-tPhi} that
\begin{align}\label{eq-est-tPhi}
  \|\mathring\Phi_s\|_{L^{\infty}_y}\leq C\|\mathring\Phi_s\|_{H^1_y}\le C \frac{\left\|\tilde{\om}_{in,k}\right\|_{L^2_y}}{1+|\fc_r|}, \text{ for }|\fc_r|\ge\max(Y_1,Y_2),
\end{align}
where $C$ is independent of $\fc_i\neq0$.

By Remark \ref{rmk-spec} and our assumption that $\mathcal R$ has no eigenvalues, $\mathbb C\setminus\mathbb R$ is in the resolvent set of $\mathcal L_{s,k}$, namely $\mathbb C\setminus\mathbb R\subset \mathbb C\setminus \s(\mathcal L_{s,k})$. Thus, $\left(\fc-\mathcal L \right)^{-1}$ is an analytic operator-value function of $\fc$ on $\mathbb C\setminus\mathbb R$. And by Proposition \ref{prop-limit} $\mathring\Phi^{\pm}_s(k,y,\fc_r)=\lim_{\fc_i\to 0\pm}\mathring\Phi_s(k,y,\fc_r+i\fc_i)=\Phi_s^{\pm}(k,y,b_s^{-1}(\fc_r))$ is continuous in $\fc_r$, where $\Phi_s^{\pm}$ is given in \eqref{eq-Phi-lim}. Then we deduce from \eqref{eq:Rep2}, \eqref{eq-est-tPhi}, and \eqref{eq-Phi-lim} that
\begin{align*}
     \tilde{\psi}(t,k,y)=&\frac{1}{2\pi i}\lim_{T\to+\infty} \int^{T}_{-T} e^{-i \left(\fc_r-ic^*\right) t}\mathring\Phi_s(k,y,\fc_r-ic^*)-e^{-i \left(\fc_r+ic^*\right) t}\mathring\Phi_s(k,y,\fc_r+ic^*) d\fc_r\\
  =&\frac{1}{2\pi i}\lim_{T\to+\infty} \int^{T}_{-T} e^{-i\fc_r  t} \left(\mathring\Phi_s^-(k,y,\fc_r)-\mathring\Phi_s^+(k,y,\fc_r)\right)d\fc_r\\
  &+\frac{1}{2\pi}\lim_{T\to+\infty} \int^{c^*}_{0} e^{-i \left(-T-i\fc_i\right) t}\mathring\Phi_s(k,y,-T-i\fc_i)-e^{-i \left(T-i\fc_i\right) t}\mathring\Phi_s(k,y,T-i\fc_i) d\fc_i\\
  &+\frac{1}{2\pi}\lim_{T\to+\infty} \int^{c^*}_{0}e^{-i \left(-T+i\fc_i\right) t}\mathring\Phi_s(k,y,-T+i\fc_i)- e^{-i \left(T+i\fc_i\right) t}\mathring\Phi_s(k,y,T+i\fc_i)d\fc_i\\
  =&\frac{1}{2\pi i}\lim_{T\to+\infty} \int^{T}_{-T} e^{-i\fc_r  t} \left(\mathring\Phi_s^-(k,y,\fc_r)-\mathring\Phi_s^+(k,y,\fc_r)\right)d\fc_r\\
  =&-\frac{1}{\pi}\int_{-\infty}^{y} e^{-ikb_s(y') t} \frac{\mathcal J_{s,1} \mathcal J_{s,4} +\mathcal J_{s,2} \mathcal J_{s,3} }{\mathcal J_{s,1}^2+\mathcal J_{s,2}^2}\int_{+\infty}^y\frac{\phi_s(k,y,y')}{\phi^2_s(k,y_1,y')}dy_1b_s'(y')dy'\\
  &\quad-\frac{1}{\pi}\int_{y}^{+\infty} e^{-ikb_s(y') t}\frac{\mathcal J_{s,1} \mathcal J_{s,4} +\mathcal J_{s,2} \mathcal J_{s,3} }{\mathcal J_{s,1}^2+\mathcal J_{s,2}^2}\int_{-\infty}^y\frac{\phi_s(k,y,y')}{\phi^2_s(k,y_1,y')}dy_1b_s'(y')dy'.
\end{align*}
Then we arrive at the result of this proposition. We also refer the readers to \cite{LiMasmoudiZhao2022critical} for a different method of proof, where more accurate estimates can be found.
\end{proof}

\subsection{Wave operator related to the Rayleigh operator}
Now, we are in the position to give the wave operator $\mathbb{D}_{s,k}$ related to $\mathcal{R}_{s,k}$.
\begin{proposition}[Wave operator]\label{pro-wave}For $k\neq0$, $w(y),g(y)\in L^2(\mathbb R)$, let
\begin{align*}
\mathbb{D}_{s,k}[w](y')=\frac{1}{\pi} b_s'(y')\frac{\mathcal J_{s,1}(k,y')\mathcal J_{s,4}[w](y')+\mathcal J_{s,2}(y')\mathcal J_{s,3}[w](k,y')}{\sqrt{\mathcal J_{s,1}^2(k,y')+\mathcal J_{s,2}^2(y')}},
\end{align*}
and
\begin{align*}
  \mathbb{D}_{s,k}^1[g]
  =&\frac{b_s'(y')g(y')\mathcal J_1(y')+ \mathcal J_{s,3}[b_s''g](k,y')}{\sqrt{\mathcal J_{s,1}^2(k,y')+\mathcal J_{s,2}^2(y')}}.
\end{align*}
It holds that
\begin{align}\label{eq-wave}
  \mathbb{D}_{s,k}\big[\mathcal{R}_{s,k}w\big](y')=b_s(y')\mathbb{D}_{s,k}[w](y'),
\end{align}
and
\begin{align}\label{eq-wave-dual}
  \int_{\mathbb R} w(y)g(y)d y= \int_{\mathbb R} \mathbb{D}_{s,k}[w](y')\mathbb{D}_{s,k}^1[g](y')d y'.
\end{align}
Moreover, there exists $C\ge1$ independent of $k,s,\nu$ such that
\begin{align}\label{eq-est-wave}
  C^{-1}\|w\|_{L^2}\leq \|\mathbb{D}_{s,k}[w]\|_{L^2}\leq C\|w\|_{L^2},\quad C^{-1}\|w\|_{L^2}\leq \|\mathbb{D}_{s,k}^1[w]\|_{L^2}\leq C\|w\|_{L^2}.
\end{align}
\end{proposition}

\begin{proof}
We first prove the upper bounds in \eqref{eq-est-wave}. Recall the definition of $\mathcal J_{s,1}$, $\mathcal J_{s,2}$, $\mathcal J_{s,3}$, and $\mathcal J_{s,4}$. From Lemma \ref{lem-J-low}, we can see that $\sqrt{\mathcal J_{s,1}^2(k,y')+\mathcal J_{s,2}^2(y')}$ has a uniform lower bound. Therefore, in order to prove the upper bounds in \eqref{eq-est-wave}, it is sufficient to estimate $\left\|\mathcal J_{s,3}[\cdot](k,y')\right\|_{L^2\to L^2}$.

We write
\begin{align*}
  \mathcal J_{s,3}[w](k,y')=&\text{P.V.}\int \frac{\int_{y'}^{y_1}  w(y_2)\phi_1(y_2,y')dy_2}{\phi^2_s(y_1,y')}dy_1
  =\text{P.V.}\int \frac{\int_{y'}^{y_1}  w(y_2)\phi_1(y_2,y')dy_2}{(y_1-y')^2}\frac{(y_1-y')^2}{\phi^2_s(y_1,y')}dy_1\\
  =&-\text{P.V.}\int \int_{y'}^{y_1}  w(y_2)\phi_1(y_2,y')dy_2\frac{(y_1-y')^2}{\phi^2_s(y_1,y')}\pa_{y_1}\frac{1}{y_1-y'}dy_1\\
  =&\text{P.V.}\int \frac{w(y_1)}{y_1-y'}\frac{(y_1-y')^2}{\left(b_s(y_1)-b_s(y')\right)^2\phi_{s,1}(y_1,y')}dy_1\\
  &+\int \frac{\int_{y'}^{y_1}  w(y_2)\phi_1(y_2,y')dy_2}{y_1-y'}\pa_{y_1} \frac{(y_1-y')^2}{\phi^2_s(y_1,y')}dy_1\\
  =&\frac{1}{(b'_s(y'))^2}\text{P.V.}\int^{y'+\frac{1}{|k|}}_{y'-\frac{1}{|k|}} \frac{w(y_1)}{y_1-y'}dy_1\\
  &+\int^{y'+\frac{1}{|k|}}_{y'-\frac{1}{|k|}} \frac{w(y_1)}{y_1-y'}\left(\frac{(y_1-y')^2}{\left(b_s(y_1)-b_s(y')\right)^2\phi_{s,1}(y_1,y')}-\frac{1}{(b'_s(y'))^2}\right)dy_1\\
  &+ \int_{\mathbb R\setminus[y'-\frac{1}{|k|},y'+\frac{1}{|k|}]} \frac{w(y_1)}{y_1-y'}\frac{(y_1-y')^2}{\left(b_s(y_1)-b_s(y')\right)^2\phi_{s,1}(y_1,y')}dy_1\\
  &+ 2\int \frac{\int_{y'}^{y_1}  w(y_2)\phi_1(y_2,y')dy_2}{y_1-y'} \frac{(y_1-y') \int^{y_1}_{y'}\left(b_s'(y_1)-b_s'(y_3)\right)d_{y_3}}{\left(b_s(y_1)-b_s(y')\right)^3\phi_{s,1}^2(y_1,y')}dy_1\\
  &- 2\int \frac{\int_{y'}^{y_1}  w(y_2)\phi_1(y_2,y')dy_2}{y_1-y'} \frac{(y_1-y')^2\pa_{y_1}\phi_{s,1}(y_1,y')}{\left(b_s(y_1)-b_s(y')\right)^2\phi_1^3(y_1,y')}dy_1\\
  \eqdef&(\mathcal H-\mathcal H_{\frac{1}{|k|}}) [w](y')\\
  &+\mathcal J_{s,3,1}[w](k,y')+\mathcal J_{s,3,2}[w](k,y')+\mathcal J_{s,3,3}[w](k,y')+\mathcal J_{s,3,4}[w](k,y').
\end{align*}

It is clear that
\begin{align*}
  \left\|(\mathcal H-\mathcal H_{\frac{1}{|k|}}) [w](y')\right\|_{L^2}\lesssim \left\|w\right\|_{L^2}.
\end{align*}
For the integrand in $\mathcal J_{s,3,1}$, $\mathcal J_{s,3,2}$, $\mathcal J_{s,3,3}$, and $\mathcal J_{s,3,4}$, we use Proposition \ref{prop-phi} to get
\begin{align*}
  &\left|\frac{w(y_1)\chi_{<\frac{1}{|k|}}(y_1-y')}{y_1-y'}\left(\frac{(y_1-y')^2}{\left(b_s(y_1)-b_s(y')\right)^2\phi_{s,1}(y_1,y')}-\frac{1}{(b'_s(y'))^2}\right)\right|\\
  \lesssim& \left|w(y_1)\right| \chi_{<\frac{1}{|k|}}(y_1-y')\left(1+k^2\left|y_1-y'\right|\right),
\end{align*}
\begin{align*}
  \left|\frac{w(y_1)\chi_{>\frac{1}{|k|}}(y_1-y')}{y_1-y'}\frac{(y_1-y')^2}{\left(b_s(y_1)-b_s(y')\right)^2\phi_{s,1}(y_1,y')}\right|\lesssim\frac{\left|w(y_1)\right| \chi_{>\frac{1}{|k|}}(y_1-y')}{e^{C|k||y_1-y'|}\left|y_1-y'\right|},
\end{align*}
\begin{align*}
  &\left|\frac{\int_{y'}^{y_1}  w(y_2)\phi_{s,1}(y_2,y')dy_2}{y_1-y'} \frac{(y_1-y') \int^{y_1}_{y'}\left(b_s'(y_1)-b_s'(y_2)\right)d_{y_2}}{\left(b_s(y_1)-b_s(y')\right)^3\phi_{s,1}(y_1,y')}\right|\\
  \lesssim&\frac{\int_{y'}^{y_1}  |w(y_2)|dy_2}{y_1-y'}\left|\frac{(y_1-y') \int^{y_1}_{y'}\left(b_s'(y_1)-b_s'(y_2)\right)d_{y_2}}{\left(b_s(y_1)-b_s(y')\right)^3\phi_{s,1}^2(y_1,y')}\right| \\
  \lesssim&M[w](y_1)\frac{1}{e^{C|k||y_1-y'|}},\\
    &\left|\frac{\int_{y'}^{y_1}  w(y_2)\phi_1(y_2,y')dy_2}{y_1-y'} \frac{(y_1-y')^2\pa_{y_1}\phi_{s,1}(y_1,y')}{\left(b_s(y_1)-b_s(y')\right)^2\phi_1^3(y_1,y')}\right|\\
  \lesssim&M[w](y_1)\frac{|k|\min(1,|k||y_1-y'|)}{e^{C|k||y_1-y'|}},
\end{align*}
where $M[\cdot]$ is the Hardy-Littlewood maximal function
\begin{align*}
  M[w](v_1)\eqdef\sup_{r>0}\frac{1}{r}\int^r_{-r}|w(y_1-y)|dy.
\end{align*}

It follows from Young's convolution inequality that
\begin{align*}
  \left\|\mathcal J_{s,3,1}[w] \right\|_{L^2_{y'}}+\left\|\mathcal J_{s,3,2}[w] \right\|_{L^2_{y'}}+\left\|\mathcal J_{s,3,3}[w] \right\|_{L^2_{y'}}+\left\|\mathcal J_{s,3,4}[w] \right\|_{L^2_{y'}}\le C \left\|w\right\|_{L^2},
\end{align*}
where $C$ is a constant that depends only on $c_m$ and $\left\|b_{in}''\right\|_{H^2}$.

Consequently, we have
\begin{align*}
  \left\|\mathcal J_{s,3}[w](k,y')\right\|_{L^2}\le C \left\|w\right\|_{L^2},
\end{align*}
which gives 
\begin{align}\label{eq: upperboundDD_1}
 \|\mathbb{D}_{s,k}[w]\|_{L^2}\leq C\|w\|_{L^2},\quad \|\mathbb{D}_{s,k}^1[w]\|_{L^2}\leq C\|w\|_{L^2}.
\end{align}

Next, we turn to \eqref{eq-wave-dual}. By taking $t=0$ in \eqref{eq-rep-psi}, we have the following representation formula:
\begin{align*}
  \tilde{\psi}_{in,k}(y)=-\int_{\mathbb R}\mathbb{D}_{s,k}\big[\Delta_{k}\tilde{\psi}_{in,k}\big](y') \frac{\chi_{y'<y}\int_{+\infty}^y\frac{\phi_s(k,y,y')}{\phi^2_s(k,y_1,y')}dy_1+\chi_{y'>y}\int_{-\infty}^y\frac{\phi_s(k,y,y')}{\phi^2_s(k,y_1,y')}dy_1}{\sqrt{\mathcal J_{s,1}^2(k,y')+\mathcal J_{s,2}^2(y')}} dy'.
\end{align*}

Let $\psi=\Delta_k^{-1}w$. We have for $g\in C_c^\infty(\mathbb R)$ that
  \begin{align*}
    &\int_{\mathbb R} w(y)g(y)d y=\int_{\mathbb R} \psi(y)\Delta_kg(y)d y\\
    =& \int_{\mathbb R}\frac{\mathbb{D}_{s,k}[w](y')}{\sqrt{\mathcal J_{s,1}^2(k,y')+\mathcal J_{s,2}^2(y')}}\int_{\mathbb R}\frac{\int^{y_1}_{y'}\phi_s(y_2,y')\Delta_kg(y_2)dy_2}{\phi^2_s(y_1,y')}dy_1 dy'.
  \end{align*}
Next, we focus on $\int_{\mathbb R}\frac{\int^{y_1}_{y'}\phi_s(y_2,y')\Delta_kg(y_2)dy_2}{\phi^2_s(y_1,y')}dy_1$. Recall that $\phi_s$ solves \eqref{eq:homRay} with $\ep=0$.  Using integration by part twice, we get
\begin{align*}
  &\int_{\mathbb R}\frac{\int^{y_1}_{y'}\phi_s(y_2,y')\Delta_kg(y_2)dy_2}{\phi^2_s(y_1,y')}dy_1\\
  =&\text{P.V.}\int_{\mathbb R}\frac{1}{\phi^2_s(y_1,y')}\Big( \int^{y_1}_{y'}\pa_y^2\phi_s(y_2,y')g(y_2)-k^2\phi_s(y_2,y')g(y_2)dy_2\\
  &\qquad\qquad\qquad\qquad\qquad\qquad+\phi_s(y_1,y')g'(y_1)-\pa_y\phi_s(y_1,y')g(y_1)+\pa_y\phi_s(y',y')g(y')\Big)dy_1\\
  =&\text{P.V.}\int_{\mathbb R}\frac{1}{\phi^2_s(y_1,y')}\Big(\int^{y_1}_{y'}b''_2(y_2)\phi_1(y_2,y')g(y_2)dy_2\\
  &\qquad\qquad\qquad\qquad\qquad\qquad+\phi_s(y_1,y')g'(y_1)-\pa_y\phi_s(y_1,y')g(y_1)+\pa_y\phi_s(y',y')g(y')\Big)dy_1.
\end{align*}
We write
\begin{align*}
  &\text{P.V.}\int_{\mathbb R}\frac{\phi_s(y_1,y')g'(y_1)-\pa_y\phi_s(y_1,y')g(y_1)+\pa_y\phi_s(y',y')g(y')}{\phi^2_s(y_1,y')}dy_1\\
  =&\text{P.V.}\int_{\mathbb R}\frac{g'(y_1)}{\phi_s(y_1,y')}dy_1-\text{P.V.}\int_{\mathbb R}\frac{g(y_1)b'_s(y_1)\phi_{s,1}(y_1,y')-\pa_y\phi_s(y',y')g(y')}{\phi^2_s(y_1,y')}dy_1\\
  &-\text{P.V.}\int_{\mathbb R}\frac{g(y_1)\left(b_s(y_1)-b_s(y')\right)\phi_{s,1}'(y_1,y')}{\phi^2_s(y_1,y')}dy_1\eqdef\mathcal K_1+\mathcal K_2+\mathcal K_3.
\end{align*}
For each term, we derive that
\begin{align*}
  \mathcal K_1=&\text{P.V.}\int \frac{\pa_{y_1}\big( g(y_1)- g(y')\big)}{\phi_s(y_1,y')}dy_1=\text{P.V.}\int \frac{\big( g(y_1)- g(y')\big)\pa_{y_1}\phi_s(y_1,y')}{(b_s(y_1)-b_s(y'))^2\phi_{s,1}^2(y_1,y')}dy_1\\
  =&\text{P.V.}\int \frac{\big( g(y_1)- g(y')\big)b_s'(y_1)\phi_{s,1}(y_1,y')}{(b_s(y_1)-b_s(y'))^2\phi_{s,1}^2(y_1,y')}dy_1\\
  &+\text{P.V.}\int \frac{\big( g(y_1)- g(y')\big)\big(b_s(y_1)-b_s(y')\big)\phi_{s,1}'(y_1,y')}{(b_s(y_1)-b_s(y'))^2\phi_{s,1}^2(y_1,y')}dy_1\\
  \eqdef& \mathcal K_{1,1}+\mathcal K_{1,2}.
\end{align*}
Recalling that $\pa_y\phi_s(y',y')=b_s'(y')$, we have
\begin{align*}
  &\mathcal K_{1,1}+\mathcal K_2=- g(y')\text{P.V.}\int \frac{b_s'(y_1)\phi_{s,1}(y_1,y')-b_s'(y')}{(b_s(y_1)-b_s(y'))^2\phi_{s,1}^2(y_1,y')}dy_1\\
  =&- g(y')\text{P.V.}\int \frac{b_s'(y_1)-b_s'(y')}{\big(b_s(y_1)-b_s(y')\big)^2}dy_1- g(y')\int^{+\infty}_{-\infty}\frac{b_s'(y_1)}{\big(b_s(y_1)-b_s(y')\big)^2}\left(\frac{1}{\phi_{s,1}(y_1,y')}-1\right)dy_1\\
  &+b_s'(y')g(y')\int^{+\infty}_{-\infty}\frac{1}{\big(b_s(y_1)-b_s(y')\big)^2}\left(\frac{1}{\phi_{s,1}^2(y_1,y')}-1\right)dy_1,
\end{align*}
and
\begin{align*}
  \mathcal K_{1,2}+\mathcal K_3=&-g(y')\int^{+\infty}_{-\infty}\frac{\big(b_s(y_1)-b_s(y')\big)\pa_y\phi_{s,1}(y_1,y')}{(b_s(y_1)-b_s(y'))^2\phi_{s,1}^2(y_1,y')}dy_1\\
  =& g(y')\int^{+\infty}_{-\infty}\frac{1}{b_s(y_1)-b_s(y')}\pa_y\left(\frac{1}{\phi_{s,1}(y_1,y')}-1\right)dy_1\\
  =& g(y')\int^{+\infty}_{-\infty}\frac{b_s'(y_1)}{\big(b_s(y_1)-b_s(y')\big)^2}\left(\frac{1}{\phi_{s,1}(y_1,y')}-1\right)dy_1.
\end{align*}
Thus, we conclude that
\begin{align*}
  &\mathcal K_1+\mathcal K_2+\mathcal K_3
  =\mathcal K_{1,1}+\mathcal K_2+\mathcal K_{1,2}+\mathcal K_3\\
  =& g(y')\text{P.V.}\int \frac{b_s'(y')-b_s'(y_1)}{\big(b_s(y_1)-b_s(y')\big)^2}dy_1\\
  &+ b_s'(y')g(y')\int^{+\infty}_{-\infty}\frac{1}{\big(b_s(y_1)-b_s(y')\big)^2}\left(\frac{1}{\phi_{s,1}^2(y_1,y')}-1\right)dy'
  = b_s'(y')g(y')\mathcal J_1(y').
\end{align*}
Therefore, we have
  \begin{align*}
    &\int_{\mathbb R} w(y)g(y)d y\\
    =& \int_{\mathbb R}\frac{\mathbb{D}_{s,k}[w](y')}{\sqrt{\mathcal J_{s,1}^2(k,y')+\mathcal J_{s,2}^2(y')}}\text{P.V.}\int_{\mathbb R}\frac{\int^{y_1}_{y'}b''_s(y_2)\phi_1(k,y_2,y')g(y_2)dy_2}{\phi^2_s(k,y_1,y')}dy_1 dy'\\
    &+ \int_{\mathbb R}\frac{\mathbb{D}_{s,k}[w](y')}{\sqrt{\mathcal J_{s,1}^2(k,y')+\mathcal J_{s,2}^2(y')}} b_s'(y') g(y')\mathcal J_1(y')dy'\\
    =& \int_{\mathbb R} \mathbb{D}_{s,k}[w](y')\mathbb{D}_{s,k}^1[g](y')d y'.
  \end{align*}
The equality \eqref{eq-wave-dual} follows from the fact that $C_c^\infty(\mathbb R)$ is dense in $L^2$.

Last, we prove \eqref{eq-wave}. We focus on the following nonlocal term in $\mathbb{D}_{s,k}\big[\mathcal{R}_{s,k}w\big]$
\begin{align*}
  \mathcal J_{s,3}[b_s''\psi] (k,y')=\text{P.V.}\int_{\mathbb R}\frac{\int^{y}_{y'}b''_s(y_2)\psi(y_2)\phi_{s,1}(k,y_2,y')dy_2}{\phi^2_s(k,y,y')}dy.
\end{align*}
Recall that $\psi=\Delta_k^{-1}w$,  $\phi_s$ solves \eqref{eq:homRay} with $\ep=0$, then $b''_s(y_2)\phi_{s,1}(k,y_2,y')=\Delta_k\phi_s(k,y,y')$, $\phi_s(k,y',y')=0$, and $\pa_y\phi_s(k,y',y')=b_s'(y')$. By the same technique used in the proof of \eqref{eq-wave-dual}, we deduce that
\begin{align*}
  &\text{P.V.}\int_{\mathbb R}\frac{\int^{y}_{y'}b''_s(y_2)\psi(y_2)\phi_{s,1}(k,y_2,y')dy_2}{\phi^2_s(k,y,y')}dy=\text{P.V.}\int_{\mathbb R}\frac{\int^{y}_{y'}\psi(y_2)\Delta_k\phi_{s}(k,y_2,y')dy_2}{\phi^2_s(k,y,y')}dy\\
  =&\text{P.V.}\int_{\mathbb R}\frac{\int^{y}_{y'}\Delta_k\psi(y_2)\phi_{s}(k,y_2,y')dy_2}{\phi^2_s(k,y,y')}dy\\
  &+\text{P.V.}\int_{\mathbb R}\frac{\psi(y)\phi_{s}'(k,y,y')-\psi(y')\phi_{s}'(k,y',y')-\psi'(y)\phi_{s}(k,y,y')}{\phi^2_s(k,y,y')}dy\\
  =&\text{P.V.}\int_{\mathbb R}\frac{\int^{y}_{y'}b_s(y_2)w(y_2)\phi_{s,1}(k,y_2,y')dy_2}{\phi^2_s(k,y,y')}dy-b_s(y')\text{P.V.}\int_{\mathbb R}\frac{\int^{y}_{y'}w(y_2)\phi_{s,1}(k,y_2,y')dy_2}{\phi^2_s(k,y,y')}dy\\
  &-b_s'(y')\psi(y')\mathcal J_1(y')\\
  =&\mathcal J_{s,3}[b_sw] (k,y')-b_s(y')\mathcal J_{s,3}[w] (k,y')-\frac{\mathcal J_{s,1}(k,y')}{\mathcal J_{s,2}(y')}\mathcal J_{s,4}[b_s''\psi] (y')
\end{align*}
Then, one can easily check that
\begin{align*}
  &\mathbb{D}_{s,k}\big[\mathcal{R}_{s,k}w\big](y')-b_s(y')\mathbb{D}_{s,k}[w](y')\\
  =&\frac{b_s'(y')}{\pi}\frac{\mathcal J_{s,1}(k,y')\mathcal J_{s,4}[b_sw-b_s''\psi] (y')+\mathcal J_{s,2}(y')\mathcal J_{s,3}[b_sw-b_s''\psi] (k,y')}{\sqrt{\mathcal J_{s,1}^2(k,y')+\mathcal J_{s,2}^2(y')}}\\
  &-\frac{b_s'(y')}{\pi}\frac{b_s(y')\mathcal J_{s,1}(k,y')\mathcal J_{s,4}[w] (y')+b_s(y')\mathcal J_{s,2}(y')\mathcal J_{s,3}[w] (k,y')}{\sqrt{\mathcal J_{s,1}^2(k,y')+\mathcal J_{s,2}^2(y')}}\\
  =&\frac{b_s'(y')}{\pi}\frac{\left(\mathcal J_{s,1}(k,y')\mathcal J_{s,4}[b_sw] (y')-b_s(y')\mathcal J_{s,1}(k,y')\mathcal J_{s,4}[w] (y')\right)}{\sqrt{\mathcal J_{s,1}^2(k,y')+\mathcal J_{s,2}^2(y')}}
  =0,
\end{align*}
which is \eqref{eq-wave}.

The lower bound in \eqref{eq-est-wave} follows directly from a duality argument and the upper bounds \eqref{eq: upperboundDD_1}. 
This completes the proof of this Proposition.
\end{proof}

Next, we present some commutator estimates for the wave operator $\mathbb{D}_{s,k}$, which will be utilized in establishing the linear enhanced dissipation. We denote the commutator of the wave operator and the derivative by
\begin{align*}
  [\pa_s,\mathbb{D}_{s,k}][w](s,y')=\pa_s \left(\mathbb{D}_{s,k}[w]\right)(s,y')-\left(\mathbb{D}_{s,k}[\pa_sw]\right)(s,y'),\\
  [\pa_{y},\mathbb{D}_{s,k}][w](s,y')=\pa_{y'} \left(\mathbb{D}_{s,k}[w]\right)(s,y')-\left(\mathbb{D}_{s,k}[\pa_yw]\right)(s,y'),
\end{align*}
and have the following estimates.
\begin{lemma}\label{lem-wave-com}
For $k\neq 0$ and $w(s,y)\in C^1([0,\infty);H^1(\mathbb{R}))$, it holds that 
\begin{align}
&\|[\pa_s,\mathbb{D}_{s,k}][w]\|_{L^2}\leq C\nu\|w\|_{L^2},\label{eq-wave-com-s}\\
&\|[\pa_y,\mathbb{D}_{s,k}][w]\|_{L^2}+\|[\pa_y,\mathbb{D}_{s,k}^1][w]\|_{L^2}\leq C\|w\|_{L^2},\label{eq-wave-com-y}\\
&\|[\pa_{yy},\mathbb{D}_{s,k}][w]\|_{L^2}+\|[\pa_{yy},\mathbb{D}_{s,k}^1][w]\|_{L^2}\leq C\Big(\|\pa_yw\|_{L^2}+\|w\|_{L^2}\Big),\label{eq-wave-com-yy}
\end{align}
where $C>1$ is a constant depends on $c_m$ and $\left\|b_{in}''\right\|_{H^3}$ and is independent of $k,s,\nu$.
\end{lemma}

\begin{remark}\label{rmk-time-back}
    It is important for us to get $(s,\nu)$ independent estimates, since the shear flow is changing in time variable $s$ and is $\nu$ related. If the shear flow is fixed namely,  a special extra force term is added in the original system \eqref{eq:NS}, then Proposition \ref{prop-est-s}, \eqref{eq-wave-com-s} and Lemma \ref{lem-back} are not necessary and the proofs of Lemma \ref{lem-bdPi2} and Lemma \ref{lem-J-low} can be easier.  
\end{remark}

\begin{proof}
It follows from the definition of $\mathbb{D}_{s,k}$ that
\begin{align*}
   [\pa_s,\mathbb{D}_{s,k}][w](s,y')
   =&\mathcal J_{s,4}[w](s,y') \pa_s \left(\frac{1}{\pi} b_s'(y')\frac{\mathcal J_{s,1}(y')}{\sqrt{\mathcal J_{s,1}^2(k,y')+\mathcal J_{s,2}^2(y')}}\right)\\
   &+ \mathcal J_{s,3}[w](s,k,y')\pa_s \left(\frac{1}{\pi} b_s'(y')\frac{\mathcal J_{s,2}(y')}{\sqrt{\mathcal J_{s,1}^2(k,y')+\mathcal J_{s,2}^2(y')}}\right)\\
   &+\frac{1}{\pi} b_s'(y')\frac{\mathcal J_{s,1}(y')[\pa_s,\mathcal J_{s,4}][w](s,y')+\mathcal J_{s,2}(y')[\pa_s,\mathcal J_{s,3}][w](s,k,y')}{\sqrt{\mathcal J_{s,1}^2(k,y')+\mathcal J_{s,2}^2(y')}}.
 \end{align*}

Recall that $b_s$ solves \eqref{eq:linearHeat}. Taking the derivative of $\mathcal J_{s,1}$ and $\mathcal J_{s,2}$ with respect to s, we obtain 
\begin{align*}
  \pa_s\mathcal J_{s,2}(y')&=\pi\pa_s\frac{b_s''(y')}{{b_s'}(y')^3}=\nu\pi\frac{b_s''''(y'){b_s'}(y')-3b_s''(y')b_s'''(y')}{{b_s'}(y')^4},
\end{align*}
and
\begin{align*}
  \pa_s\mathcal J_{s,1}=&-\pa_s\mathcal H\Big((b_{s}^{-1})''\Big)(b_s(y'))+\pa_s\Pi_{s,2}(k,y')\\
  =&-\nu b''_s(y')\mathcal H\Big((b_{s}^{-1})'''\Big)(b_s(y'))-\mathcal H\Big(\pa_s(b_{s}^{-1})''\Big)(b_s(y'))+\pa_s\Pi_{s,2}(k,y').
\end{align*}
Here
\begin{align*}
  (b_{s}^{-1})'''(v)=\pa_v^3(b_s^{-1})(v)=-\frac{b_s'''(b_s^{-1}(v))b_s'(b_s^{-1}(v))-3 \left(b_s''(b_s^{-1}(v))\right)^2}{\left(b_s'(b_s^{-1}(v))\right)^5},
\end{align*}
and
\begin{align*}
  &\pa_s (b_s^{-1})''(v)=\pa_s\pa_v^2(b_s^{-1})(v)\\
  =&-\nu\frac{b_s''''(b_s^{-1}(v))\left(b_s'(b_s^{-1}(v))\right)-4b_s'''(b_s^{-1}(v))b_s''(b_s^{-1}(v))b_s'(b_s^{-1}(v))+3\left(b_s''(b_s^{-1}(v))\right)^3}{\left(b_s'(b_s^{-1}(v))\right)^5}.
\end{align*}

We decompose $\pa_s\Pi_{s,2}(k,y')$ into two parts,
\begin{align*}
  \pa_s\Pi_{s,2}(k,y')=&-2 \nu\int \frac{b_s''(y)-b_s''(y')}{(b_s(y)-b_s(y'))^3}\left(\frac{1}{\phi_{s,1}^2(k,y,y')}-1\right) dy\\
  &-2 \int \frac{1}{(b_s(y)-b_s(y'))^2} \frac{\pa_s\phi_{s,1}}{\phi_{s,1}^3(k,y,y')}  dy.
\end{align*}
For the first term, it follows from Lemma \ref{lem-back} and Proposition \ref{prop-phi} that
\begin{align*}
  \left|\nu\int \frac{b_s''(y)-b_s''(y')}{(b_s(y)-b_s(y'))^3}\left(\frac{1}{\phi_{s,1}^2(k,y,y')}-1\right) dy\right|\le C\nu|k|.
\end{align*}
For the second term, by using \eqref{eq-est-s}, we have 
\begin{align*}
  \left|\int \frac{1}{(b_s(y)-b_s(y'))^2} \frac{\pa_s\phi_{s,1}}{\phi_{s,1}^3(k,y,y')}  dy\right|\lesssim \nu.
\end{align*}

Combing the above estimates with Lemma \ref{lem-J-low}, we obtain
\begin{align*}
  \left|\pa_s\mathcal J_{s,1}(k,y')\right|+\left|\pa_s\mathcal J_{s,2}(y')\right|\lesssim \nu \left(\left| \mathcal J_{s,1}(k,y')\right|+\left|\mathcal J_{s,2}(y')\right|\right).
\end{align*}

It follows that
\begin{equation}\label{eq-est-J1J2s}
  \begin{aligned}    
  &\left|\pa_s\frac{\mathcal J_{s,1}(y')}{\sqrt{\mathcal J_{s,1}^2(k,y')+\mathcal J_{s,2}^2(y')}}\right|+\left|\pa_s\frac{\mathcal J_{s,2}(y')}{\sqrt{\mathcal J_{s,1}^2(k,y')+\mathcal J_{s,2}^2(y')}}\right|\\
  \le& \left|\frac{\mathcal J_{s,2}(y')^2\pa_s\mathcal J_{s,1}(y')-\mathcal J_{s,1}(k,y')\mathcal J_{s,2}(y')\pa_s\mathcal J_{s,2}(y')}{\left(\sqrt{\mathcal J_{s,1}^2(k,y')+\mathcal J_{s,2}^2(y')}\right)^3}\right|\\
  &+\left|\frac{\mathcal J_{s,1}(k,y')^2\pa_s\mathcal J_{s,2}(y')-\mathcal J_{s,1}(k,y')\mathcal J_{s,2}(y')\pa_s\mathcal J_{s,1}(y')}{\left(\sqrt{\mathcal J_{s,1}^2(k,y')+\mathcal J_{s,2}^2(y')}\right)^3}\right|
  \le C \frac{\nu}{\delta^{\frac{1}{2}}}.    
  \end{aligned}
\end{equation}
where $\delta$ is given in Lemma \ref{lem-J-low}.

Next, we focus on $[\pa_s,\mathcal J_{s,3}][w]$. It holds that
\begin{align*}
  &[\pa_s,\mathcal J_{s,3}][w](s,k,y')\\
  =&\pa_s\text{P.V.}\int \frac{\int_{y'}^{y_1}w(s,y_2)\phi_{s,1}(y_2,y')dy_2}{\phi^2_s(y_1,y')}dy_1-\text{P.V.}\int \frac{\int_{y'}^{y_1}\pa_sw(s,y_2)\phi_{s,1}(y_2,y')dy_2}{\phi^2_s(y_1,y')}dy_1\\
  =&\int \frac{\int_{y'}^{y_1}w(s,y_2)\pa_s\phi_{s,1}(y_2,y')dy_2}{\phi^2_s(y_1,y')}dy_1\\
  &-2\int \frac{\int_{y'}^{y_1}w(s,y_2)\phi_{s,1}(y_2,y')dy_2 \left(b_s(y_1)-b_s(y')\right)\pa_s\phi_{s,1}(y_1,y')}{\phi^3_s(y_1,y')}dy_1\\
  &-2\text{P.V.}\int \frac{\int_{y'}^{y_1}w(s,y_2)\phi_{s,1}(y_2,y')dy_2 \left(\pa_sb_s(y_1)-\pa_sb_s(y')\right)\phi_{s,1}(y_1,y')}{\phi^3_s(y_1,y')}dy_1\\
  \eqdef&\mathcal J_{s,3}^{(s),1}[w](k,y')+\mathcal J_{s,3}^{(s),2}[w](k,y')+\mathcal J_{s,3}^{(s),3}[w](k,y').
\end{align*}
For $\mathcal J_{s,3}^{(s),1}[w](k,y')$ and $\mathcal J_{s,3}^{(s),2}[w](k,y')$, we can see from Proposition \ref{prop-est-s} that
\begin{align*}
  \left|\frac{\int_{y'}^{y_1}w(s,y_2)\pa_s\phi_{s,1}(y_2,y')dy_2}{\phi^2_s(y_1,y')}\right|\le M[w](y_1)\frac{\nu\min(|k|^2|y_1-y'|,\frac{1}{|y_1-y'|})}{e^{C|k||y_1-y'|}},
\end{align*}
and
\begin{align*}
  \left|\frac{\int_{y'}^{y_1}w(s,y_2)\phi_{s,1}(y_2,y')dy_2 \left(b_s(y_1)-b_s(y')\right)\pa_s\phi_{s,1}(y_1,y')}{\phi^3_s(y_1,y')}\right|\le M[w](y_1)\frac{\nu\min(|k|^2|y_1-y'|,\frac{1}{|y_1-y'|})}{e^{C|k||y_1-y'|}}.
\end{align*}
Then it follows from Young’s convolution inequality that 
\begin{align*}
  \left\|\mathcal J_{s,3}^{(s),1}[w](k,y')\right\|_{L^2}+\left\|\mathcal J_{s,3}^{(s),2}[w](k,y')\right\|_{L^2}\lesssim\nu \left\|w\right\|_{L^2}.
\end{align*}

The last term $\mathcal J_{s,3}^{(s),3}[w](k,y')$ is the most troublesome term. We write
\begin{align*}
  &\mathcal J_{s,3}^{(s),3}[w](k,y')\\
  =&2\nu\text{P.V.}\int \frac{\int_{y'}^{y_1}w(s,y_2)\phi_{s,1}(y_2,y')dy_2 \left(b_s''(y_1)-b_s''(y')\right)(y_1-y')^2}{\left(b_s(y_1)-b_s(y')\right)^3\phi_{s,1}^2(y_1,y')}\pa_{y_1}\frac{1}{y_1-y'}dy_1\\
  =&-2\nu\text{P.V.}\int \frac{w(s,y_1)\left(b_s''(y_1)-b_s''(y')\right)(y_1-y')^2}{(y_1-y')\left(b_s(y_1)-b_s(y')\right)^3\phi_{s,1}(y_1,y')}dy_1\\
  &-2\nu\text{P.V.}\int \frac{\int_{y'}^{y_1}w(s,y_2)\phi_{s,1}(y_2,y')dy_2 }{y_1-y'}\pa_{y_1}\frac{\left(b_s''(y_1)-b_s''(y')\right)(y_1-y')^2}{\left(b_s(y_1)-b_s(y')\right)^3\phi_{s,1}^2(y_1,y')}dy_1\\
  =&-2\nu\frac{b'''_s(y')}{(b'_s(y'))^3}\text{P.V.}\int^{y'+\frac{1}{|k|}}_{y'-\frac{1}{|k|}} \frac{w(s,y_1)}{y_1-y'}dy_1\\
  &-2\nu\int^{y'+\frac{1}{|k|}}_{y'-\frac{1}{|k|}} \frac{w(s,y_1)}{y_1-y'}\left(\frac{\left(b_s''(y_1)-b_s''(y')\right)(y_1-y')^2}{\left(b_s(y_1)-b_s(y')\right)^3\phi_{s,1}(y_1,y')}-\frac{b'''_s(y')}{(b'_s(y'))^3}\right)dy_1\\
  &-2\nu \int_{\mathbb R\setminus[y'-\frac{1}{|k|},y'+\frac{1}{|k|}]} \frac{w(s,y_1)}{y_1-y'}\frac{\left(b_s''(y_1)-b_s''(y')\right)(y_1-y')^2}{\left(b_s(y_1)-b_s(y')\right)^3\phi_{s,1}(y_1,y')}dy_1\\
  &-2\nu\int \frac{\int_{y'}^{y_1}  w(s,y_2)\phi_1(y_2,y')dy_2}{y_1-y'}\frac{1}{\phi_{s,1}^2(y_1,y')} \pa_{y_1}\frac{\left(b_s''(y_1)-b_s''(y')\right)(y_1-y')^2}{\left(b_s(y_1)-b_s(y')\right)^3}dy_1\\
  &+ 4\nu\int \frac{\int_{y'}^{y_1}  w(s,y_2)\phi_1(y_2,y')dy_2}{y_1-y'} \frac{\left(b_s''(y_1)-b_s''(y')\right)(y_1-y')^2\pa_{y_1}\phi_{s,1}(y_1,y')}{\left(b_s(y_1)-b_s(y')\right)^3\phi_1^3(y_1,y')}dy_1.
\end{align*}
One can utilize the approach employed to treat $\left\|\mathcal J_{s,3}[\cdot](k,y')\right\|_{L^2\to L^2}$ in Proposition \ref{pro-wave} to estimate all these terms and obtain
\begin{align*}
  \left\|\mathcal J_{s,3}^{(s),3}[w](k,y')\right\|_{L^2}\lesssim\nu \left\|w\right\|_{L^2},
\end{align*}
which gives
\begin{align*}
  \left\|[\pa_s,\mathcal J_{s,3}][w]\right\|_{L^2}\le C\nu \left\|w\right\|_{L^2}
\end{align*}

One can easily check that
\begin{align*}
  \left\|[\pa_s,\mathcal J_{s,4}][w]\right\|_{L^2}\lesssim\nu \left\|w\right\|_{L^2}.
\end{align*}
As a conclusion, we have
\begin{align*}
  \|[\pa_s,\mathbb{D}_{s,k}][w]\|_{L^2}\leq C\nu \|w\|_{L^2},
\end{align*}
where $C$ is a constant depends only on $c_m$ and $\left\|b_{in}''\right\|_{H^3}$.

Next, we turn to \eqref{eq-wave-com-y}. We have
\begin{align*}
  \pa_{y'}\mathcal J_{s,1}=&-\pa_{y'}\mathcal H\Big((b_{s}^{-1})''\Big)(b_s(y'))+\pa_{y'}\Pi_{s,2}(k,y').
\end{align*}
It holds that
\begin{align*}
   &-\pa_{y'}\mathcal H\Big((b_{s}^{-1})''\Big)(b_s(y'))=- b'_s(y')\mathcal H\Big((b_{s}^{-1})'''\Big)(b_s(y'))-\mathcal H\Big((b_{s}^{-1})'''\Big)(b_s(y')),
\end{align*} 
and
\begin{align*}
  \pa_{y'}\Pi_{s,2}(k,y')=&\int \pa_G\left(\frac{1}{(b_s(y)-b_s(y'))^2}\left(\frac{1}{\phi_{s,1}^2(k,y,y')}-1\right)\right) dy\\
  =&-2 \int \frac{b_s'(y)-b_s'(y')}{(b_s(y)-b_s(y'))^3}\left(\frac{1}{\phi_{s,1}^2(k,y,y')}-1\right) dy-2 \int \frac{1}{(b_s(y)-b_s(y'))^2} \frac{\pa_G\phi_{s,1}}{\phi_{s,1}^3(k,y,y')}  dy.
\end{align*}
Then by employing Proposition \ref{prop-est-G}, we obtain
\begin{align*}
  \left|\pa_{y'}\Pi_{s,2}(k,y')\right|\le C|k|.
\end{align*} 

It is easy to check that
\begin{align*}
  \pa_{y'}\mathcal J_{s,2}(y')&=\pi\pa_{y'}\frac{b_s''(y')}{{b_s'}(y')^3}=\pi\frac{b_s'''(y'){b_s'}(y')-3 \left(b_s''(y')\right)^2}{{b_s'}(y')^4}.
\end{align*}
Then similar to \eqref{eq-est-J1J2s}, we have
\begin{align*}
  \left|\pa_{y'}\frac{\mathcal J_{s,1}(y')}{\sqrt{\mathcal J_{s,1}^2(k,y')+\mathcal J_{s,2}^2(y')}}\right|+\left|\pa_{y'}\frac{\mathcal J_{s,2}(y')}{\sqrt{\mathcal J_{s,1}^2(k,y')+\mathcal J_{s,2}^2(y')}}\right| \le &C \frac{1}{\delta^{\frac{1}{2}}}.
\end{align*}

Similar to the proof of \eqref{eq-wave-com-s}, the estimate of $[\pa_s,\mathcal J_{s,4}][w]$ is straightforward, and the crucial step is to estimate $[\pa_{y},\mathcal J_{s,3}][w]$. It holds that
\begin{align*} 
  \pa_{y'}\mathcal J_{s,3}[w](k,y')=&\pa_{y'}\text{P.V.}\int \frac{\int_{y'}^{y_1}w(y_2)\phi_{s,1}(y_2,y')dy_2}{\phi^2_s(y_1,y')}dy_1\\
  =&\text{P.V.}\int (\pa_{y_1}+\pa_{y'})\frac{\int_{y'}^{y_1}w(y_2)\phi_{s,1}(y_2,y')dy_2}{\phi^2_s(y_1,y')}dy_1\\
  =&\text{P.V.}\int \int_{y'}^{y_1}(\pa_{y_2}+\pa_{y_1}+\pa_{y'})\frac{w(y_2)\phi_{s,1}(y_2,y')dy_2}{\phi^2_s(y_1,y')}dy_1\\
  =&\text{P.V.}\int \frac{\int_{y'}^{y_1}\pa_{y_2}w(y_2)\phi_{s,1}(y_2,y')dy_2}{\phi^2_s(y_1,y')}dy_1\\
  &+ \int \frac{\int_{y'}^{y_1}w(y_2)(\pa_{y_2}+\pa_{y'})\phi_{s,1}(y_2,y')dy_2}{\phi^2_s(y_1,y')}dy_1\\
  &+\text{P.V.}\int \int_{y'}^{y_1}w(y_2)\phi_{s,1}(y_2,y')dy_2(\pa_{y_1}+\pa_{y'})\frac{1}{\phi^2_s(y_1,y')}dy_1.
\end{align*}
Then we have
\begin{align*}
  &\pa_{y'}\mathcal J_{s,3}[w](k,y')-\mathcal J_{s,3}[\pa_yw](k,y')\\
  =& \int \frac{\int_{y'}^{y_1}w(y_2)\pa_G\phi_{s,1}(y_2,y')dy_2}{\phi^2_s(y_1,y')}dy_1\\
  &-2 \int \frac{\int_{y'}^{y_1}w(y_2)\phi_{s,1}(y_2,y')dy_2}{(b_s(y')-b_s(y_1))^2}\frac{\pa_G\phi_{s,1}(y_1,y')}{\phi^3_{s,1}(y_1,y')}dy_1\\
  &-2\text{P.V.}\int \frac{\int_{y'}^{y_1}w(y_2)\phi_{s,1}(y_2,y')dy_2}{(b_s(y')-b_s(y_1))^3}\frac{b_s'(y')-b_s'(y_1)}{\phi^2_{s,1}(y_1,y')}dy_1.
\end{align*}
By using Proposition \ref{prop-est-G} and following the proof of \eqref{eq-wave-com-s}, we obtain
\begin{align*}
  \left\|[\pa_{y},\mathcal J_{s,3}][w]\right\|_{L^2}\leq C \|w\|_{L^2},
\end{align*}
and
\begin{align*}
  \|[\pa_{y},\mathbb{D}_{s,k}](w)\|_{L^2}\leq C \|w\|_{L^2}.
\end{align*}

At last, we give the proof of \eqref{eq-wave-com-yy}. Similar to \eqref{eq-wave-com-s}, one can easily check that 
\begin{align*}
  \left|\pa_{y'y'}\frac{\mathcal J_{s,1}(y')}{\sqrt{\mathcal J_{s,1}^2(k,y')+\mathcal J_{s,2}^2(y')}}\right|+\left|\pa_{y'y'}\frac{\mathcal J_{s,2}(y')}{\sqrt{\mathcal J_{s,1}^2(k,y')+\mathcal J_{s,2}^2(y')}}\right|\lesssim \frac{1}{\delta},
\end{align*}
and
\begin{align*}
  \left\|[\pa_{yy},\mathcal J_{s,4}][w]\right\|_{L^2}\lesssim\nu \left\|w\right\|_{H^1}.
\end{align*}

For the crucial term $[\pa_{yy},\mathcal J_{s,3}][w]$, we have
\begin{align*} 
  \pa_{y'y'}\mathcal J_{s,3}[w](k,y')
  =&\text{P.V.}\int (\pa_{y_1}+\pa_{y'})^2\frac{\int_{y'}^{y_1}w(y_2)\phi_{s,1}(y_2,y')}{\phi^2_s(y_1,y')}dy_2dy_1\\
  =&\text{P.V.}\int \int_{y'}^{y_1}(\pa_{y_2}+\pa_{y_1}+\pa_{y'})^2\frac{w(y_2)\phi_{s,1}(y_2,y')dy_2}{\phi^2_s(y_1,y')}dy_1\\
  =&\text{P.V.}\int \frac{\int_{y'}^{y_1}\pa_{y_2y_2}w(y_2)\phi_{s,1}(y_2,y')dy_2}{\phi^2_s(y_1,y')}dy_1\\
  &+ \int \frac{\int_{y'}^{y_1}w(y_2)(\pa_{y_2}+\pa_{y'})^2\phi_{s,1}(y_2,y')dy_2}{\phi^2_s(y_1,y')}dy_1\\
  &+ \text{P.V.}\int \int_{y'}^{y_1}w(y_2)\phi_{s,1}(y_2,y')dy_2(\pa_{y_1}+\pa_{y'})^2\frac{1}{\phi^2_s(y_1,y')}dy_1\\
  &+2 \int \frac{\int_{y'}^{y_1}\pa_{y_2}w(y_2)(\pa_{y_2}+\pa_{y'})\phi_{s,1}(y_2,y')dy_2}{\phi^2_s(y_1,y')}dy_1\\
  &+2 \text{P.V.}\int \int_{y'}^{y_1}\pa_{y_2}w(y_2)\phi_{s,1}(y_2,y')dy_2(\pa_{y_1}+\pa_{y'})\frac{1}{\phi^2_s(y_1,y')}dy_1\\
  &+2 \int \int_{y'}^{y_1} w(y_2)(\pa_{y_2}+\pa_{y'})\phi_{s,1}(y_2,y')dy_2(\pa_{y_1}+\pa_{y'})\frac{1}{\phi^2_s(y_1,y')}dy_1,
\end{align*}
and then
\begin{align*}
  &\pa_{y'y'}\mathcal J_{s,3}[w](k,y')-\mathcal J_{s,3}[\pa_{yy}f](k,y')\\
  =& \int \frac{\int_{y'}^{y_1}w(y_2)(\pa_{y_2}+\pa_{y'})^2\phi_{s,1}(y_2,y')dy_2}{\phi^2_s(y_1,y')}dy_1\\
  &+ \text{P.V.}\int \int_{y'}^{y_1}w(y_2)\phi_{s,1}(y_2,y')dy_2(\pa_{y_1}+\pa_{y'})^2\frac{1}{\phi^2_s(y_1,y')}dy_1\\
  &+2 \int \frac{\int_{y'}^{y_1}\pa_{y_2}w(y_2)(\pa_{y_2}+\pa_{y'})\phi_{s,1}(y_2,y')dy_2}{\phi^2_s(y_1,y')}dy_1\\
  &+2 \text{P.V.}\int \int_{y'}^{y_1}\pa_{y_2}w(y_2)\phi_{s,1}(y_2,y')dy_2(\pa_{y_1}+\pa_{y'})\frac{1}{\phi^2_s(y_1,y')}dy_1\\
  &+2 \int \int_{y'}^{y_1} w(y_2)(\pa_{y_2}+\pa_{y'})\phi_{s,1}(y_2,y')dy_2(\pa_{y_1}+\pa_{y'})\frac{1}{\phi^2_s(y_1,y')}dy_1.
\end{align*}
By following the proof of \eqref{eq-wave-com-s}, and applying Proposition \ref{prop-est-G}, we get
\begin{align*}
  \left\|[\pa_{yy},\mathcal J_{s,3}][w]\right\|_{L^2}\leq C \|w\|_{H^1},
\end{align*}
and
\begin{align*}
  \|[\pa_{yy},\mathbb{D}_{s,k}](w)\|_{L^2}\leq C \|w\|_{H^1},
\end{align*}
where $C$ is a constant depends only on $c_m$ and $\left\|b_{in}''\right\|_{H^3}$.
\end{proof}

\section{Linear enhanced dissipation}\label{sec-lin}
In this section, we study the linearized equation of \eqref{eq:NS2}, and establish the linear enhanced dissipation using the wave operator $\mathbb{D}_{s,k}$ which is introduced in the last section.

For $t>s$, let $S(t,s)g$ solves the following linearized Navier-Stokes equation:
\begin{equation}\label{eq: linearNS}
  \left\{
    \begin{array}{l}
      \pa_t\om+b(t,y)\pa_x\om-b''(t,y)\pa_x\Delta^{-1}\om-\nu\Delta\om=0,\\
      \om(s,x,y)=g(x,y),
    \end{array}
  \right.
\end{equation}
with $\int_{\mathbb T} g(x,y) d x=0$. We have the following estimates.
\begin{proposition}\label{prop-semi}
Given $g(x,y)\in H^{log}_xL^2_y(\mathbb T\times\mathbb R)$ with $\int_{\mathbb T} g(x,y) d x=0$. There exists constants $c_0,C_0>0$ such that for any $t\ge s\ge0$, it holds that
\begin{align}
\left\|e^{c_0\nu^{\frac{1}{3}}t}S(t,s)g(x,y)\right\|_{H^{log}_xL^2_y}&\leq C_0 \left\|e^{c_0\nu^{\frac{1}{3}}s}g(x,y)\right\|_{H^{log}_xL^2_y}\label{btsp:om-point}\\
\left(\int_{s}^t\|e^{c_0\nu^{\frac{1}{3}}s'} \nabla S(s',s)g(x,y)\|_{H^{log}_xL^2_y}^2ds'\right)^{\frac{1}{2}}&\leq C_0\nu^{-\frac{1}{2}}\left\|e^{c_0\nu^{\frac{1}{3}}s}g(x,y)\right\|_{H^{log}_xL^2_y}\label{btsp:pa_xyomL2L2L2}\\
\int_{s}^t\|e^{c_0\nu^{\frac{1}{3}}s'} \pa_xS(s',s)g(x,y)\|_{H^{log}_xL^2_y}ds'&\leq C_0\nu^{-\frac{1}{2}}\left\|e^{c_0\nu^{\frac{1}{3}}s}g(x,y)\right\|_{H^{log}_xL^2_y}\label{btsp:pa_xomL1L2L2}\\
\left(\int_{s}^t\|e^{c_0\nu^{\frac{1}{3}}s'} S(s',s)g(x,y)\|_{L^{\infty}_{x,y}}^2ds'\right)^{\frac{1}{2}}&\leq C_0\nu^{-\frac{1}{2}}\left\|e^{c_0\nu^{\frac{1}{3}}s}g(x,y)\right\|_{H^{log}_xL^2_y}\label{btsp:omL2L1Linfty}
\end{align}
\begin{align}
\left(\int_{s}^t\|e^{c_0\nu^{\frac{1}{3}}s'}\pa_x\Delta^{-1}S(s',s)g(x,y)\|_{L^{\infty}_{x,y}}^2ds'\right)^{\frac{1}{2}}
&\leq C_0\left\|e^{c_0\nu^{\frac{1}{3}}s}g(x,y)\right\|_{H^{log}_xL^2_y},\label{btsp:V_2L2LinftyLinfty}\\
\left(\int_{s}^t\|e^{c_0\nu^{\frac{1}{3}}s'}|D_x|^{\frac{1}{2}} \pa_x\Delta^{-1}S(s',s)g(x,y)\|_{H^{log}_xL^\infty_y}^2ds'\right)^{\frac{1}{2}}
&\leq C_0\left\|e^{c_0\nu^{\frac{1}{3}}s}g(x,y)\right\|_{H^{log}_xL^2_y},\label{btsp:V_2L2L2Linfty}\\
\left(\int_{s}^t\|e^{c_0\nu^{\frac{1}{3}}s'}\pa_{xy}\Delta^{-1}S(s',s)g(x,y)\|_{H^{log}_xL^2_y}^2ds'\right)^{\frac{1}{2}}
&\leq C_0\left\|e^{c_0\nu^{\frac{1}{3}}s}g(x,y)\right\|_{H^{log}_xL^2_y};\label{btsp:V_1L2L2L2}
\end{align}
\begin{align}
  \sup_{s'\in[s,t]}\|e^{c_0\nu^{\frac{1}{3}}s'}\pa_{y}\Delta^{-1}S(s',s)g(x,y)\|_{L^{\infty}_{x,y}}
\leq C_0\left\|e^{c_0\nu^{\frac{1}{3}}s}g(x,y)\right\|_{H^{log}_xL^2_y}.\label{btsp:V_1LinftyL1Linfty}
\end{align}
\end{proposition}
\subsection{New equations}
Recalling \eqref{eq: waveoperatorinxy} and \eqref{eq-wave}, we apply the wave operator $\mathbb{D}_t$ on \eqref{eq: linearNS} and get 
\begin{align}
  \pa_t(\mathbb{D}_t[\om])&+b(t,y)\pa_x(\mathbb{D}_t[\om])
-\nu\Delta(\mathbb{D}_t[\om])+\mathrm{com}_t[\om]=0,
\end{align}
where $\mathrm{com}_t[\om]=\sum\limits_{k}\Big([\pa_t,\mathbb{D}_{t,k}][\tilde{\om}_k](t,y)+\nu [\pa_{yy},\mathbb{D}_{t,k}][\tilde{\om}_{k}](t,y)\Big)e^{ik x}$. 
We now introduce a change of coordinate. Let $v=b(t,y)$, $z=x-tv$ and 
\begin{align*}
F(t,z,v)=(\mathbb{D}_t[\om])(t,x,y), \quad \mathrm{Com}(t,z,v)=\mathrm{com}_t[\om](t,x,y)
\end{align*}
Thus 
\begin{align*}
&\pa_t(\mathbb{D}_t[\om])=\pa_tF-b(t,y)\pa_zF-t\pa_tb(t,y)\pa_zF+\pa_tb(t,y)\pa_vF\\
&\Delta(\mathbb{D}_t[\om])=\pa_{zz}F+(\pa_yb(t,y))^2(\pa_v-t\pa_z)^2F+\pa_{yy}b(t,y)(\pa_v-t\pa_z)F
\end{align*}
which gives that
\begin{align}
\pa_tF-\nu \left(\pa_{zz}F+ \left(\pa_yb(t, b^{-1}(t,v))\right)^2(\pa_v-t\pa_z)^2F\right)+\mathrm{Com}(t,z,v)=0. 
\end{align}
Taking the Fourier transform in $z$, we get
\begin{align}\label{eq-F-k}
  \pa_t\tilde F_k-\nu \left(\pa_{zz}\tilde F_k+ \left(b'_t\circ b_t^{-1}\right)^2(\pa_v-ikt)^2\tilde F_k\right)+\widetilde{\mathrm{Com}}_k(t,v)=0. 
\end{align}

\subsection{Time dependent Fourier multipliers}
To obtain the enhanced dissipation and (weak) inviscid damping, we introduce the following ghost type multiplier \cite{Alinhac2001} which is widely used in studying the asymptotic stability of Couette flow, see \cite{BGM2015,BVW2018}. 
For any $k\neq 0$, let 
\begin{align*}
  \mathring{\mathrm{A}}_k(t,\eta)=\left( \arctan \left(\frac{\eta-kt}{k}\right)+ \arctan\Big(\nu^{\frac13}|k|^{\frac23}\frac{\eta-kt}{k}\Big)+\pi +1\right),
\end{align*}
and
\begin{align*}
  \mathrm{A}_k(t,\eta)=e^{c_0\nu^{\f13}t}\mathring{\mathrm{A}}_k(t,\eta)\ln(e+|k|). 
\end{align*}
It is easy to check that 
\begin{align*}
  e^{c_0\nu^{\f13}t}\ln(e+|k|)\leq \mathrm{A}_k(t,\eta)\leq (2\pi +1)e^{c_0\nu^{\f13}t}\ln(e+|k|).
\end{align*}

Let $\mathrm{A}_k(t,\pa_v)\tilde F_k(v)=\mathcal{F}^{-1}_\eta(\mathrm{A}_k(t,\eta)\hat{F}_k(t,\eta))$, and $\mathrm{A}(t,\pa_z,\pa_v)F=\mathcal{F}^{-1}_{k,\eta}(\mathrm{A}_k(t,\eta)\hat{F}_k(t,\eta))$. In order to estimate the dissipation term $\mathrm{A}(t,\pa_z,\pa_v) \Big(\left(b'_t\circ b_t^{-1}\right)^2(\pa_v-t\pa_z)^2F\Big)$. We study the kernel of the Fourier multipliers. 
\begin{lemma}\label{lem-ker-arct}
  It holds that
  \begin{equation}\label{eq-inver-F1}
    \begin{aligned}    
      &\mathcal{F}^{-1}_{\eta}\Big(\arctan \left( \frac{\eta-kt}{k}\right)\hat{F}_k(t,\eta)\Big)(t,k,v)\\
      =&\frac{i\mathrm{sgn}(k)}{2}P.V.\int \frac{e^{-|k(v-v')|}e^{ikt(v-v')}}{v-v'}\tilde F_k(t,v')dv',
    \end{aligned}
  \end{equation}
  and 
  \begin{equation}\label{eq-inver-F2}
    \begin{aligned}    
      &\mathcal{F}^{-1}_{\eta}\Big(\arctan \left( \nu^{\frac13}|k|^{\frac23}\frac{\eta-kt}{k}\right)\hat{F}_k(t,\eta)\Big)(t,k,v)\\
      =&\frac{i\mathrm{sgn}(k)}{2}P.V.\int \frac{e^{-\nu^{-\frac13}|k|^{-\frac23}|k(v-v')|}e^{ikt(v-v')}}{v-v'}\tilde F_k(t,v')dv'.
    \end{aligned}
  \end{equation}
\end{lemma}
\begin{proof}
We introduce an auxiliary function $f(t,x,v)$ such that $f(t,x,v)=F(t,z,v)$ with $x=z+tv$. It holds that $\tilde{F}_k(t,v)=\tilde{f}_k(t,v)e^{iktv}$, $\hat{F}_k(t,\eta)=\hat{f}_k(t,\eta-kt)$ and 
\begin{align*}
\arctan \left(\frac{\eta-kt}{k}\right) \hat{F}_k(t,\eta)= \arctan \left(\frac{\xi}{k}\right) \hat{f}_k(t,\xi)\bigg|_{\xi=\eta-kt}.
\end{align*}
To prove \eqref{eq-inver-F1}, it suffice to give the inverse Fourier transform of $\arctan\big(\frac{\xi}{k}\big)\hat{f}_k(t,\xi)$.

Note that $\arctan \big(\frac{\xi}{k}\big)$, $\text{sgn}(\xi)\in \mathcal S'$.  We have for any  $h\in \mathcal S$ that
\begin{align*}
  \mathcal{F}^{-1}_{\xi}\Big(\arctan \big(\frac{\xi}{k}\big)\hat{h}(\xi)\Big)(v)  =&\mathcal{F}^{-1}_{\xi} \left(\arctan \big(\frac{\xi}{k}\big)\right)*\mathcal{F}^{-1}_{\xi}\left(\hat{h}(\xi)\right)(v)\\
  =&\frac{1}{2}\mathcal{F}^{-1}_{\xi} \left[ \Big(\frac{k}{k^2+(\cdot)^2}*\text{sgn}(\cdot)\Big)(\xi)\right]*\mathcal{F}^{-1}_{\xi}\left(\hat{h}(\xi)\right)(v)\\
  =&\pi\left(\mathcal{F}^{-1}_{\xi} \left( \frac{k}{k^2+\xi^2}\right)\mathcal{F}^{-1}_{\xi} \left(\text{sgn}(\xi)\right) \right)*\mathcal{F}^{-1}_{\xi}\left(\hat{h}(\xi)\right)(v)\\
  =&\frac{i\mathrm{sgn}(k)}{2}P.V.\int \frac{e^{-|k(v-v')|}}{v-v'}h(v')dv'.
\end{align*}
Here we use the fact that
\begin{align*}
  \frac{k}{k^2+{\xi}^2}=\frac{\mathrm{sgn}(k)}{2}\mathcal{F}_v \left(e^{-|kv|}\right)(\xi),\quad \text{sgn}(\xi)=\frac{i}{\pi}\mathcal{F}_v \left(\text{P.V.}\frac{1}{v}\right)(\xi). 
\end{align*}

Accordingly, we deduce that
\begin{align*}
  &\mathcal{F}^{-1}_{\eta}\Big(\arctan \left( \frac{\eta-kt}{k}\right)\hat{F}_k(t,\eta)\Big)(t,k,v)=e^{iktv}\mathcal{F}^{-1}_{\xi}\Big(\arctan \big(\frac{\xi}{k}\big)\hat{f}_k(t,\xi)\Big)(t,k,v)\\
  =&e^{iktv}\frac{i\mathrm{sgn}(k)}{2}P.V.\int \frac{e^{-|k(v-v')|}}{v-v'}\tilde f_k(t,v')dv'= \frac{i\mathrm{sgn}(k)}{2}\text{P.V.}\int \frac{e^{-|k(v-v')|}e^{ikt(v-v')}}{v-v'}\tilde F_k(t,v')dv',
\end{align*}
which is \eqref{eq-inver-F1}.

Note that $\arctan \big(\nu^{\frac13}|k|^{\frac23}\frac{\xi}{k}\big)=\Big(\frac{\nu^{\frac13}|k|^{\frac23}k}{k^2+\nu^{\frac23}|k|^{\frac43}(\cdot)^2}*\text{sgn}(\cdot)\Big)(\xi)$. Using the same approach, we can also obtain \eqref{eq-inver-F2}.
\end{proof}

\begin{lemma}\label{lem-A-nonlocal}
For $k\neq0$, it holds that
\begin{align*}
& \int \mathrm{A}_k(t,\pa_v)\left(b'_t\circ b_t^{-1}(v)\right)^2(\pa_v-ikt)^2\tilde F_k(t,v)\overline{\mathrm{A}_k(t,\pa_v)\tilde F_k(t,v)}dv\\
&\quad\quad\quad\quad\quad\quad\quad\quad
+\big\|\left(b'_t\circ b_t^{-1}(v)\right)\mathrm{A}_k(t,\pa_v)(\pa_v-ikt)\tilde F_k(t,v)\big\|_{L^2}^2 \\
\leq& C\|e^{c_0\nu^{\frac{1}{3}}t}\ln (e+|k|)(\pa_v-ikt)\tilde F_k(t,v)\|_{L^2}\|e^{c_0\nu^{\frac{1}{3}}t}\ln (e+|k|)\tilde F_k(t,v)\|_{L^2}.
\end{align*}
\end{lemma}
\begin{proof} 
By direct calculation, we can obtain
\begin{align*}
  &\int \mathrm{A}_k(t,\pa_v)\left(\left(b'_t\circ b_t^{-1}(v)\right)^2(\pa_v-ikt)^2\tilde F_k(t,v)\right)\overline{\mathrm{A}_k(t,\pa_v)\tilde F_k(t,v)} dv\\
  =&\int \mathrm{A}_k(t,\pa_v)(\pa_v-ikt)\left(\left(b'_t\circ b_t^{-1}(v)\right)^2(\pa_v-ikt)\tilde F_k(t,v)\right)\overline{\mathrm{A}_k(t,\pa_v)\tilde F_k(t,v)} dv\\
  &-\int \mathrm{A}_k(t,\pa_v) \left(\pa_v\left(b'_t\circ b_t^{-1}(v)\right)^2\right)(\pa_v-ikt)\tilde F_k(t,v) \overline{\mathrm{A}_k(t,\pa_v)\tilde F_k(t,v)} dv\\
  =&-\int \left(b'_t\circ b_t^{-1}(v)\right)^2\mathrm{A}_k(t,\pa_v)\left((\pa_v-ikt)\tilde F_k(t,v)\right)\overline{\mathrm{A}_k(t,\pa_v)(\pa_v-ikt)\tilde F_k(t,v)} dv\\
  &+\int  \left[\left(b'_t\circ b_t^{-1}(v)\right)^2,\mathrm{A}_k(t,\pa_v)\right] \left((\pa_v-ikt)\tilde F_k(t,v)\right)\overline{\mathrm{A}_k(t,\pa_v)(\pa_v-ikt)\tilde F_k(t,v)} dv\\
  &-\int \mathrm{A}_k(t,\pa_v) \left(\pa_v\left(b'_t\circ b_t^{-1}(v)\right)^2\right)(\pa_v-ikt)\tilde F_k(t,v) \overline{\mathrm{A}_k(t,\pa_v)\tilde F_k(t,v)} dv\\
  \eqdef&I+II+III.
\end{align*}
It is clear that
\begin{align*}
  I=-\left\|b'_t\circ b_t^{-1}(v)\mathrm{A}_k(t,\pa_v)(\pa_v-ikt)\tilde F_k(t,v)\right\|_{L^2}^2.
\end{align*}
Recall that $\mathring{\mathrm{A}}_k$ is a ghost type multiplier. By using Lemma \ref{lem-back}, one can get
\begin{align*}
  \left|III\right|\lesssim\left\|e^{c_0\nu^{\f13}t}\ln (e+|k|)(\pa_v-ikt)\tilde F_k(t,v)\right\|_{L^2}\left\|e^{c_0\nu^{\f13}t}\ln (e+|k|)\tilde F_k(t,v)\right\|_{L^2}.
\end{align*}
The most troublesome term is $II$. To  treat this term, we need to provide the estimate for the following commutator
\begin{align*}
  \left[\left(b'_t\circ b_t^{-1}(v)\right)^2,\mathrm{A}_k(t,\pa_v)\right]=e^{c_0\nu^{\f13}t} \ln(e+|k|)\left[\left(b'_t\circ b_t^{-1}(v)\right)^2,\mathring{\mathrm{A}}_k(t,\pa_v)\right].
\end{align*}

With the help of Lemma \ref{lem-ker-arct}, we have
\begin{align*}
  &\left[\left(b'_t\circ b_t^{-1}(v)\right)^2, \mathring{\mathrm{A}}_k(t,\pa_v)\right] (\pa_v-ikt)\tilde F_k(t,v)\\
  =&\frac{i\mathrm{sgn}(k)}{2}\int \frac{e^{-|k(v-v')|}e^{itk(v-v')}\left(b'_t\circ b_t^{-1}(v))-b'_t\circ b_t^{-1}(v'))\right)}{v-v'}(\pa_{v'}-ikt)\tilde F_k(t,v')dv'\\
  &+\frac{i\mathrm{sgn}(k)}{2}\int \frac{e^{-\nu^{-\frac13}|k|^{-\frac23}|k(v-v')|}e^{itk(v-v')}\left(b'_t\circ b_t^{-1}(v))-b'_t\circ b_t^{-1}(v'))\right)}{v-v'}(\pa_{v'}-ikt)\tilde F_k(t,v')dv'\\
  =&-\frac{i\mathrm{sgn}(k)}{2}\int \frac{|k|\mathrm{sgn}(v-v')e^{-|k(v-v')|}e^{itk(v-v')}\left(b'_t\circ b_t^{-1}(v))-b'_t\circ b_t^{-1}(v'))\right)}{v-v'}\tilde F_k(t,v')dv'\\
  &-\frac{i\mathrm{sgn}(k)}{2}\int e^{-|k(v-v')|}e^{itk(v-v')} \left(\pa_{v'}\frac{ \left(b'_t\circ b_t^{-1}(v))-b'_t\circ b_t^{-1}(v'))\right)}{v-v'}\right)\tilde F_k(t,v')dv'\\
  &-\frac{i\mathrm{sgn}(k)}{2}\int \nu^{-\frac13}|k|^{-\frac23}|k|\mathrm{sgn}(v-v')e^{-\nu^{-\frac13}|k|^{-\frac23}|k(v-v')|}e^{itk(v-v')}\\
  &\qquad\qquad\qquad\qquad\qquad\qquad\qquad\qquad\qquad\times \frac{\left(b'_t\circ b_t^{-1}(v))-b'_t\circ b_t^{-1}(v'))\right)}{v-v'}\tilde F_k(t,v')dv'\\
  &-\frac{i\mathrm{sgn}(k)}{2}\int e^{-\nu^{-\frac13}|k|^{-\frac23}|k(v-v')|}e^{itk(v-v')} \left(\pa_{v'}\frac{ \left(b'_t\circ b_t^{-1}(v))-b'_t\circ b_t^{-1}(v'))\right)}{v-v'}\right)\tilde F_k(t,v')dv'.
\end{align*}
By Lemma \ref{lem-back} we can see that
\begin{align*}
\left|\frac{ \left(b'_t\circ b_t^{-1}(v))-b'_t\circ b_t^{-1}(v'))\right)}{v-v'}\right|\le C\left\|b_{in}''\right\|_{H^1},\\
  \left|\pa_{v'}\frac{ \left(b'_t\circ b_t^{-1}(v))-b'_t\circ b_t^{-1}(v'))\right)}{v-v'}\right|\le C\left\|b_{in}''\right\|_{H^2},
\end{align*}
and then
\begin{align*}
  \left|\frac{|k|\mathrm{sgn}(v-v')e^{-|k(v-v')|}e^{itk(v-v')}\left(b'_t\circ b_t^{-1}(v))-b'_t\circ b_t^{-1}(v'))\right)}{v-v'}\tilde F_k(t,v')\right|\lesssim |k| e^{-|k(v-v')|}\left|\tilde F_k(t,v')\right|.
\end{align*}
Therefore,
\begin{align*}
  &\left\|\int \frac{|k|\mathrm{sgn}(v-v')e^{-|k(v-v')|}e^{itk(v-v')}\left(b'_t\circ b_t^{-1}(v))-b'_t\circ b_t^{-1}(v'))\right)}{v-v'}\tilde F_k(t,v')dv'\right\|_{L^2}\\
  \lesssim&\left\||k| e^{-|k(\cdot)|}*\tilde F_k(t,\cdot)\right\|_{L^2}\lesssim \left\||k| e^{-|k(v)|}\right\|_{L^1}\left\|\tilde F_k(t,v)\right\|_{L^2}\lesssim\left\|\tilde F_k(t,v)\right\|_{L^2}.
\end{align*}
Repeating this technique,  we obtain
\begin{align*}
  \left\|\left[\left(b'_t\circ b_t^{-1}(v)\right)^2,\mathring{\mathrm{A}}_k(t,\pa_v)\right] (\pa_v-ikt)\tilde F_k(t,v) \right\|_{L^2}\lesssim \left\|\tilde F_k(t,v)\right\|_{L^2}.
\end{align*}

It follows that
\begin{align*}
  \left|II\right|\lesssim&\left\|\left[\left(b'_t\circ b_t^{-1}(v)\right)^2,\mathrm{A}_k(t,\pa_v)\right] (\pa_v-ikt)\tilde F_k(t,v) \right\|_{L^2}\left\|\mathrm{A}_k(t,\pa_v)(\pa_v-ikt)\tilde F_k(t,v)\right\|_{L^2}\\
  \lesssim&\left\|e^{c_0\nu^{\f13}t} \ln(e+|k|)\left[\left(b'_t\circ b_t^{-1}(v)\right)^2,\mathring{\mathrm{A}}_k(t,\pa_v)\right] (\pa_v-ikt)\tilde F_k(t,v) \right\|_{L^2} \\
  &\qquad\qquad\qquad\qquad\qquad\qquad\qquad\qquad\times\left\|e^{c_0\nu^{\f13}t}\ln (e+|k|)(\pa_v-ikt)\tilde F_k(t,v)\right\|_{L^2}\\
  \lesssim&\left\|e^{c_0\nu^{\f13}t}\ln (e+|k|)\tilde F_k(t,v)\right\|_{L^2}\left\|e^{c_0\nu^{\f13}t}\ln (e+|k|)(\pa_v-ikt)\tilde F_k(t,v)\right\|_{L^2}.
\end{align*}
This gives the result of this lemma.
\end{proof}
\subsection{Energy estimate for $\tilde F_k$}
In this subsection, we give the energy estimate for \eqref{eq-F-k}.
\begin{lemma}
For $k\neq0$, $t\ge s\ge0$, their exists $c_1>0$ that is independent of $k$, $t$, and $s$ such that
\begin{align}\label{est-F-1}
  \big\|\mathrm{A}_k(t,\pa_v)\tilde F_k(t,v)\big\|_{L^2}^2\le e^{-c_1 \left(\nu^{\frac{1}{3}}|k|^{\frac{2}{3}}+\nu |k|^2\right)(t-s)}\big\|\mathrm{A}_k(s,\pa_v)\tilde F_k(s,v)\big\|_{L^2}^2.
\end{align}
It also holds that
\begin{align}\label{est-F-2}
  \int^t_s\nu\left\|\mathrm{A}_k(s',\pa_v)(\pa_v-iks')\tilde F_k(s',v)\right\|_{L^2}^2ds'\le\big\|\mathrm{A}_k(s,\pa_v)\tilde F_k(s,v)\big\|_{L^2}^2
\end{align} 
\begin{align}\label{est-F-3}
  \int^t_s\nu\left\|\mathrm{A}_k(s',\pa_v)|k|\tilde F_k(s',v)\right\|_{L^2}^2ds'\le\big\|\mathrm{A}_k(s,\pa_v)\tilde F_k(s,v)\big\|_{L^2}^2
\end{align}
\begin{align}\label{est-F-4}
  \int^t_s \left\|\mathrm{A}_k(s',\pa_v)k(-\Delta_{L})^{-\frac{1}{2}}\tilde F_k(s',v)\right\|_{L^2}^2ds'\le\big\|\mathrm{A}_k(s,\pa_v)\tilde F_k(s,v)\big\|_{L^2}^2,
\end{align}
where $\Delta_{L}=-k^2+(\pa_v-ikt)^2$.
\end{lemma}
\begin{proof}
We start with the estimate of the communicator term.  Recalling the relationship between the old and new coordinate systems, and using Lemma \ref{lem-back}, one can see that
  \begin{align*}
     \tilde F_k(t,v(t,y))e^{-itkv(t,y)}=\mathbb{D}_{t,k}[\tilde\om_k](t,y),
   \end{align*} 
and
\begin{align}\label{eq-rela-FDw}
  \left\|\tilde F_k(t,v)\right\|_{L^2_v}\approx \left\|\mathbb{D}_{t,k}[\tilde\om_k](t,y)\right\|_{L^2_y},\ \left\|\left(\pa_v-itk\right) \tilde F_k(t,v)\right\|_{L^2_v}\approx \left\|\pa_y\mathbb{D}_{t,k}[\tilde\om_k](t,y)\right\|_{L^2_y}.
\end{align}
Then we deduce from Lemma \ref{lem-wave-com} that
\begin{align*}
  \|\mathrm{A}_k(t,\pa_v)\widetilde{\mathrm{Com}}_k(t,v)\|_{L^2_v}\leq C\nu \left(\left\|\mathrm{A}_k(t,\pa_v)\tilde F_k(t,v)\right\|_{L^2_v}+\left\|\mathrm{A}_k(t,\pa_v)\left(\pa_v-itk\right) \tilde F_k(t,v)\right\|_{L^2_v} \right).
\end{align*}

Next, we use a standard energy method to prove the results of this lemma.

Recall the definition of $\mathrm{A}_k(t,\pa_v)$. By using \eqref{eq-F-k}, we have
\begin{align*}
  &\frac{1}{2}\frac{d}{dt}\int \mathrm{A}_k(t,\pa_v) \tilde F_k(t,v) \overline{\mathrm{A}_k(t,\pa_v)\tilde F_k(t,v)} dv\\
  =& \mathfrak R\int \mathrm{A}_k(t,\pa_v) \pa_t\tilde F_k(t,v) \overline{\mathrm{A}_k(t,\pa_v)\tilde F_k(t,v)} dv\\
  &+ \mathfrak R\int \left(\pa_t\mathrm{A}_k(t,\pa_v)\right)\tilde F_k(t,v) \overline{\mathrm{A}_k(t,\pa_v)\tilde F_k(t,v)} dv\\
  =&\nu\mathfrak R\int \mathrm{A}_k(t,\pa_v) \left( \left(-k^2+ \left(b'_t\circ b_t^{-1}(v)\right)^2(\pa_v-ikt)^2\right) \tilde F_k(t,v)\right) \overline{\mathrm{A}_k(t,\pa_v)\tilde F_k(t,v)} dv\\
  &-\mathfrak R\int \mathrm{A}_k(t,\pa_v) \widetilde{\mathrm{Com}}_k(t,v) \overline{\mathrm{A}_k(t,\pa_v)\tilde F_k(t,v)} dv\\
  &+c_0\nu^{\frac{1}{3}}\mathfrak R\int \mathrm{A}_k(t,\pa_v)   \tilde F_k(t,v) \overline{\mathrm{A}_k(t,\pa_v)\tilde F_k(t,v)} dv\\
  &+ \mathfrak R\int e^{c_0\nu^{\f13}t} \left(\pa_t\mathring{\mathrm{A}}_k(t,\pa_v)\right)\ln(e+|k|)   \tilde F_k(t,v) \overline{\mathrm{A}_k(t,\pa_v)\tilde F_k(t,v)} dv.
\end{align*}
For the last term, it holds that
\begin{align*}
  &\mathfrak R\int e^{c_0\nu^{\frac{1}{3}}t} \left(\pa_t\mathring{\mathrm{A}}_k(t,\pa_v)\right)\ln(e+|k|)   \tilde F_k(t,v) \overline{\mathrm{A}_k(t,\pa_v)\tilde F_k(t,v)} dv\\
  =&-\mathfrak R\int e^{c_0\nu^{\frac{1}{3}}t} \left(\frac{1}{1+\left(\frac{\eta}{k}-t\right)^2}+ \frac{\nu^{\frac{1}{3}}|k|^{\frac{2}{3}}}{1+\nu^{\frac{2}{3}}|k|^{\frac{4}{3}}\left(\frac{\eta}{k}-t\right)^2}\right)\ln(e+|k|)\hat{F}_k(t,\eta)\overline{\mathrm{A}_k(t,\eta)\hat{F}_k(t,\eta)} d\eta\\
  \le&-\left\| k(-\Delta_{L})^{-\frac{1}{2}}\mathrm{A}_k(t,\pa_v)\tilde F_k(t,v)\right\|_{L^2}^2-\left\|\frac{e^{c_0\nu^{\frac{1}{3}}t}\nu^{\frac{1}{6}}|k|^{\frac{1}{3}}\ln(e+|k|)}{\sqrt{1+(\nu k^2)^{\frac{2}{3}}\big(t-\frac{\eta}{k}\big)^2}}\hat{F}_k(t,\eta)\right\|_{L^2_\eta}^2.
\end{align*}

Then by using Lemma \ref{lem-A-nonlocal}, we have
\begin{align*}
  &\frac{1}{2}\frac{d}{dt}\big\|\mathrm{A}_k(t,\pa_v)\tilde F_k(t,v)\big\|_{L^2}^2\\
  &+\nu\left\|b'_t\circ b_t^{-1}(v)\mathrm{A}_k(t,\pa_v)(\pa_v-ikt)\tilde F_k(t,v)\right\|_{L^2}^2+\nu\left\|k\mathrm{A}_k(t,\pa_v)\tilde F_k(t,v)\right\|_{L^2}^2\\
  &+\left\| k(-\Delta_{L})^{-\frac{1}{2}}\mathrm{A}_k(t,\pa_v)\tilde F_k(t,v)\right\|_{L^2}^2+\left\|\frac{e^{c_0\nu^{\frac{1}{3}}t}\nu^{\frac{1}{6}}|k|^{\frac{1}{3}}\ln(e+|k|)}{\sqrt{1+(\nu k^2)^{\frac{2}{3}}\big(t-\frac{\eta}{k}\big)^2}}\hat{F}_k(t,\eta)\right\|_{L^2}^2\\
  \lesssim&c_0\nu^{\frac{1}{3}}\big\|\mathrm{A}_k(t,\pa_v)\tilde F_k(t,v)\big\|_{L^2}^2+\nu\|\mathrm{A}_k(t,\pa_v)(\pa_v-ikt)\tilde F_k(t,v)\|_{L^2}\|\mathrm{A}_k(t,\pa_v)\tilde F_k(t,v)\|_{L^2}\\
  &+\|\mathrm{A}_k(t,\pa_v) \mathrm{Com}(t,z,v)\|_{L^2}\|\mathrm{A}_k(t,\pa_v)\tilde F_k(t,v)\|_{L^2}\\
  \lesssim&c_0\nu^{\frac{1}{3}}\big\|\mathrm{A}_k(t,\pa_v)\tilde F_k(t,v)\big\|_{L^2}^2+\nu\|\mathrm{A}_k(t,\pa_v)(\pa_v-ikt)\tilde F_k(t,v)\|_{L^2}\|\mathrm{A}_k(t,\pa_v)\tilde F_k(t,v)\|_{L^2}\\
  &+\nu\|\mathrm{A}_k(t,\pa_v)\tilde F_k(t,v)\|_{L^2}^2.
\end{align*}

It follows from the fact
\begin{equation}
  \left\{
    \begin{array}{ll}
      \frac{ \nu^{\frac{1}{3}}|k|^{\frac{2}{3}} }{ 1+(\nu k^2)^{\frac{2}{3}}\big(t-\frac{\eta}{k}\big)^2}\ge \frac{1}{2}\nu^{\frac{1}{3}}|k|^{\frac{2}{3}}&\text{ for }(\nu k^2)^{\frac{2}{3}}\big(t-\frac{\eta}{k}\big)^2\le 1;\\
      \nu (\eta-kt)^2 \ge\nu^{\frac{1}{3}}|k|^{\frac{2}{3}}&\text{ for }(\nu k^2)^{\frac{2}{3}}\big(t-\frac{\eta}{k}\big)^2\ge 1,
    \end{array}
  \right.
\end{equation}
that 
\begin{align*}
  \nu (\eta-kt)^2+\frac{ \nu^{\frac{1}{3}}|k|^{\frac{2}{3}} }{ 1+(\nu k^2)^{\frac{2}{3}}\big(t-\frac{\eta}{k}\big)^2}\ge \frac{1}{2}\left(\nu|k|^2\right)^{\frac{1}{3}},
\end{align*}
therefore 
\begin{align*}
  &\nu\left\|b'_t\circ b_t^{-1}(v)\mathrm{A}_k(t,\pa_v)(\pa_v-ikt)\tilde F_k(t,v)\right\|_{L^2}^2+\left\|\frac{e^{c_0\nu^{\frac{1}{3}}t}\nu^{\frac{1}{6}}|k|^{\frac{1}{3}}\ln(e+|k|)}{\sqrt{1+(\nu k^2)^{\frac{2}{3}}\big(t-\frac{\eta}{k}\big)^2}}\widehat{F}(t,k,\eta)\right\|_{L^2}^2\\
  \ge& \frac{1}{C}\big\|\left(\nu|k|^2\right)^{\frac{1}{6}}\mathrm{A}_k(t,\pa_v)\tilde F_k(t,v)\big\|_{L^2}^2.
\end{align*}

This implies that 
\begin{align*}
  &\frac{3}{4}\frac{d}{dt}\big\|\mathrm{A}_k(t,\pa_v)\tilde F_k(t,v)\big\|_{L^2}^2+\frac{1}{C}\big\|\left(\nu|k|^2\right)^{\frac{1}{6}}\mathrm{A}_k(t,\pa_v)\tilde F_k(t,v)\big\|_{L^2}^2 \\
  &+\nu\left\|b'_t\circ b_t^{-1}(v)\mathrm{A}_k(t,\pa_v)(\pa_v-ikt)\tilde F_k(t,v)\right\|_{L^2}^2+\nu\left\|\mathrm{A}_k(t,\pa_v)|k|\tilde F_k(t,v)\right\|_{L^2}^2\\
  &+\left\| k(-\Delta_{L})^{-\frac{1}{2}}\mathrm{A}_k(t,\pa_v)\tilde F_k(t,v)\right\|_{L^2}^2+\left\|\frac{e^{c_0\nu^{\frac{1}{3}}t}\nu^{\frac{1}{6}}|k|^{\frac{1}{3}}\ln(e+|k|)}{\sqrt{1+(\nu k^2)^{\frac{2}{3}}\big(t-\frac{\eta}{k}\big)^2}}\widehat{F}(t,k,\eta)\right\|_{L^2}^2\\
  \lesssim&c_0\nu^{\frac{1}{3}}\big\|\mathrm{A}_k(t,\pa_v)F\big\|_{L^2}^2+\nu^{\frac{1}{3}} \left(\nu^{\frac{1}{2}}\|\mathrm{A}_k(t,\pa_v)(\pa_v-ikt)\tilde F_k(t,v)\|_{L^2}\right) \|\nu^{\frac{1}{6}}\mathrm{A}_k(t,\pa_v)\tilde F_k(t,v)\|_{L^2} \\
  &+\nu^{\frac{2}{3}}\|\nu^{\frac{1}{6}}\mathrm{A}_k(t,\pa_v)\tilde F_k(t,v)\|_{L^2}^2.
\end{align*}
Then by taking $\nu$ and $c_0$ small enough, we deduce that
\begin{align*}
  &\frac{d}{dt}\big\|\mathrm{A}_k(t,\pa_v)\tilde F_k(t,v)\big\|_{L^2}^2+\frac{1}{C}\big\|\left(\nu|k|^2\right)^{\frac{1}{6}}\mathrm{A}_k(t,\pa_v)\tilde F_k(t,v)\big\|_{L^2}^2 \\
  &+\nu\left\|b'_t\circ b_t^{-1}(v)\mathrm{A}_k(t,\pa_v)(\pa_v-ikt)\tilde F_k(t,v)\right\|_{L^2}^2+\nu\left\|\mathrm{A}_k(t,\pa_v)|k|\tilde F_k(t,v)\right\|_{L^2}^2\\
  &+\left\| k(-\Delta_{L})^{-\frac{1}{2}}\mathrm{A}_k(t,\pa_v)\tilde F_k(t,v)\right\|_{L^2}^2+\left\|\frac{e^{c_0\nu^{\frac{1}{3}}t}\nu^{\frac{1}{6}}|k|^{\frac{1}{3}}\ln(e+|k|)}{\sqrt{1+(\nu k^2)^{\frac{2}{3}}\big(t-\frac{\eta}{k}\big)^2}}\widehat{F}(t,k,\eta)\right\|_{L^2}^2 \le 0.
\end{align*}
Through the above energy dissipation inequality, we obtain \eqref{est-F-1}, \eqref{est-F-2}, \eqref{est-F-3}, and \eqref{est-F-4} immediately.
\end{proof}

\subsection{Estimates for $S(t,s)$}
In this subsection, we use the energy estimates $\tilde F_k(t,v)$ to derive estimates for $\omega$, the solution of \eqref{eq: linearNS}, and give the proof of Proposition \ref{prop-semi}. 
\begin{proof}[Proof of Proposition \ref{prop-semi}]
It follows from the assumption $\int_{\mathbb T} \om(s,x,y) d x=0$ that $\int_{\mathbb T} \om(t,x,y) d x=0$, and $\om(t,x,y)=\om_{\neq}(t,x,y)$.

Recalling \eqref{eq-rela-FDw}, we deduce from \eqref{eq-est-wave} and \eqref{est-F-1} that
\begin{align*}
  &\left\|e^{c_0\nu^{\frac{1}{3}}t}{\om}(t,x,y)\right\|_{H^{log}_xL^2_{y}}^2
  \lesssim \sum_{k\neq0}\left\|e^{c_0\nu^{\frac{1}{3}}t}\ln(e+|k|)\mathbb{D}_{t,k}[\tilde\om_k](t,y)\right\|_{L^2_{y}}^2\\
  \lesssim& \sum_{k\neq0}\left\|e^{c_0\nu^{\frac{1}{3}}t}\ln(e+|k|)\tilde F_k(t,v)\right\|_{L^2_{v}}^2
  \lesssim \sum_{k\neq0} \big\|e^{c_0\nu^{\frac{1}{3}}s}\ln(e+|k|)\tilde F_k(s,v)\big\|_{L^2_v}^2\\
  \lesssim&  \sum_{k\neq0}\big\|e^{c_0\nu^{\frac{1}{3}}s}\ln(e+|k|)\tilde\om_k(s,y)\big\|_{L^2_y}^2\lesssim  \left\|e^{c_0\nu^{\frac{1}{3}}s}{\om}(s,x,y)\right\|_{H^{log}_xL^2_{y}}^2,
\end{align*}
which is \eqref{btsp:om-point}.

Similarly, from \eqref{eq-wave-com-y}, \eqref{est-F-2}, and \eqref{est-F-3} we obtain
\begin{align*}
   &\int_{s}^t\|e^{c_0\nu^{\frac{1}{3}}s'}\pa_x{\om}(s',x,y)\|_{H^{log}_xL^2_{y}}^2ds' \lesssim \sum_{k\neq0}\int_{s}^t\|e^{c_0\nu^{\frac{1}{3}}s'}|k|\ln(e+|k|) \tilde F_k(s',v)\|_{L^2_{v}}^2ds'\\
   \lesssim&\nu^{-1}\sum_{k\neq0} \|e^{c_0\nu^{\frac{1}{3}}s} \ln(e+|k|) \tilde F_k(s,v)\|_{L^2_{v}}^2 \lesssim\nu^{-1}\left\|e^{c_0\nu^{\frac{1}{3}}s}{\om}(s,x,y)\right\|_{H^{log}_xL^2_{y}}^2,
\end{align*}
and
\begin{align*}
   &\int_{s}^t\|e^{c_0\nu^{\frac{1}{3}}s'}\pa_y{\om}(s')\|_{H^{log}_xL^2_{y}}^2ds' \lesssim \sum_{k\neq0}\int_{s}^t\|e^{c_0\nu^{\frac{1}{3}}s'}\ln(e+|k|)\pa_y\tilde{\om}_{k}(s')\|_{L^2_{y}}^2ds'\\
   \lesssim&\sum_{k\neq0}\int_{s}^t\|e^{c_0\nu^{\frac{1}{3}}s'}\ln(e+|k|)\pa_y\mathbb{D}_{s',k}[\tilde\om_k](s',y)\|_{L^2_{y}}^2+\|e^{c_0\nu^{\frac{1}{3}}s'}\ln(e+|k|)\mathbb{D}_{s',k}[\tilde\om_k](s',y)\|_{L^2_{y}}^2ds'\\
   \lesssim&\sum_{k\neq0}\int_{s}^t\|e^{c_0\nu^{\frac{1}{3}}s'}\ln(e+|k|)(\pa_v-iks') \tilde F_k(s',v)\|_{L^2_{v}}^2+\|e^{c_0\nu^{\frac{1}{3}}s'}\ln(e+|k|) \tilde F_k(s',v)\|_{L^2_{v}}^2ds'\\
   \lesssim&\nu^{-1}\sum_{k\neq0} \|e^{c_0\nu^{\frac{1}{3}}s}\ln(e+|k|) \tilde F_k(s,v)\|_{L^2_{v}}^2 \lesssim\nu^{-1}\left\|e^{c_0\nu^{\frac{1}{3}}s}{\om}(s,x,y)\right\|_{H^{log}_xL^2_{y}}^2,
\end{align*}
which is \eqref{btsp:pa_xyomL2L2L2}.

By using \eqref{est-F-3} and \eqref{est-F-1}, we have
\begin{align*}
  &\int_{s}^t\|e^{c_0\nu^{\frac{1}{3}}s'}\pa_x{\om}(s')\|_{H^{log}_xL^2_{y}}ds'\\
  \lesssim&\int_{s}^t \left(\sum_{k\neq0}\|e^{c_0\nu^{\frac{1}{3}}s'}|k|\ln(e+|k|)\mathbb{D}_{s',k}[\tilde\om_k](s',y)\|_{L^2_{y}}^2 \right)^{\frac{1}{2}}ds'\\
  \lesssim& \left(\int_{s}^{s+1} \sum_{k\neq0}\|e^{c_0\nu^{\frac{1}{3}}s'}|k|\ln(e+|k|)\tilde F_k(s',v)\|_{L^2_{y}}^2 ds'\right)^{\frac{1}{2}}\\
  &+\int_{s+1}^t \left(\sum_{k\neq0}\|e^{c_0\nu^{\frac{1}{3}}s'}|k|\ln(e+|k|)\tilde F_k(s',v)\|_{L^2_{y}}^2\right)^{\frac{1}{2}} ds'\\
  \lesssim& \left(\int_{s}^{s+1} \sum_{k\neq0}\|e^{c_0\nu^{\frac{1}{3}}s'}|k|\ln(e+|k|)\tilde F_k(s',v)\|_{L^2_{y}}^2 ds'\right)^{\frac{1}{2}}\\
  &+\int_{s+1}^t \left(\sum_{k\neq0}\frac{\nu |k|^2(s'-s)^3e^{-\nu^{\frac{1}{3}}|k|^{\frac{2}{3}}(s'-s)}}{\nu |k|^2(s'-s)^3} |k|^2\|e^{c_0\nu^{\frac{1}{3}}s}\ln(e+|k|)\tilde F_k(s,v)\|_{L^2_{y}}^2\right)^{\frac{1}{2}} ds'\\
  \lesssim&\nu^{-\frac{1}{2}} \left(\sum_{k\neq0} \|e^{c_0\nu^{\frac{1}{3}}s} \ln(e+|k|) \tilde F_k(s,v)\|_{L^2_{v}}^2\right)^{\frac{1}{2}} \lesssim\nu^{-\frac{1}{2}}\left\|e^{c_0\nu^{\frac{1}{3}}s}{\om}(s,x,y)\right\|_{H^{log}_xL^2_{y}},
\end{align*}
which is \eqref{btsp:pa_xomL1L2L2}.

By using the Gagliardo–Nirenberg interpolation inequality, we deduce from \eqref{btsp:pa_xyomL2L2L2} that 
\begin{align*}
   &\int_{s}^t\| e^{c_0\nu^{\frac{1}{3}}s'}{\om}(s')\|_{L^{\infty}_{x,y}}^2ds'\lesssim\int_{s}^t\sum_{k\neq0}\|e^{c_0\nu^{\frac{1}{3}}s'}|k|^{\frac{1}{2}}\ln(e+|k|) \tilde{\om}_k(s')\|_{L^\infty_y}^2ds'\\
   \lesssim&\int_{s}^t\sum_{k\neq0}\|e^{c_0\nu^{\frac{1}{3}}s'}|k| \ln(e+|k|) \tilde{\om}_k(s')\|_{L^2_y}\|e^{c_0\nu^{\frac{1}{3}}s'}\ln(e+|k|) \pa_y\tilde{\om}_k(s')\|_{L^2_y}ds'\\
   \lesssim&\left(\int_{s}^t\sum_{k\neq0}\|e^{c_0\nu^{\frac{1}{3}}s'}|k| \ln(e+|k|) \tilde{\om}_k(s')\|_{L^2_y}^2\right)^{\frac{1}{2}}\left(\int_{s}^t\sum_{k\neq0} \|e^{c_0\nu^{\frac{1}{3}}s'}\ln(e+|k|) \pa_y\tilde{\om}_k(s')\|_{L^2_y}^2ds'\right)^{\frac{1}{2}}\\
   \lesssim&\nu^{-1}\left\|e^{c_0\nu^{\frac{1}{3}}s}{\om}(s,x,y)\right\|_{H^{log}_xL^2_{y}}^2,
\end{align*}
which is \eqref{btsp:omL2L1Linfty}.

Recall that $\psi=\Delta^{-1}\om$, $\left(u^{(1)},u^{(2)}\right)=(-\pa_y\psi,\pa_x\psi)$, and
\begin{align*}
  \|(\pa_{xy},\pa_{xx})\Delta^{-1}\om(t)\|_{H^{log}_xL^2_{y}}^2=\|(\pa_{x}u^{(1)},\pa_{x}u^{(2)})(t)\|_{H^{log}_xL^2_{y}}^2.
\end{align*}
By the duality argument, we deduce from \eqref{eq-wave-dual} that
\begin{align*}
  &\|(\pa_{x}u^{(1)},\pa_{x}u^{(2)})(t)\|_{H^{log}_xL^2_{y}}^2\\
  =&\left|\left\langle \ln(e+|D_x|)\pa_{x}\om(t),\ln(e+|D_x|)\pa_{x}\psi(t) \right\rangle_{x,y}\right|\\
  =&\left|\left\langle \ln(e+|D_x|)\pa_{x}\mathbb{D}_t[\om](t),\ln(e+|D_x|)\pa_{x}\mathbb{D}_t^1[\psi](t) \right\rangle_{x,y}\right|\\
  =&\left|\left\langle \ln(e+|D_x|)\pa_{x}\mathbb{D}_t[\om](t),\left(\pa_vb^{-1}\right)\ln(e+|D_x|)\pa_{x}\mathbb{D}_t^1[\psi](t) \right\rangle_{x,v}\right|\\
  =&\left|\left\langle \ln(e+|D_x|)\pa_{x}(-\Delta_{x,v})^{-\frac{1}{2}}\mathbb{D}_t[\om](t),(-\Delta_{x,v})^{\frac{1}{2}}\left(\left(\pa_vb^{-1}\right)\ln(e+|D_x|)\pa_{x}\mathbb{D}_t^1[\psi](t)\right) \right\rangle_{x,v}\right|\\
  \lesssim&\left\|\pa_{x}(-\Delta_{x,v})^{-\frac{1}{2}}\mathbb{D}_t[\om](t)\right\|_{H^{log}_xL^2_{v}}\left\|\left(\pa_x,\pa_v\right) \left(\left(\pa_vb^{-1}\right)\ln(e+|D_x|)\pa_{x}\mathbb{D}_t^1[\psi](t)\right)\right\|_{L^2_{x,v}}.
\end{align*}
Note that $\int_{\mathbb T} \psi(x,y) d x=0$. It follows from \eqref{eq-est-wave}, \eqref{eq-wave-com-y}, and Lemma \ref{lem-back} that
\begin{align*}
  &\left\|\left(\pa_x,\pa_v\right) \left(\left(\pa_vb^{-1}\right)\ln(e+|D_x|)\pa_{x}\mathbb{D}_t^1[\psi](t)\right)\right\|_{L^2_{x,v}}\\
  \lesssim&\left\|\left(\pa_x,\pa_y\right) \ln(e+|D_x|)\pa_{x}\mathbb{D}_t^1[\psi](t) \right\|_{L^2_{x,y}}+\left\| \ln(e+|D_x|)\pa_{x}\mathbb{D}_t^1[\psi](t) \right\|_{L^2_{x,y}}\\
  \lesssim&\left\| \ln(e+|D_x|)\pa_{x}\mathbb{D}_t^1[\pa_x\psi](t) \right\|_{L^2_{x,y}}+\left\| \ln(e+|D_x|)\pa_{x}\mathbb{D}_t^1[\pa_y\psi](t) \right\|_{L^2_{x,y}}\\
  &+\left\| \ln(e+|D_x|)\pa_{x}\mathbb{D}_t^1[\psi](t) \right\|_{L^2_{x,y}}\\
  \lesssim&\left\| \pa_{xx}\psi(t) \right\|_{H^{log}_xL^2_{y}}+\left\| \pa_{xy}\psi(t) \right\|_{H^{log}_xL^2_{y}}\lesssim\| (\pa_{x}u^{(1)},\pa_{x}u^{(2)})(t)\|_{H^{log}_xL^2_{y}}.
\end{align*}
Therefore, 
\begin{align*}
  \|(\pa_{x}u^{(1)},\pa_{x}u^{(2)})(t)\|_{H^{log}_xL^2_{y}}\lesssim\left\|\pa_{x}(-\Delta_{x,v})^{-\frac{1}{2}}\mathbb{D}_t[\om](t)\right\|_{H^{log}_xL^2_{v}}.
\end{align*}
Then by using \eqref{est-F-4}, we have
\begin{equation}\label{eq-est-L2L2L2-strong}
  \begin{aligned}    
    &\int_{s}^t\|e^{c_0\nu^{\frac{1}{3}}s'}\left(\pa_{xy},\pa_{xx}\right)\Delta^{-1}\om(s')\|_{H^{log}_xL^2_{y}}^2ds'= \int_{s}^t\|e^{c_0\nu^{\frac{1}{3}}s'}(\pa_{x}u^{(1)},\pa_{x}u^{(2)})(s')\|_{H^{log}_xL^2_{y}}^2ds'\\
  \lesssim&\int_{s}^t\left\|e^{c_0\nu^{\frac{1}{3}}s'}\pa_{x}(-\Delta_{x,v})^{-\frac{1}{2}}\mathbb{D}_t[\om](s')\right\|_{H^{log}_xL^2_{v}}^2ds'\\
  \lesssim&\sum_{k\neq0}\int_{s}^t\left\|e^{c_0\nu^{\frac{1}{3}}s'}\ln(e+|k|)|k|\Delta^{-\frac{1}{2}}_{L}\tilde F_k(s',v)\right\|_{L^2_{v}}^2ds'\lesssim\left\|e^{c_0\nu^{\frac{1}{3}}s}{\om}(s,x,y)\right\|_{H^{log}_xL^2_{y}}^2,
  \end{aligned}
\end{equation}
this gives \eqref{btsp:V_1L2L2L2}.

By using the Gagliardo–Nirenberg interpolation inequality, we deduce from \eqref{eq-est-L2L2L2-strong} that 
\begin{align*}
  &\int_{s}^{t}\|e^{c_0\nu^{\frac{1}{3}}s'}|D_x|^{\frac{1}{2}}\pa_x\Delta^{-1}\om(s')\|_{H^{log}_xL^\infty_y}^2ds'\\
  \lesssim&\int_{s}^{t}\|e^{c_0\nu^{\frac{1}{3}}s'}\pa_{xx}\Delta^{-1}\om(s')\|_{H^{log}_xL^2_y}\|e^{c_0\nu^{\frac{1}{3}}s'}\pa_{xy}\Delta^{-1}\om(s')\|_{H^{log}_xL^2_y}ds'\\
  \lesssim& \left(\int_{s}^{t}\|e^{c_0\nu^{\frac{1}{3}}s'}\pa_{xx}\Delta^{-1}\om(s')\|_{H^{log}_xL^2_y}^2 ds'\right)^{\frac{1}{2}}\left(\int_{s}^{t} \|e^{c_0\nu^{\frac{1}{3}}s'}\pa_{xy}\Delta^{-1}\om(s')\|_{H^{log}_xL^2_y}^2ds'\right)^{\frac{1}{2}}\\
  \lesssim&\left\|e^{c_0\nu^{\frac{1}{3}}s}{\om}(s,x,y)\right\|_{H^{log}_xL^2_{y}}^2,
\end{align*}
this is \eqref{btsp:V_2L2L2Linfty}. 

Using the Sobolev inequality, we deduce from \eqref{btsp:V_2L2L2Linfty} 
\begin{align*}
  &\int_{s}^t\|e^{c_0\nu^{\frac{1}{3}}s'}\pa_x\Delta^{-1}\om (s')\|_{L^{\infty}_{x,y}}^2ds'\\
  \lesssim&\int_{0}^{t}\|e^{c_0\nu^{\frac{1}{3}}s'}|D_x|^{\frac{1}{2}}\pa_x\Delta^{-1}\om(s')\|_{H^{log}_xL^\infty_y}^2ds'\lesssim\left\|e^{c_0\nu^{\frac{1}{3}}s}{\om}(s,x,y)\right\|_{H^{log}_xL^2_{y}}^2,
\end{align*}
which is \eqref{btsp:V_2L2LinftyLinfty}.

By using the Gagliardo–Nirenberg interpolation inequality, we deduce from \eqref{est-F-1} that 
\begin{align*}
  &\sup_{s'\in[s,t]}\|e^{c_0\nu^{\frac{1}{3}}s'}\pa_{y}\Delta^{-1}\om (s')\|_{L^{\infty}_{x,y}}^2\lesssim\sup_{s'\in[s,t]}\|e^{c_0\nu^{\frac{1}{3}}s'}|D_x|^{\frac{1}{2}}\pa_y\Delta^{-1}\om (s')\|_{H^{log}_xL^\infty_y}^2\\
  \lesssim&\sup_{s'\in[s,t]}\|e^{c_0\nu^{\frac{1}{3}}s'}|D_x|\pa_y\Delta^{-1}\om (s')\|_{H^{log}_xL^2_y}\|e^{c_0\nu^{\frac{1}{3}}s'}\pa_{yy}\Delta^{-1}\om (s')\|_{H^{log}_xL^2_y}\\
  \lesssim&\sup_{s'\in[s,t]} \|e^{c_0\nu^{\frac{1}{3}}s'} \om (s')\|_{H^{log}_xL_{y}^2}^2\lesssim \|e^{c_0\nu^{\frac{1}{3}}s}{\om}(s)\|_{H^{log}_xL^2_{y}},
\end{align*}
which is \eqref{btsp:V_1LinftyL1Linfty}. 

This finishes the proof of Proposition \ref{prop-semi}.  
\end{proof}

\section{Nonlinear enhanced dissipation and inviscid damping}\label{sec-nonlin}
In this section, we study the nonlinear system and give the proof of the main result.

We now consider the nonlinear equation,
\begin{equation}
  \left\{
    \begin{array}{l}
      \pa_t\om_{\neq}+b(t,y)\pa_x\om_{\neq}-b''(t,y)\pa_x\Delta^{-1}\om_{\neq}-\nu\Delta\om_{\neq}=-\mathcal N_1-\mathcal N_3-\mathcal N_3,\\
      \omega_{\neq}|_{t=0}=P_{\neq}\omega_{in}(x,y),
    \end{array}
  \right.
\end{equation}
with
\begin{align*}
  \mathcal{N}^{1}=\left(u_{\neq}^{(1)} \partial_x \omega_{\neq}\right)_{\neq}+\left(u_{\neq}^{(2)} \partial_y \omega_{\neq}\right)_{\neq},\ \mathcal{N}^{2}=u_0^{(1)} \partial_x \omega_{\neq},\  \mathcal{N}^{3}=u_{\neq}^{(2)} \partial_y \omega_0,
\end{align*}
where $\omega_0(t,y)$ is the $0$ mode of vorticity which satisfies
\begin{equation*}
  \left\{
    \begin{array}{l}
      \pa_t\omega_0-\nu\pa_y^2\omega_0=-\big(u^{(1)}_{\neq}\pa_x\omega_{\neq}\big)_0-\big(u^{(2)}_{\neq}\pa_y\omega_{\neq}\big)_0,\\
      \omega_0|_{t=0}=P_{0}\omega_{in}(x,y),
    \end{array}
  \right.
\end{equation*}
and $u^{(1)}_0$ is the $0$ mode of horizontal velocity which satisfies
\begin{equation*}
  \left\{
    \begin{array}{l}
      \pa_t u^{(1)}_0-\nu\pa_y^2u^{(1)}_0=-\big(u^{(1)}_{\neq}\pa_xu^{(1)}_{\neq}\big)_0-\big(u^{(2)}_{\neq}\pa_yu^{(1)}_{\neq}\big)_0,\\
      u^{(1)}_0|_{t=0}=P_{0}u^{(1)}_{in}(y).
    \end{array}
  \right.
\end{equation*}
Then we have
\begin{align}\label{eq:DH-1}
  \omega_{\neq}(t,x,y)=S(t,0)\omega_{in,\neq}(x,y)-\int^t_0 S(t,s)(\mathcal N^{1}+\mathcal N^{2}+\mathcal N^{3})(s,x,y)ds,
\end{align}
\begin{align}\label{eq:DH-2}
  \omega_0(t,y)=e^{t\nu\pa_y^2}\omega_{in,0}(x,y)-\int^t_0 e^{(t-s)\nu\pa_y^2}\Big(\big(u^{(1)}_{\neq}\pa_x\omega_{\neq}\big)_0+\big(u^{(2)}_{\neq}\pa_y\omega_{\neq}\big)_0\Big)(s,y) ds,
\end{align}
and
\begin{align}\label{eq:DH-3}
  u^{(1)}_0(t,y)=e^{t\nu\pa_y^2}u^{(1)}_{in,0}(x,y)-\int^t_0 e^{(t-s)\nu\pa_y^2}\Big(\big(u^{(1)}_{\neq}\pa_xu^{(1)}_{\neq}\big)_0+\big(u^{(2)}_{\neq}\pa_yu^{(1)}_{\neq}\big)_0\Big)(s,y) ds.
\end{align}
Here $S(t,s)$ is the solution operator of the linear equation \eqref{eq: linearNS}.

Suppose $\left\|\omega_{i n}\right\|_{H^{log}_xL_{y}^2}+\left\|u_{i n}\right\|_{L_{x, y}^2} \leq \varepsilon_0 \nu^\beta$ and for some $T>0$, the following inequalities hold:

1. Uniform bound of $u^{(1)}_0$ and $\omega_0$:
\begin{align}\label{eq-assu-1}
  \sup_{t\in[0,T]}\left\|u^{(1)}_0(t)\right\|_{L^2_y}\le 8C_1\varepsilon_0 \nu^\beta;
\end{align}

\begin{align}\label{eq-assu-2}
  \sup_{t\in[0,T]}\left\|u^{(1)}_0(t)\right\|_{L^\infty_y}\le 8C_1\varepsilon_0 \nu^\beta;
\end{align}

\begin{align}\label{eq-assu-3}
  \left(\int^T_0\left\|\pa_y\omega_0(t)\right\|_{L^2_y}^2dt\right)^{\frac{1}{2}}\le 8C_1\nu^{-\frac{1}{2}}\left\| \omega_{i n}\right\|_{H^{log}_xL_{y}^2};
\end{align}
2. Enhanced dissipation:
\begin{align}\label{eq-assu-4}
    \sup_{t\in[0,T]}\left\|e^{c_0\nu^{\frac{1}{3}}t} {\om}_{\neq}(t,x,y)\right\|_{H^{log}_xL^2_{y}}\le 8C_1\left\| \omega_{i n}\right\|_{H^{log}_xL_{y}^2};
\end{align}
\begin{align}\label{eq-assu-5}
     \left(\int_{0}^T\|e^{c_0\nu^{\frac{1}{3}}t} \nabla{\om}_{\neq}(t,x,y)\|_{H^{log}_xL^2_{y}}^2dt\right)^{\frac{1}{2}}\le 8C_1\nu^{-\frac{1}{2}}\left\| \omega_{i n}\right\|_{H^{log}_xL_{y}^2};
\end{align}
\begin{align}\label{eq-assu-6}
     \int_{0}^T\|e^{c_0\nu^{\frac{1}{3}}t}\pa_x{\om}_{\neq}(t,x,y)\|_{H^{log}_xL^2_{y}}dt \le 8C_1\nu^{-\frac{1}{2}}\left\| \omega_{i n}\right\|_{H^{log}_xL_{y}^2};
\end{align}
\begin{align}\label{eq-assu-7}
     \left(\int_{0}^T\|e^{c_0\nu^{\frac{1}{3}}t} {\om}_{\neq}(t,x,y)\|_{L^\infty_{x,y}}^2dt\right)^{\frac{1}{2}}\le 8C_1\nu^{-\frac{1}{2}}\left\| \omega_{i n}\right\|_{H^{log}_xL_{y}^2};
\end{align}

3. Inviscid damping:
\begin{align}\label{eq-assu-8}
     \left(\int_{0}^T\|e^{c_0\nu^{\frac{1}{3}}t}u^{(2)}_{\neq}(t,x,y)\|_{L^\infty_{x,y}}^2dt\right)^{\frac{1}{2}}\le 8C_1 \left\| \omega_{i n}\right\|_{H^{log}_xL_{y}^2};
\end{align}
\begin{align}\label{eq-assu-9}
     \left(\int_{0}^T\| e^{c_0\nu^{\frac{1}{3}}t}\left|D_x\right|^{\frac{1}{2}} u^{(2)}_{\neq}(t,x,y)\|_{H^{log}_xL^\infty_{y}}^2dt\right)^{\frac{1}{2}}\le 8C_1 \varepsilon_0 \nu^\beta;
\end{align}
\begin{align}\label{eq-assu-10}
     \left(\int_{0}^T\|e^{c_0\nu^{\frac{1}{3}}t} \pa_xu^{(1)}_{\neq}(t,x,y)\|_{H^{log}_xL^2_{y}}^2dt\right)^{\frac{1}{2}}\le 8C_1 \left\| \omega_{i n}\right\|_{H^{log}_xL_{y}^2};
\end{align}

4. Uniform bound of $u^{(1)}_{\neq}$:
\begin{align}\label{eq-assu-11}
     \sup_{t\in[0,T]}\|e^{c_0\nu^{\frac{1}{3}}t}u^{(1)}_{\neq}(t)\|_{L^{\infty}_{x,y}}\le 8C_1 \left\| \omega_{i n}\right\|_{H^{log}_xL_{y}^2}.
\end{align}

The constants $C_1$ and $\varepsilon_0$ will be determined later. 

\begin{proposition}\label{prop-boot}
  Let $\beta \geq 1 / 2$. Assume that $\left\|\omega_{i n}\right\|_{H_x^{\log } L_y^2}+\left\|u_{i n}\right\|_{L_{x, y}^2} \leq \varepsilon_0 \nu^\beta$ and that for some $T>0$, the estimate \eqref{eq-assu-1}-\eqref{eq-assu-11} hold. Then there exists $\nu_0$ so that for $\nu<\nu_0$ and $\varepsilon_0$ sufficiently small depending only on $C_1$ (in particular, independent of $T$ ), these same estimates hold with all the occurrences of $8$ on the right-hand side replaced by $4$.
\end{proposition} 

This proposition implies Theorem \ref{thm-main} by the standard bootstrap argument. Before presenting the proof of Proposition \ref{prop-boot}, we first introduce the following lemma, which provides estimates for the nonlinear terms.
\begin{lemma}
  Under the bootstrap assumptions \eqref{eq-assu-1}-\eqref{eq-assu-11}, there is a constant $C_2$ independent of $C_1$ and $\varepsilon_0, \nu$ so that for any $0 \le t\le T$, it holds that
  \begin{align*}
  \sum_{i=1}^3\left\|e^{c_0\nu^{\frac{1}{3}}s} \mathcal{N}^i(s)\right\|_{L_{s}^1\left([0, t], H^{log}_xL_{y}^2\right)}   \leq C_2 C_1^2\varepsilon_0  \nu^{\beta-\frac{1}{2}}\left\| \omega_{i n}\right\|_{H^{log}_xL_{y}^2}.
  \end{align*}
  We also have
  \begin{align*}
    &\| \big(u^{(1)}_{\neq}\pa_xu^{(1)}_{\neq}\big)_0(s)\|_{L^1_s([0, t];L^2_{y})}+\| \big(u^{(2)}_{\neq}\pa_yu^{(1)}_{\neq}\big)_0(s)\|_{L^1_s([0, t];L^2_{y})}\\
    &+\| \big(u^{(1)}_{\neq}\pa_x\omega_{\neq}\big)_0(s)\|_{L^1_s([0, t];L^2_{y})}+\| \big(u^{(2)}_{\neq}\pa_y\omega_{\neq}\big)_0(s)\|_{L^1_s([0, t];L^2_{y})}\\
    \le&C_2 C_1^2\varepsilon_0  \nu^{\beta-\frac{1}{2}}\left\| \omega_{i n}\right\|_{H^{log}_xL_{y}^2}.
  \end{align*}
\end{lemma}
\begin{proof}
According to Bony's decomposition in $x$, we write
\begin{align*}
  u_{\neq}^{(1)} \partial_x \omega_{\neq}=\text{T}_{\partial_x \omega_{\neq}}u_{\neq}^{(1)}+\text{T}_{u_{\neq}^{(1)}}^*\partial_x \omega_{\neq},\quad u_{\neq}^{(2)} \partial_y \omega_{\neq}=\text{T}_{\partial_y \omega_{\neq}}u_{\neq}^{(2)}+\text{T}_{u_{\neq}^{(2)}}^*\partial_y \omega_{\neq}.
\end{align*}

Then by using  \eqref{eq-assu-5}-\eqref{eq-assu-11} and \eqref{eq-Ber-1}-\eqref{eq-Ber-3} we have
\begin{align*}
  &\|e^{c_0\nu^{\frac{1}{3}}s} \mathcal N^{(1)}\|_{L_{s}^1\left([0, t], H^{log}_xL_{y}^2\right)}\\
  \le&\|e^{c_0\nu^{\frac{1}{3}}s} \text{T}_{\partial_x \omega_{\neq}}u_{\neq}^{(1)}\|_{L_{s}^1\left([0, t], H^{log}_xL_{y}^2\right)}+\|e^{c_0\nu^{\frac{1}{3}}s} \text{T}_{u_{\neq}^{(1)}}^*\partial_x \omega_{\neq}\|_{L_{s}^1\left([0, t], H^{log}_xL_{y}^2\right)}\\
  &+\|e^{c_0\nu^{\frac{1}{3}}s} \text{T}_{\partial_y \omega_{\neq}}u_{\neq}^{(2)}\|_{L_{s}^1\left([0, t], H^{log}_xL_{y}^2\right)}+\|e^{c_0\nu^{\frac{1}{3}}s} \text{T}_{u_{\neq}^{(2)}}^*\partial_y \omega_{\neq}\|_{L_{s}^1\left([0, t], H^{log}_xL_{y}^2\right)}\\
  \le&C\Big(\|e^{c_0\nu^{\frac{1}{3}}s} \pa_x u_{\neq}^{(1)}\|_{L_{s}^2\left([0, t], H^{log}_xL_{y}^2\right)}\|e^{c_0\nu^{\frac{1}{3}}s} \omega_{\neq}\|_{L_{s}^2\left([0, t], L_{x, y}^\infty\right)}\\
  &+\|e^{c_0\nu^{\frac{1}{3}}s} u_{\neq}^{(1)}\|_{L_{s}^\infty\left([0, t], L_{x, y}^\infty\right)}\|e^{c_0\nu^{\frac{1}{3}}s} \pa_x \omega_{\neq}\|_{L_{s}^1\left([0, t], H^{log}_xL_{y}^2\right)}\\
  &+\|e^{c_0\nu^{\frac{1}{3}}s}|D_x|^{\frac{1}{2}} u_{\neq}^{(2)}\|_{L_{s}^2\left([0, t], H^{log}_xL^{\infty}_y\right)}\|e^{c_0\nu^{\frac{1}{3}}s} \pa_y\omega_{\neq}\|_{L_{s}^2\left([0, t], L_{x, y}^2\right)}\\
  &+\|e^{c_0\nu^{\frac{1}{3}}s}u_{\neq}^{(2)}\|_{L_{s}^2\left([0, t], L^{\infty}_{x,y}\right)}\|e^{c_0\nu^{\frac{1}{3}}s} \pa_y\omega_{\neq}\|_{L_{s}^2\left([0, t], H^{log}_xL_{y}^2\right)}\Big)\\
  \le&CC_1^2\varepsilon_0  \nu^{\beta-\frac{1}{2}}\left\| \omega_{i n}\right\|_{H^{log}_xL_{y}^2}.
\end{align*}
Similarly, by using \eqref{eq-assu-2}, \eqref{eq-assu-2}, \eqref{eq-assu-6}, and \eqref{eq-assu-10}, we have 
  \begin{align*}
  &\|e^{c_0\nu^{\frac{1}{3}}s} \mathcal N^{(2)}\|_{L_{s}^1\left([0, t], H^{log}_xL_{y}^2\right)}+\|e^{c_0\nu^{\frac{1}{3}}s} \mathcal N^{(3)}\|_{L_{s}^1\left([0, t], H^{log}_xL_{y}^2\right)}\\
  \le&\|e^{c_0\nu^{\frac{1}{3}}s} u^{(1)}_{0}\pa_x\omega_{\neq}\|_{L_{s}^1\left([0, t], H^{log}_xL_{y}^2\right)}+\|e^{c_0\nu^{\frac{1}{3}}s} u^{(2)}_{\neq}\pa_y\omega_0\|_{L_{s}^1\left([0, t], H^{log}_xL_{y}^2\right)}\\
  \le&\|u^{(1)}_{0}\|_{L_{s}^\infty\left([0, t], L_{y}^\infty\right)}\|e^{c_0\nu^{\frac{1}{3}}s} \pa_x\omega_{\neq}\|_{L_{s}^1\left([0, t], H^{log}_xL_{y}^2\right)}+\|e^{c_0\nu^{\frac{1}{3}}s} u^{(2)}_{\neq}\|_{L_{s}^2\left([0, t], H^{log}_xL^\infty_y\right)}\|\pa_y\omega_0\|_{L_{s}^2\left([0, t], L_{y}^2\right)}\\
  \le&CC_1^2\varepsilon_0  \nu^{\beta-\frac{1}{2}}\left\| \omega_{i n}\right\|_{H^{log}_xL_{y}^2}.
\end{align*}

A similar argument shows that
\begin{align*}
    &\| \big(u^{(1)}_{\neq}\pa_xu^{(1)}_{\neq}\big)_0\|_{L^1_s([0, t];L^2_{y})}+\| \big(u^{(2)}_{\neq}\pa_yu^{(1)}_{\neq}\big)_0\|_{L^1_s([0, t];L^2_{y})}\\
    &+\| \big(u^{(1)}_{\neq}\pa_x\omega_{\neq}\big)_0\|_{L^1_s([0, t];L^2_{y})}+\| \big(u^{(2)}_{\neq}\pa_y\omega_{\neq}\big)_0\|_{L^1_s([0, t];L^2_{y})}\\
    \le&\|  u^{(1)}_{\neq}\|_{L^\infty_s([0, t];L^\infty_{x,y})}\| \pa_xu^{(1)}_{\neq} \|_{L^1_s([0, t];L^2_{x,y})}+\|  u^{(2)}_{\neq}\|_{L^2_s([0, t];L^\infty_{x,y})}\| \pa_yu^{(1)}_{\neq} \|_{L^2_s([0, t];L^2_{x,y})}\\
    &+\|  u^{(1)}_{\neq}\|_{L^\infty_s([0, t];L^\infty_{x,y})}\| \pa_x\omega_{\neq} \|_{L^1_s([0, t];L^2_{x,y})}+\|  u^{(2)}_{\neq}\|_{L^2_s([0, t];L^\infty_{x,y})}\| \pa_y\omega_{\neq} \|_{L^2_s([0, t];L^2_{x,y})}\\
    \le&CC_1^2\varepsilon_0  \nu^{\beta-\frac{1}{2}}\left\| \omega_{i n}\right\|_{H^{log}_xL_{y}^2}.
\end{align*}
Here we used the fact that
\begin{align*}
  &\| \pa_xu^{(1)}_{\neq} \|_{L^1_s([0, t];L^2_{x,y})}+\| \pa_yu^{(1)}_{\neq} \|_{L^2_s([0, t];L^2_{x,y})}\le\| \pa_x\omega_{\neq} \|_{L^1_s([0, t];L^2_{x,y})}+\| \pa_y\omega_{\neq} \|_{L^2_s([0, t];L^2_{x,y})}.
\end{align*}
\end{proof}

Now we are in a position to prove Proposition \ref{prop-boot}.
\begin{proof}[Proof of Proposition \ref{prop-boot}]
  By using the properties of heat kernel, we get for $0\le t\le T$ that
  \begin{align*}
    \left\|u_0^{(1)}(t)\right\|_{L_y^2}+\left\|\omega_0(t)\right\|_{L_y^2}
    \le&\left\|e^{t\nu\pa_y^2} u_{in,0}^{(1)} \right\|_{L_y^2}+\left\|e^{t\nu\pa_y^2}\omega_{in,0}\right\|_{L_y^2}\\
    &+\left\| \int^t_0 e^{(t-s)\nu\pa_y^2}\Big(\big(u^{(1)}_{\neq}\pa_xu^{(1)}_{\neq}\big)_0+\big(u^{(2)}_{\neq}\pa_yu^{(1)}_{\neq}\big)_0\Big)(s,y) ds\right\|_{L^2_y} \\
  &+\left\| \int^t_0 e^{(t-s)\nu\pa_y^2}\Big(\big(u^{(1)}_{\neq}\pa_x\omega_{\neq}\big)_0+\big(u^{(2)}_{\neq}\pa_y\omega_{\neq}\big)_0\Big)(s,y) ds\right\|_{L^2_y} \\
  \le&\left\|e^{t\nu\pa_y^2} u_{in,0}^{(1)} \right\|_{L_y^2}+\left\|e^{t\nu\pa_y^2}\omega_{in,0}\right\|_{L_y^2}\\
  &+\| \big(u^{(1)}_{\neq}\pa_xu^{(1)}_{\neq}\big)_0\|_{L^1_s([0, t];L^2_{y})}+\| \big(u^{(2)}_{\neq}\pa_yu^{(1)}_{\neq}\big)_0\|_{L^1_s([0, t];L^2_{y})}\\
    &+\| \big(u^{(1)}_{\neq}\pa_x\omega_{\neq}\big)_0\|_{L^1_s([0, t];L^2_{y})}+\| \big(u^{(2)}_{\neq}\pa_y\omega_{\neq}\big)_0\|_{L^1_s([0, t];L^2_{y})}\\
  \le&C(1+C_2 C_1^2\varepsilon_0  \nu^{\beta-\frac{1}{2}})\varepsilon_0 \nu^{\beta}.
  \end{align*}
  It follows that
  \begin{align*}
   \left\|u_0^{(1)}(t)\right\|_{L_y^\infty}\le C \left\|u_0^{(1)}(t)\right\|_{L_y^2}^{\frac{1}{2}}\left\|\omega_0(t)\right\|_{L_y^2}^{\frac{1}{2}}\le C(1+C_2 C_1^2\varepsilon_0  \nu^{\beta-\frac{1}{2}})\varepsilon_0 \nu^{\beta}.
  \end{align*}
Recalling the property of heat kernel that
\begin{align*}
  \|\pa_y e^{t\nu\pa_y^2}f\|_{L^2_t([0, T];L^2_y)}\le C\nu^{-\frac{1}{2}}\|f\|_{L^2_y},
\end{align*}
we have
\begin{align*}
  &\|\pa_y\omega_0\|_{ L^2_t\left([0,T]; L^2_y\right)}\\
  \le& \left\|\pa_ye^{t\nu\pa_y^2}\omega_{in,0}\right\|_{ L^2_t\left([0,T]; L^2_y\right)}\\
  &+\left\|\int^t_0 \pa_ye^{(t-s)\nu\pa_y^2}\Big(\big(u^{(1)}_{\neq}\pa_x\omega_{\neq}\big)_0+\big(u^{(2)}_{\neq}\pa_y\omega_{\neq}\big)_0\Big)(s,y) ds\right\|_{ L^2_t\left([0,T]; L^2_y\right)}\\
  \le& \left\|\pa_ye^{t\nu\pa_y^2}\omega_{in,0}\right\|_{ L^2_t\left([0,T]; L^2_y\right)}\\
  &+\int^{T}_0\left(\int^{T}_{s} \left\|\pa_ye^{(t-s)\nu\pa_y^2}\Big(\big(u^{(1)}_{\neq}\pa_x\omega_{\neq}\big)_0+\big(u^{(2)}_{\neq}\pa_y\omega_{\neq}\big)_0\Big)(s)\right\|_{L^2_y}^2dt\right)^{\frac{1}{2}}ds\\
  \le&C\nu^{-\frac{1}{2}}\left\| \omega_{in,0}\right\|_{L^2_y}+
  C\nu^{-\frac{1}{2}}\int^{T}_0\left\|\Big(\big(u^{(1)}_{\neq}\pa_x\omega_{\neq}\big)_0+\big(u^{(2)}_{\neq}\pa_y\omega_{\neq}\big)_0\Big)(s)\right\|_{L^2_y}ds\\
  \le&C\nu^{-\frac{1}{2}}(1+C_2 C_1^2\varepsilon_0  \nu^{\beta-\frac{1}{2}})\left\| \omega_{i n}\right\|_{H^{log}_xL_{y}^2}.
\end{align*}

From \eqref{eq:DH-1} and Proposition \ref{prop-semi}, we have
\begin{align*}
  \left\|e^{c_0\nu^{\frac{1}{3}}t} {\om}_{\neq}(t,x,y)\right\|_{H^{log}_xL^2_{y}}
  &\le\left\|e^{c_0\nu^{\frac{1}{3}}t} S(t,0){\om}_{in,\neq}(x,y)\right\|_{H^{log}_xL^2_{y}}\\
  &\quad+\left\|\int^t_0 e^{c_0\nu^{\frac{1}{3}}(t-s)} S(t,s) \left(e^{c_0\nu^{\frac{1}{3}}s}(\mathcal N^{1}+\mathcal N^{2}+\mathcal N^{3})(s)\right)ds \right\|_{H^{log}_xL^2_{y}}\\
  \le&C_0\left\|  {\om}_{in,\neq}(x,y)\right\|_{H^{log}_xL^2_{y}}+C_0\left\| e^{c_0\nu^{\frac{1}{3}}s} (\mathcal N^{1}+\mathcal N^{2}+\mathcal N^{3})(s) \right\|_{L^1_s([0, t];H^{log}_xL^2_{y})}\\
  \le&\left(C_0+C_0C_2 C_1^2\varepsilon_0  \nu^{\beta-\frac{1}{2}}\right)\left\| \omega_{i n}\right\|_{H^{log}_xL_{y}^2}.
\end{align*}
It follows that
\begin{align*}
  \|e^{c_0\nu^{\frac{1}{3}}s}u^{(1)}_{\neq}(s)\|_{L^{\infty}_{x,y}}\le& C\left\|e^{c_0\nu^{\frac{1}{3}}t} {\om}_{\neq}(t,x,y)\right\|_{H^{log}_xL^2_{y}}\le C\left(C_0+C_0C_2 C_1^2\varepsilon_0  \nu^{\beta-\frac{1}{2}}\right)\left\| \omega_{i n}\right\|_{H^{log}_xL_{y}^2}.
\end{align*}

By using similar arguments, we deduce that
\begin{align*}
  &\nu^{\frac{1}{2}}\left(\int_{0}^T\|e^{c_0\nu^{\frac{1}{3}}t} \nabla{\om}_{\neq}(t,x,y)\|_{H^{log}_xL^2_{y}}^2dt\right)^{\frac{1}{2}}+\nu^{\frac{1}{2}}\int_{0}^T\|e^{c_0\nu^{\frac{1}{3}}t} \pa_x{\om}_{\neq}(t,x,y)\|_{H^{log}_xL^2_{y}}dt\\
  &+\nu^{\frac{1}{2}}\left(\int_{0}^T\|e^{c_0\nu^{\frac{1}{3}}t} {\om}_{\neq}(t,x,y)\|_{L^\infty_{x,y}}^2dt\right)^{\frac{1}{2}}+\left(\int_{0}^T\|e^{c_0\nu^{\frac{1}{3}}t}u^{(2)}_{\neq}(t,x,y)\|_{L^\infty_{x,y}}^2dt\right)^{\frac{1}{2}}\\
  &+ \left(\int_{0}^T\| e^{c_0\nu^{\frac{1}{3}}t}\left|D_x\right|^{\frac{1}{2}} u^{(2)}_{\neq}(t,x,y)\|_{H^{log}_xL^\infty_{y}}^2dt\right)^{\frac{1}{2}}+\left(\int_{0}^T\|e^{c_0\nu^{\frac{1}{3}}t} \pa_xu^{(1)}_{\neq}(t,x,y)\|_{H^{log}_xL^2_{y}}^2dt\right)^{\frac{1}{2}}\\
  \le&C\left(C_0+C_0C_2 C_1^2\varepsilon_0  \nu^{\beta-\frac{1}{2}}\right)\left\| \omega_{i n}\right\|_{H^{log}_xL_{y}^2}.
\end{align*}

Then by taking $C_1\ge CC_0$ big enough and then taking $\varepsilon_0$ small enough, we complete the proof of Proposition \ref{prop-boot}.
\end{proof}
\begin{appendix}
  \section{Bony's decomposition}
For the completeness of the article, in this appendix, we review some basic conclusions of Bony's decomposition on $\mathbb T$, the same content can be found in the appendix of \cite{MasmoudiZhao2020cpde}. For a more thorough and detailed presentation of this theory, we refer the readers to \cite{danchin2005fourier}.  

Let $(\varphi,\chi)$ be a couple of smooth functions defined on $\mathbb R$ and valued in in $[0,1]$, such that $\mathrm{supp}\, \varphi\subset \left\{\xi:~\frac{3}{4}|\xi|\le \frac{8}{3}\right\}$, $\mathrm{supp}\,\chi\subset \left\{\xi:~|\xi|\le \frac{4}{3}\right\}$, and 
\begin{align*}
  \chi(\xi)+\sum_{j\in\mathbb N}\varphi(2^{-j}\xi)=1,\ \forall\xi\in\mathbb R,
\end{align*}
where $\mathbb N=\left\{0,1,2,\dots\right\}$.

Denoting
\begin{align*}
  h_j(x)= \frac{1}{2\pi}\sum_{k\in\mathbb Z}\varphi(2^{-j}\beta)e^{ik x},
\end{align*}
we define the periodic dyadic blocks as
\begin{align*}
  \Delta_{-1}^{per} u(x)&=\frac{1}{2\pi}\tilde u_0,\\
  \Delta_j^{per} u(x)&=\frac{1}{2\pi}\sum_{k\in\mathbb Z}\varphi(2^{-j}\beta)\tilde u_ke^{ik x}=\int_{\mathbb T} h_j(x-y)u(y) dy,\text{ for }j\in\mathbb N, 
\end{align*}
and the lower-frequency cut-off
\begin{align*}
  S_j^{per} u(x)=\sum_{-1\le i\le j-1} \Delta_j^{per} u(x)=\frac{1}{2\pi}\sum_{k\in\mathbb Z}\chi(2^{-j}\beta)\tilde u_ke^{ik x},\text{ for }j\in\mathbb N.
\end{align*}
Then, we have Bony's decomposition: $T_fg=\sum_{j\ge 1}S_{j-1}f\Delta_j^{per}g$ and $T^*_gf=fg-T_fg$.

The following Bernstein type inequalities will be used.
\begin{align}
  & \left\|\ln \left(e+\left|D_x\right|\right) T_f g\right\|_{L_x^2}+\left\|\ln \left(e+\left|D_x\right|\right) T_f^* g\right\|_{L_x^2} \leq C \left\| f\right\|_{L_x^{\infty}} \left\|\ln \left(e+\left|D_x\right|\right) g \right\|_{L_x^2},\label{eq-Ber-1} \\
& \left\|\ln \left(e+\left|D_x\right|\right) T_f g\right\|_{L_x^2} \leq C \left\|f\right\|_{L^2_x} \left\|\left|D_x\right|^{\frac{1}{2}} \ln \left(e+\left|D_x\right|\right) g\right\|_{L^2_x},\label{eq-Ber-2} \\
& \left\| \ln \left(e+\left|D_x\right|\right) T_{\partial_x f} g\right\|_{L_x^2}  \leq C \left\| f\right\|_{L_x^{\infty}}\left\|\ln \left(e+\left|D_x\right|\right) \pa_xg \right\|_{L_x^2}\label{eq-Ber-3}.
\end{align}
Here we give the proof of \eqref{eq-Ber-2}, 
\begin{align*}
  &\left\|\ln \left(e+\left|D_x\right|\right) T_f g\right\|_{L_x^2}=\left\|\ln \left(e+\left|D_x\right|\right) \sum_{j\ge 1}S_{j-1}f\Delta_j^{per}g\right\|_{L_x^2}\\
  \lesssim&\left(\sum_{k\ge-1}\left\|\langle k\rangle \Delta_k^{per} \left(\sum_{j\ge1} S_{j-1}f \Delta_j^{per}g\right)\right\|_{L^2_x}^2\right)^{\frac{1}{2}}\\
  \lesssim& \left(\sum_{k\ge-1}\langle k\rangle^2 \sum_{|j-k|\le2} \left\|S_{j-1}f\right\|_{L^\infty_x}^2\left\|\Delta_j^{per}g\right\|_{L^2_x}^2\right)^{\frac{1}{2}}\\
  \lesssim& \left(\sum_{k\ge-1}\langle k\rangle^2 \sum_{|j-k|\le2}2^j \left\|S_{j-1}f\right\|_{L^2_x}^2\left\|\Delta_j^{per}g\right\|_{L^2_x}^2\right)^{\frac{1}{2}}\\
  \lesssim& \left\|f\right\|_{L^2_x}\left(\sum_{k\ge-1}\langle k\rangle^22^k\left\|\Delta_k^{per}g\right\|_{L^2_x}^2\right)^{\frac{1}{2}}\le C \left\|f\right\|_{L^2_x} \left\|\left|D_x\right|^{\frac{1}{2}} \ln \left(e+\left|D_x\right|\right) g\right\|_{L^2_x}.
\end{align*}
The estimates \eqref{eq-Ber-1} and \eqref{eq-Ber-3} can be obtained by the same argument.
\section{Properties of the background shear flow}
The wave operator is defined on background shear flow $b_s(y)$. Here we give some properties of $b_s(y)$.
\begin{lemma}\label{lem-back}
  Under Assumption \ref{assum}, it holds that 
  \begin{align}
    0<c_m\leq b_s'(y)\leq c_m^{-1}, \ \forall s\ge0,\label{eq-est-b1}\\
    \lim_{y\to \pm\infty}b'_{in}(y)=C_{\pm},\ \forall s\ge0,\label{eq-est-b2}\\
    \left\|b_s''(y)\right\|_{W^{1,1}}+\left\|b_s''(y)\right\|_{H^3}\le C \left(\left\|b_{in}''(y)\right\|_{W^{1,1}}+\left\|b_{in}''(y)\right\|_{H^3}\right),\ \forall s\ge0,\label{eq-est-b2.5}\\
\left\|\sqrt{1+|y|}b_s''(y)\right\|_{L^2}\le C \left( \left\|\sqrt{1+|y|}b_{in}''(y)\right\|_{L^2}+ \left\|b_{in}''(y)\right\|_{L^1}\right)  ,\ \forall s\ge0,\label{eq-est-b3}\\
\left\|\sqrt{1+|y|}b_s'''(y)\right\|_{L^2}\le C \left(\left\|\sqrt{1+|y|}b_{in}'''(y)\right\|_{L^2}+ \left\|b_{in}'''(y)\right\|_{L^1}\right)\label{eq-est-b4},\ \forall s\ge0.
  \end{align}
Here $C$ is a positive constant that depends only on $c_m$.
\end{lemma}
\begin{proof}
  We only consider the solution that $b_s''(y)$ in energy spaces. Then with the help of the fundamental solution, we have
  \begin{align*}
     b_s'(y)=\frac{1}{\sqrt{4\pi \nu s}}\int_{\mathbb R} e^{-\frac{(y_1-y)^2}{4\nu s}}b_{in}'(y_1) d y_1.
   \end{align*} 
The estimates \eqref{eq-est-b1} follow immediately.

We write
\begin{align*}
  b_s'(y)-C_-=\frac{1}{\sqrt{4\pi \nu s}}\int_{\mathbb R} e^{-\frac{(y_1-y)^2}{4\nu s}} \left(b_{in}'(y_1)-C_-\right) d y_1.
\end{align*}
It follows from $\lim\limits_{y\to -\infty}b'_{in}(y)=C_{-}$ that
$
  \lim\limits_{y\to -\infty}b_s'(y)=C_{-},
$
and similarly
$\lim\limits_{y\to +\infty}b_s'(y)=C_{+}.$

It holds that
\begin{align*}
  b_s''(y)=\frac{1}{\sqrt{4\pi \nu s}}\int_{\mathbb R} e^{-\frac{(y_1-y)^2}{4\nu s}}b_{in}''(y_1) d y_1,\ b_s'''(y)=\frac{1}{\sqrt{4\pi \nu s}}\int_{\mathbb R} e^{-\frac{(y_1-y)^2}{4\nu s}}b_{in}'''(y_1) d y_1,
\end{align*}
therefore
\begin{align*}
  \int_{\mathbb R}b_s''(y) d y=&\int_{\mathbb R}\frac{1}{\sqrt{4\pi \nu s}}\int_{\mathbb R} e^{-\frac{(y_1-y)^2}{4\nu s}}b_{in}''(y_1) d y_1 d y\\
  =&\int_{\mathbb R}b_{in}''(y_1)\int_{\mathbb R} \frac{1}{\sqrt{4\pi \nu s}}e^{-\frac{(y_1-y)^2}{4\nu s}} d y d y_1=\int_{\mathbb R}b_{in}''(y_1) d y_1,
\end{align*}
and $\left\|b_s''(y)\right\|_{L^1_y}\le\left\|b''_{in}(y)\right\|_{L^1_y}$. We also have $\left\|b_s'''(y)\right\|_{L^1_y}\le\left\|b'''_{in}(y)\right\|_{L^1_y}$ and $\left\|b_s''(y)\right\|_{H^3_y}\le\left\|b''_{in}(y)\right\|_{H^3_y}$. 

Next, we give the decay estimates. We write 
\begin{align*}
  \sqrt{1+|y|}b_s''(y)=&\frac{1}{\sqrt{4\pi \nu s}}\int_{\mathbb{R}}e^{-\frac{(y-y_1)^2}{4\nu s}}\sqrt{1+|y_1|}b''_{in}(y_1)dy_1\\
  &+\frac{1}{\sqrt{4\pi \nu s}}\int_{\mathbb{R}}e^{-\frac{(y-y_1)^2}{4\nu s}}\frac{|y|-|y_1|}{\sqrt{1+|y_1|}\sqrt{1+|y|}}b''_{in}(y_1)dy_1\\
  \eqdef&I+II.
\end{align*}
It is clear that
\begin{align*}
  \left\|I\right\|_{L^2}\le\left\|\sqrt{1+|y|}b_{in}''(y)\right\|_{L^2}.
\end{align*}
As $\left|\frac{|y|-|y_1|}{\sqrt{1+|y_1|}\sqrt{1+|y|}}\right|\le |y-y_1|^{\frac{1}{2}}$, we have
\begin{align*}
  \left\|II\right\|_{L^2}\le \left\|b_{in}''(y)\right\|_{L^1}\frac{1}{\sqrt{4\pi \nu s}}\left\|e^{-\frac{y^2}{4\nu s}}|y|^{\frac{1}{2}}\right\|_{L^2}\lesssim\left\|b_{in}''(y)\right\|_{L^1}.
\end{align*}
As a result, we have
\begin{align*}
  \left\|\sqrt{1+|y|}b_s''(y)\right\|_{L^2}\lesssim \left\|\sqrt{1+|y|}b_{in}''(y)\right\|_{L^2}+ \left\|b_{in}''(y)\right\|_{L^1},
\end{align*}
which gives \eqref{eq-est-b3}. By the same argument, we also have \eqref{eq-est-b4}.
\end{proof}
  \section{Spectrum of the Rayleigh operator}
\begin{remark}\label{rmk-spec}
Assume that $b_s'\ge c_m$ and $b_s''\in H^2(\mathbb R)$.  The continuous spectrum of $\mathcal{L}_{s,k}=\Delta_k^{-1}(b_s(y) \Delta_k-b''_s(y)\mathrm{Id})$ cannot be distributed on $\mathbb{C}\setminus\mathbb{R}$.
\end{remark}
\begin{proof}
  Assume that $\fc\in\sigma(\mathcal{L}_{s,k})$ with $\fc_i\neq0$, then there exists sequence $\left\{\Phi_n(k,y),\psi_n(k,y)\right\}$ such that $\left\|\Phi_n(k,y)\right\|_{H^2}=1$, $\left\|\psi_n(k,y)\right\|_{H^2}\to0$ as $n\to+\infty$, and
  \begin{align}\label{eq-eigen-Phi}
    (\fc-\mathcal{L}_{s,k})\Phi_n(k,y)=\psi_n(k,y).
  \end{align}
  Let $\Om_n(k,y)=\Delta_k\Phi_n(k,y)$, $\om_n(k,y)=\Delta_k\phi_n(k,y)$. It is clear that $\left\|\Om_n(k,y)\right\|_{L^2}\approx \left\|\Phi_n(k,y)\right\|_{H^2}$, $\left\|\om_n(k,y)\right\|_{L^2}\approx \left\|\phi_n(k,y)\right\|_{H^2}$. Equation \eqref{eq-eigen-Phi} can be rewritten as
  \begin{align}\label{eq-eigen-Om}
    \Om_n(k,y)-\frac{b_s''(y)}{b_s(y)-\fc}\Delta_k^{-1}\Om_n(k,y)=\frac{\om_n(k,y)}{b_s(y)-\fc}.
  \end{align}
Since $\fc_i\neq0$ and $\left\|\om_n(k,y)\right\|_{L^2}\to0$ as $n\to+\infty$, it follows that
\begin{align}\label{eq-lim-comp}
  \left\|\Om_n(k,y)-\frac{b_s''(y)}{b_s(y)-\fc}\Delta_k^{-1}\Om_n(k,y)\right\|_{L^2}\to0 \text{ as }n\to+\infty.
\end{align}
Recalling that $\left\|\Om_n(k,y)\right\|_{L^2}\approx1$ and $b_s'\ge c_m$, we have
\begin{align*}
  \left\|\frac{b_s''(y)}{b_s(y)-\fc}\Delta_k^{-1}\Om_n(k,y)\right\|_{H^1}\le C,\quad \left\|y\frac{b_s''(y)}{b_s(y)-\fc}\Delta_k^{-1}\Om_n(k,y)\right\|_{L^2}\le C.
\end{align*}
Then by a compactness argument, there exits $\Om(k,y)$ such that $\left\{\frac{b_s''(y)}{b_s(y)-\fc}\Delta_k^{-1}\Om_n(k,y)\right\}$ strongly convergence to $\Om(k,y)$ in $L^2$ norm (by passing to a subsequence). It follows from \eqref{eq-lim-comp} that
\begin{align*}
  \left\|\Om_n(k,y)-\Om(k,y)\right\|_{L^2}\to0 \text{ as }n\to+\infty,
\end{align*}
and $\fc$ is an eigenvalue of the Rayleigh operator $\mathcal{R}_{s,k}$. Accordingly, $\fc$ is an eigenvalue of $\mathcal{L}_{s,k}$ with eigenfunction $\Delta_k^{-1}\Om(k,y)$.

In conclusion, if $\fc\in\sigma(\mathcal{L}_{s,k})$ with $\fc_i\neq0$, then $\fc$ can only be the eigenvalue of $\mathcal{L}_{s,k}$.
\end{proof}

\end{appendix}

\bibliographystyle{siam.bst} 
\bibliography{references.bib}

\end{document}